\documentclass[article]{article}
\usepackage[dvipsnames, svgnames, x11names]{xcolor}  % ???��?�Vae?????��? ???
\usepackage{graphicx}
\usepackage{amsmath}
\usepackage{amssymb}
\usepackage{mathrsfs}
\usepackage[numbers,sort&compress]{natbib}
\usepackage{amsthm}
\usepackage{graphicx}
\usepackage{tikz}
\usepackage{stmaryrd}
\usepackage{subfigure}
\usepackage{hyperref}
\hypersetup{hypertex=true,
            colorlinks=true,
            linkcolor=blue,
            anchorcolor=blue,
            citecolor=red}
\newtheorem{theorem}{Theorem}[section]
\newtheorem{lemma}{Lemma}[section]
\newtheorem{proposition}{Proposition}[section]
\DeclareMathOperator*{\res}{Res}

\DeclareMathOperator*{\im}{Im}
\DeclareMathOperator*{\re}{Re}
\numberwithin{equation}{section}

{\LARGE }

\newtheorem{remark}{Remark}[section]

\begin{document}

	\baselineskip=17pt
	
	\title{  The  defocusing  NLS equation with  nonzero background: Large-time  asymptotics   in the solitonless  region }
	\author{ Zhaoyu Wang and   Engui Fan\thanks{Corresponding author: faneg@fudan.edu.cn}\\[4pt]
	\small	School of Mathematical Science and Key Laboratory   for Nonlinear Science,\\ \small Fudan University,
		Shanghai 200433, P.R. China}
	\maketitle

\begin{abstract}
	\baselineskip=17pt
	 We consider the Cauchy problem for the defocusing   Schr$\ddot{\text{o}}$dinger (NLS) equation with   a nonzero background
	\begin{align}
&iq_t+q_{xx}-2(|q|^2-1)q=0, \nonumber\\
&q(x,0)=q_0(x), \quad \lim_{x \to \pm \infty}q_0(x)=\pm 1.\nonumber\
	\end{align}
 Recently, for the space-time  region  $|x/(2t)|<1$ which is a solitonic  region  without   stationary phase  points on the jump contour,  Cuccagna and  Jenkins presented
  the asymptotic stability of the $N$-soliton solutions for the   NLS equation by using the $\bar{\partial}$ generalization of the Deift-Zhou nonlinear steepest descent method.
 Their large-time asymptotic expansion takes     the form
\begin{align}
q(x,t)=   T(\infty)^{-2} q^{sol,N}(x,t) + \mathcal{O}(t^{-1 }),\label{res1}
\end{align}
whose leading term is N-soliton and the second term $\mathcal{O}(t^{-1})$ is a residual error  from a $\overline\partial$-equation.
In this paper, we are interested in the  large-time asymptotics  in
  the space-time  region  $ |x/(2t)|>1$  which is outside the soliton region, but   there will be  two stationary points appearing on
the jump contour $\mathbb{R}$.
We found   a asymptotic  expansion that is different from  (\ref{res1}) 
 \begin{align}
q(x,t)= e^{-i\alpha(\infty)} \left(1 +t^{-1/2} h(x,t) \right)+\mathcal{O}\left(t^{-3/4}\right),\label{res2}
\end{align}
whose leading term is a nonzero  background,  the second $t^{-1/2}$ order term  is from continuous spectrum
 and the third term $\mathcal{O}(t^{-3/4})$ is a residual error  from a $\overline\partial$-equation.
The above two  asymptotic  results (\ref{res1}) and (\ref{res2}) imply  that the region $ |x/(2t)|<1$ considered by  Cuccagna and  Jenkins is  a  fast decaying   soliton solution region,
 while the region $ |x/(2t)|>1$ considered by us is a slow decaying nonzero background region.
%    Additionally, the asymptotic stability property for   the N-soliton solutions of the defocusing  NLS equation  is obtained.
		\\[5pt]

\end{abstract}

\newpage

\tableofcontents%

	\section{Introduction}
		In this paper, we study the Cauchy problem for the defocusing nonlinear Schr$\ddot{\text{o}}$dinger (NLS) equation on $\mathbb{R} \times \mathbb{R}_+  $:
	\begin{equation}\label{q}
		iq_t+q_{xx}-2(|q|^2-1)q=0,
	\end{equation}
	\begin{equation}\label{inq}
		q(x,0)=q_0(x), \quad \lim_{x \to \pm \infty}q_0(x)=\pm 1.
	\end{equation}
The NLS equation is one of the basic models in   many branches of science, such as fibre optics, biology, physics, mathematics and social science \cite {A1,A2,A3,A4,A5,A7}.
 The Lax pair  of the NLS equation was first derived by Zakharov and Shabat   \cite{ZS1};  The  well-posedness of the NLS equation   with
the initial data in Sobolev spaces $L^2(\mathbb{R})$  and   $ H^s(\mathbb{R}), \ s>0 $  was proved  by  Tsutsumi  and Bourgain  respectively \cite{Tsutsumi,Bourgain}.
 The inverse scattering transform (IST)    for the focusing NLS equation  with zero boundary conditions   was first developed by Zakharov
and Shabat   \cite{ZS2}.  The next important step of the development of  IST method   is  that the Riemann-Hilbert (RH) method,  as the  modern version of IST, was   established  by Zakharov and Shabat   \cite{Zakharov1979}.
It has   become clear that the    RH  method  is applicable to  construction  of exact solutions   and  asymptotic  analysis of  solutions
for   a wide class of    integrable systems  \cite{MonvelCH,O1,F3,F1,F2,F4,Bio3,F5,ZF2019,RN12,wengyan}.

  We mention the following works  on the long-time asymptotics of the defocusing NLS equation.
    %by Zakharov-Manakov, Deift-Zhou, Fromm-Lenells-Quirchmayr, Dieng-Mclaughlin, Cuccagna-Jenkins and Vartanian.
    For the initial data in  the Schwarz space $\mathcal{S}(\mathbb{R})$ and   using the IST method,  Zakharov and Manakov
  obtained the long-time asymptotics  of the  NLS equation  \cite{ZS3}
    \begin{equation}
    	q(x,t)=t^{-1/2}h(x/t)e^{\frac{ix^2}{4t} \pm 2i |h(x/t)|^2  \log t}+\mathcal{O}(t^{-3/2}),\nonumber
    \end{equation}
   with an arbitrary function $h(x)$.  In 1981, Its presented a stationary phase method
   to analyze the long-time asymptotic behavior   for the NLS equation \cite{its1}.  In 1993,  Deift and Zhou  developed a  nonlinear steepest descent method to rigorously obtain the long-time asymptotic behavior   for the mKdV equation
 \cite{sdmRHp}.   Later this method was extended  to get the leading  and high-order   asymptotic behavior for the solution of the  NLS equation (\ref{q})
 with the initial data  $q_0(x) \in \mathcal{S}(\mathbb{R})$  \cite{PX2,PX21}
    \begin{equation}
	q(x,t)=t^{-1/2} \alpha(z_0) e^{\frac{ix^2}{2t} - i v(z_0) \log 4t}+ \mathcal{O}(t^{-1} \log t).\nonumber
\end{equation}
 %where
 %\begin{equation}
 %	v(z)=-\frac{1}{2\pi} \log(1-|r(z)|^2), \quad |\alpha(z)|^2=v(z)^2,\quad z_0=-\frac{x}{2t},
% \end{equation}
%and
%\begin{equation}
%	\arg \alpha(z)=\frac{1}{\pi} \int_{-\infty}^{z} \log(z-s)\, \mathrm{d} \left( \log \left(1- |r(s)|^2 \right)\right) + \frac{\pi}{4}+\arg \Gamma(iv(z))-\arg r(z).
%\end{equation}
Under  much weaker  weighted Sobolev initial data $q_0 \in H^{1,1} (\mathbb{R})$,  Deift and Zhou  found  the following result \cite{PX3}
    \begin{equation}
	q(x,t)=t^{-1/2} \alpha(z_0) e^{\frac{ix^2}{2t} - i v(z_0) \log 4t}+ \mathcal{O}(t^{-1/2-\kappa}),\ \  0<\kappa<1/4.\label{erdd}
\end{equation}
In the serial of articles \cite{VAH1,VAH2,VAH3},  Vartanian applied IST method to
 compute  the leading and first correlation terms in the asymptotic behavior of
 the NLS equation with  the finite     density initial data  $ q_0(x)-q_0(\pm \infty)  \in \mathcal{S}(\mathbb{R})$  as $x,t \to \pm \infty$ for
  both $|x/(2t)|>1$ (outside the soliton light cone)  and $|x/(2t)|<1$ (inside the soliton light cone).
  In 2008,  for $q_0\in H^{1,1}(\mathbb{R})$, Dieng and Mclaughlin  applied the  $\bar{\partial}$-steepest descent method to
 obtain a sharp estimate \cite{MK}
    \begin{equation}
	q(x,t)=t^{-1/2} \alpha(z_0) e^{\frac{ix^2}{2t} - i v(z_0) \log 4t}+ \mathcal{O}(t^{-3/4}).\nonumber
\end{equation}
 Jenkins  investigated long-time/zero-dispersion limit of
the solutions to   the defocusing NLS equation  associated to  step-like  initial data   \cite{Jenkins}
\begin{equation}
q(x,0) \sim \begin{cases}
		Ae^{-2i\mu x/\varepsilon},\quad x\to -\infty,\\
	1, \quad x\to \infty.	\nonumber
	\end{cases}	
\end{equation}
Fromm, Lenells and Quirchmayr  studied the long-time asymptotics  for the defocusing NLS equation with  the step-like boundary condition \cite{Le}
\begin{equation}
q(x,t) \sim \begin{cases}
		\alpha e^{2i\beta x +i \omega t},\quad x\to -\infty,\\
	0, \quad x\to \infty.	\nonumber
	\end{cases}	
\end{equation}

Recently,  for  the finite density initial data $q_0(x)-\tanh x \in H^{4,4} (\mathbb{R})$,
 Cuccagna and Jenkins  derived the leading order approximation to the solution of NLS  equation in the solitonic
  space-time region I:  $|\xi|<1$ with $  \xi=x/(2t)$  by using  $\bar{\partial}$ generalization of the Deift-Zhou nonlinear steepest descent method  \cite{CJ} 
\begin{align}
& q(x,t)=   T(\infty)^{-2} q^{sol,N}(x,t) + \mathcal{O}(t^{-1 }), \nonumber
\end{align}
whose leading term is N-soliton and the second term $\mathcal{O}(t^{-1})$ is a residual error  from a $\overline\partial$-equation.
 See the region I in  Figure \ref{spacetime}.
 In our  study,  for  the  solitonless  region II:   $1<  |\xi|, \ \xi=\mathcal{O}(1) $,
we further    obtain the large-time asymptotic behavior of   the NLS  equation (\ref{q})-(\ref{inq}) in the form
\begin{equation}
  q(x,t)= e^{-i\alpha(\infty)} \left( 1 +t^{-1/2} h(x,t) \right)+\mathcal{O}\left(t^{-3/4}\right). \nonumber
 \end{equation}
Compared to the results of Vartanian, see  Theorem 2.2.1-2.2.2 in \cite{VAH2},
 we give a more holistic description  of the solution as the parameters of a multi-soliton are modulated by soliton-soliton and soliton-radiation interactions.
We also weaken the request on the initial data from the Schwartz  space $\mathcal{S}(\mathbb{R})$ to a weighted Sobolev space  $H^{4,4}(\mathbb{R})$.
The $\bar{\partial}$-steepest descent method  was   first    developed    by McLaughlin and   Miller
   to analyze the asymptotics of orthogonal polynomials with non-analytical weights  \cite{McL1,McL2}.
In recent years,   this method has been successfully used to obtain  the long-time asymptotics and
 the soliton resolution conjecture for some integrable systems  \cite{fNLS,Liu3,LJQ,YF1,YYLmch}.

       \begin{figure}
	\begin{center}
		\begin{tikzpicture}
            \draw[yellow!20, fill=yellow!20] (0,0)--(4.5,0)--(4.5,2)--(0, 2);
            \draw[blue!20, fill=green!20] (0,0 )--(4.5,2)--(0,2)--(0,0);
            \draw[blue!20, fill=green!20] (0,0 )--(-4.5,0)--(-4.5,2)--(0, 2)--(0,0);
           \draw[yellow!20, fill=yellow!20] (0,0 )--(-4.5,0)--(-4.5,2)--(0,0);
		%\draw[CadetBlue!20, fill=LightSteelBlue!20] (0,0)--(2.5,0)--(2.5,2)--(0, 2);
		%\draw[CadetBlue!20, fill=LightSteelBlue!20] (-2.5,0)--(0,0)--(0,2)--(-2.5, 2);
		%\draw[GreenYellow!20, fill=YellowGreen!20] (-2.5,-2)--(0,-2)--(0,0)--(-2.5, 0);
		%\draw[GreenYellow!20, fill=YellowGreen!20] (0,-2)--(2.5,-2)--(2.5,0)--(0, 0);
		%\draw[CadetBlue!20, fill=LightSteelBlue!20] (-3.5,2.5)--(0,0)--(3.5,2.5);
		%\draw[GreenYellow!20, fill=LightSteelBlue!] (-3.5,2.5)--(-3.5,0)--(0,0);
		%\draw[GreenYellow!20, fill=LightSteelBlue!] (3.5,2.5)--(3.5,0)--(0,0);
		\draw [-> ](-5,0)--(5,0);
		\draw [-> ](0,0)--(0,2.8);
		\node    at (0.1,-0.3)  {$0$};
		\node    at (5.26,0)  { $x$};
		\node    at (0,3.2)  { $t$};
		%\draw [dashed](0,0)--(-4,2);
		%\draw [dashed](0,0)--(4,2);
		 \node  [below]  at (4.4,2.7) {\small$\xi=1$};
		 \node  [below]  at (-4.4,2.7) {\small$\xi=-1$};
		%\node at (-0.1,1.5) {\small$|x/t|<2$};
		\draw [red](0,0)--(-4.5,2);
		\draw [red](0,0)--(4.5,2);
		%\node  []  at (4,0.8) {\small$\xi=K$};
		%\node  []  at (-4,0.8) {\small$\xi=-K$};
	   \node  []  at (0.7,1.5) { $I$};
     \node  []  at (-0.7,1.5) { $I$};
     \node [] at (0,0.85) {Solitonic region};
		\node  []  at (-3.2,0.9) { $II$};
        \node [] at (-2.8,0.3) {Solitonless region};
	\node  []  at (3.2,0.9) { $II$};
    \node [] at (2.8,0.3) {Solitonless region};
		\end{tikzpicture}
	\end{center}
	\caption{ \small The space-time region of $x$ and $t$, where the blue region I: $|\xi|<1$  which is solitonic region and  was discussed by Cuccagna and Jenkins;
the yellow regions II: $|\xi|>1, \ \xi=\mathcal{O}(1)$ which is solitonless region and  we
will consider  in this paper.  }
	\label{spacetime}
\end{figure}
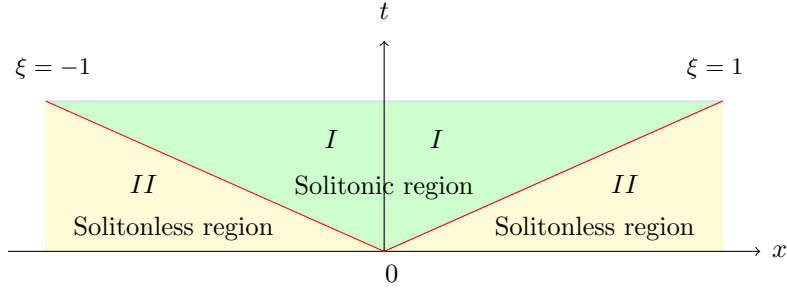

    %Moreover,  our results show  the soliton resolution has been supported in our cases.

    %\subsection{Statement of main results}
	%    Our main result provide a verification of the soliton resolution conjecture for (\ref{q}) with the initial data (\ref{inq}) which possess a certain number of derivatives and moments.

\begin{remark}\label{whyweneedM}        We consider the  region  $|\xi|>1, \ \xi=\mathcal{O}(1)$ due to the following reasons:
    \begin{itemize}
        \item  This region for $|x|>2t$ as $t \to \infty$  was completely ignored by  \cite{CJ}
        and the long-time asymptotic behavior of the defocusing NLS in this region is still unknown.

        \item  The above region also ensure  that the phase   point    $\xi_2$  defined by \eqref{xi1} and \eqref{xi2}  is   bounded,  such that
         subsequent  estimates  on  ${\rm Re} (i\theta(z))$,  the jump matrix  and  $\bar{\partial}$-derivatives are  reasonable,  see Proposition \ref{prop1}---\ref{prop3},
        Proposition \ref{Estv2}---\ref{sum5}, etc.

       % \item The higher power term of the expansion $\theta(z)$ near   phase  point could decay as $t\rightarrow\infty$, by which we can
       % set up the local model. See Proposition \ref{sum2}.

        \item  For  the case $|\xi|\rightarrow \infty$,  we can derive the large-$x$ behavior of the NLS equation using   the   $\bar{\partial}$-steepest descent method.
        \item   We will consider the transition space-time region $|x\pm2t|=O(1)$ in the next study.
    \end{itemize}
\end{remark}

	   %\subsection{Organization of the rest of paper}
	   The structure of this work is as follows.  In Section 2, we quickly recall
	  the spectral analysis of the Lax pair which was obtained in \cite{CJ},  but is useful for our work.
	   In Section 3,  we list some important  properties of the scattering data and the reflection coefficient, and formulate an RH problem for  $M(z)$ to characterize the Cauchy problem (\ref{q})-(\ref{inq}).
       In Section 4 and Section 5,    we make two deformations to the RH problem for   $M(z)$.
       One is to regularize the RH problem and the other is to obtain a mixed $\bar{\partial}$-problem
by  continuous extensions of  jump matrices.
     Then,  we solve  the mixed $\bar{\partial}$-problem by decomposing it
      into a pure RH problem and a pure
     $\bar\partial$-problem according to  the $\bar{\partial}$-derivative.
       We give the detailed analysis on the pure RH problem and the pure $\bar\partial$-problem  in Section 6 and Section 7 respectively.
       In Section 8,  we calculate the large-time asymptotics for the NLS equation.
     % In addition,  we give an asymptotic stability type result for the multi-solitons.

       \section{Spectral analysis on  the Lax pair}

   To state our results, we give the following definition of normed spaces:
	    \begin{align*}
	    	    	&L^{p,s}(\mathbb{R}):=\{ \mu(x)\in L^{p}(\mathbb{R}) | \langle x \rangle^s \mu(x) \in L^{p}(\mathbb{R})  \}\; \text{with}\; ||\mu||_{L^{p,s}} := ||\langle x \rangle^s \mu||_{L^{p}},\\
	    	    	&W^{k,p}(\mathbb{R}):=\{ \mu(x)\in L^{p}(\mathbb{R}) | \partial^j\mu(x) \in L^{p}(\mathbb{R}), j=1,\cdots,k \}\;\text{with}\; ||\mu||_{W^{k,p}} := \sum_{j=0}^{k}||\partial^{j} \mu||_{L^{p}},\\
	    	    %	&H^k(\mathbb{R}) :=\{ \mu(x)                  \}
                    &H^{k,s}(\mathbb{R}):=\{ \mu(x)\in L^{2}(\mathbb{R}) | \langle x \rangle^s \mu(x) \in L^{2}(\mathbb{R}),\langle y \rangle^k \hat{h}(y) \in L^{2}(\mathbb{R})\},
	    \end{align*}
	   where $ \langle x \rangle := \sqrt{1+|x|^2}$ and $\hat{h}$ is the Fourier transform defined by
	   \begin{equation}
	   	  \hat{h}(y) :=\frac{1}{2 \pi } \int_{\mathbb{R}} e^{-ixy}h(x)\, \mathrm{d}x.\nonumber
	   \end{equation}

Below we give a quick review of the direct problem for the initial value problem (\ref{q})-(\ref{inq}).  For details, see \cite{CJ}.
 The NLS equation (\ref{q}) admits a Lax pair
  \begin{align}
  	  	&\psi_x=L \psi, \quad L=L(z;x,t)=i \sigma_3 \left(Q-\lambda\left(z\right)\right),\label{laxx}\\
  	  	&\psi_t=T \psi,\quad T=T(z;x,t)=-2\lambda(z) L+i(Q^2-I)\sigma_3+Q_x, \label{laxt}
  \end{align}
	where
	\begin{equation*}
		Q=Q(x,t)=\left(\begin{array}{cc}
			0 & \bar{q}(x,t)\\
			q(x,t) & 0
		\end{array}\right),\quad \lambda(z) = \frac{z+z^{-1}}{2}, \quad \sigma_3= \left(\begin{array}{cc}
		1 & 0\\
		0 & -1
	\end{array}\right).
	\end{equation*}
    Replacing $L$ and $T$ by their limits as $x \to \pm \infty$, the Lax pair (\ref{laxx})-(\ref{laxt}) changes to
     \begin{align}
   	&\varphi_x^\pm=L_{\pm} \varphi^\pm, \quad L_\pm=L_\pm(z;x,t)=i \sigma_3 (Q_\pm-\lambda I),\label{alaxx}\\
   	&\varphi_t^\pm=T_\pm \varphi^\pm,\quad T_\pm=T_\pm(z;x,t)=-2\lambda L_\pm, \  \ Q_\pm=\pm \sigma_1.\label{alaxt}
   \end{align}
      This spectral problem  is easy to be solved with a solution
    \begin{align}
    	\varphi^\pm & = Y_\pm e^{-it\theta(z)\sigma_3}, \quad z \not= \pm 1,\nonumber \\
    	\varphi^\pm & = Y_\pm, \quad z = \pm 1, \nonumber
    \end{align}
	where
    \begin{align}
       & Y_\pm=I\pm \sigma_1/z,\quad \text{det}Y_\pm=1-z^{-2},\\
       & \theta(z)=\zeta(z) \big(x/t- 2 \lambda(z) \big),\quad\zeta(z)=(z-z^{-1})/2.
     \end{align}

  We define Jost solutions $\psi^\pm$ of  Lax pairs (\ref{laxx})-(\ref{laxt})  with the  asymptotic behavior
	\begin{equation*}
		\psi^\pm(z;x,t) \sim  \varphi^\pm(z;x,t),\quad x \to \pm \infty.
	\end{equation*}
	Make  a transformation
	\begin{equation}
		m^\pm(z;x,t)=\psi^\pm(z;x,t)e^{it\theta(z)\sigma_3},
	\end{equation}
then we have
    \begin{equation}
    		m^\pm(z;x,t) \sim Y_{\pm}, \quad x \to \pm \infty,
    \end{equation}
	and  a new Lax pair
	\begin{align}
		&(Y^{-1}_\pm m^\pm)_x - i \zeta \left[ Y^{-1}_\pm m^\pm, \sigma_3 \right]=Y^{-1}_\pm \Delta L_\pm m^\pm, \quad z \not= \pm 1,\label{mt1}\\
		&(Y^{-1}_\pm m^\pm)_t + 2i \zeta \lambda \left[ Y^{-1}_\pm m^\pm, \sigma_3 \right]=Y^{-1}_\pm \Delta T_\pm m^\pm, \quad z \not= \pm 1,\label{mt2}		
	\end{align}
    where $\Delta L_\pm=L-L_\pm$, $\Delta T_\pm =T-T_\pm$.
  The formula (\ref{mt1}) can be converted  into  the integral equation:
For $z \not= \pm 1$,
\begin{equation}\label{mn1}
m^\pm(z;x)=Y_\pm + \int_{\pm \infty}^{x} \left[ Y_\pm e^{-i \zeta(z) (x-y)\sigma_3}  Y^{-1}_\pm    \right] \left[ \Delta L_\pm(z;y)m^\pm(z;y)      \right] e^{i \zeta(z) (x-y) \sigma_3}  \, \mathrm{d}y.
\end{equation}
Then taking the limit of (\ref{mn1}), we find as $z \to \pm 1$,
\begin{equation}
m^\pm(\pm 1;x)=Y_\pm + \int_{\pm \infty}^{x} \left[ I + (x-y)L_\pm      \right] \Delta L_\pm (z;y) m^\pm(z;y) \, \mathrm{d}y,
\end{equation}
where we use the property that $\Delta L_\pm$  is sufficiently decayed.
The existence, analyticity  and differentiation  of $m^\pm$ can be proven directly,     here we just list their  properties,
for details, see \cite{CJ}.

       \begin{figure}
    	\begin{center}
    		\begin{tikzpicture}
    			%\draw[CadetBlue!20, fill=LightSteelBlue!20] (0,0)--(2.5,0)--(2.5,2)--(0, 2);
    			%\draw[CadetBlue!20, fill=LightSteelBlue!20] (-2.5,0)--(0,0)--(0,2)--(-2.5, 2);
    			%\draw[GreenYellow!20, fill=YellowGreen!20] (-2.5,-2)--(0,-2)--(0,0)--(-2.5, 0);
    			%\draw[GreenYellow!20, fill=YellowGreen!20] (0,-2)--(2.5,-2)--(2.5,0)--(0, 0);
    			\draw [-> ](-3.5,0)--(3.5,0);
    			\draw [-> ](0,-2.8)--(0,2.8);
    			\node    at (0.3,-0.3)  {$0$};
    			\node    at (4,0)  { Re$z$};
    			\node    at (0,3.2)  { Im$z$};
    			\node  [below]  at (1.3,1.2) {$\mathbb{C}^+$};
    			\node  [below]  at (-1.2,1.2) {$\mathbb{C}^+$};
    			\node  [below]  at (-1.2,-0.8) {$\mathbb{C}^-$};
    			\node  [below]  at (1.3,-0.8) {$\mathbb{C}^-$};
    		\end{tikzpicture}
    	\end{center}
    	\caption{\small{The analytical region of $m(z)$: $m_1^-$ and $m_2^+$ are analytic in $\mathbb{C}^+$; $m_1^+$ and $m_2^-$ are analytic in $\mathbb{C}^-$.} }
    	\label{analytic}
    \end{figure}
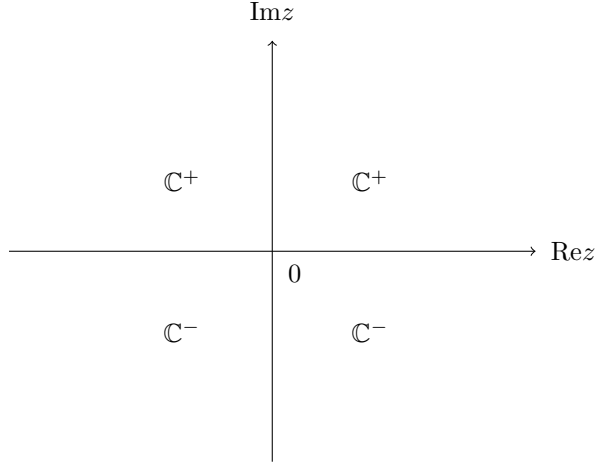

  %  \begin{lemma}
    	%Let $q(x) \in \tanh (x)+L^1(\mathbb{R})$.
    	%$m_1^-(z)$ and $m_2^+(z)$ can be analytically extended to $z \in \mathbb{C}^+$. Similarly,  $m_1^+(z)$ and $m_2^-(z)$  can be analytically extended to $z \in \mathbb{C}^-$. See Figure \ref{analytic}.
       % For any $x_0 \in \mathbb{R}$, $m_1^-(z)$ and $m_2^+(z)$ are continuous maps defined on the upper half plane:
    	%\begin{align}
    	%	m^-_1(z)&: \overline{\mathbb{C}^+}/ \{-1,0,1 \} \to C^1\big((-\infty, x_0], \mathbb{C}^2\big) \cap W^{1,\infty}\big((-\infty, x_0], \mathbb{C}^2\big),\\
    	%	m_2^+(z)&: \overline{\mathbb{C}^+}/ \{-1,0,1 \} \to C^1\big([x_0,\infty), \mathbb{C}^2\big) \cap W^{1,\infty}\big([x_0,\infty), \mathbb{C}^2\big).
    	%\end{align}
    	%Similarly, $m_1^+(z)$ and $m_2^-(z)$ are continuous maps defined on the lower half plane:
    	%    	\begin{align}
    	%	m^-_2(z)&: \overline{\mathbb{C}^+}/ \{-1,0,1 \} \to C^1\big((-\infty, x_0], \mathbb{C}^2\big) \cap W^{1,\infty}\big((-\infty, x_0], \mathbb{C}^2\big),\\
    	%	m_1^+(z)&: \overline{\mathbb{C}^+}/ \{-1,0,1 \} \to C^1\big([x_0,\infty), \mathbb{C}^2\big) \cap W^{1,\infty}\big([x_0,\infty), \mathbb{C}^2\big).
    	%\end{align}	
   % \end{lemma}
    %	For the differentiability of Jost functions, we have the following lemma.
    	\begin{lemma}\label{sin1}
        Let $m^\pm_j(z;x), j=1,2$ be the column vector solutions of  the equation (\ref{mn1}).
    	Given $n \in \mathbb{N}_0$ and $q\in \tanh(x) + L^{1,n}(\mathbb{R})$,
    	$m_1^+(z;x)$ and $m_2^-(z;x)$  can be analytically extended to $z \in \mathbb{C}^-$. Similarly, 	$m_1^-(z;x)$ and $m_2^+(z;x)$  can be analytically extended to $z \in \mathbb{C}^+$  (See Figure \ref{analytic}).
    	%$m^\pm_1(z;x)$ and $m^\pm_2(z;x)$ are differentiable for $z$.
        For any $x_0 \in \mathbb{R}$, $m_1^+(z;x)$ and $m^-_2(z;x)$ are continuous differentiable maps defined on the lower half plane.
    	\begin{align}
    	\partial^n_z m^+_1(z;x)&: \overline{\mathbb{C}^-} \backslash \{-1,0,1 \} \to C^1\big([x_0,\infty), \mathbb{C}^2\big) \cap W^{1,\infty}\big( [x_0,\infty), \mathbb{C}^2\big),\\
    	\partial^n_z m_2^-(z;x)&: \overline{\mathbb{C}^+} \backslash \{-1,0,1 \} \to C^1\big((-\infty, x_0], \mathbb{C}^2\big) \cap W^{1,\infty}\big((-\infty, x_0], \mathbb{C}^2\big).
    	\end{align}
    	Moreover, maps $q(x) \to  \frac{\partial^n}{\partial z^n} m^+_1(z;x)$ and $q(x) \to  \frac{\partial^n}{\partial z^n} m^-_2(z;x)$ are locally Lipschitz continuous from
    	\begin{align}
    	\tanh(x)+L^{1,n}(\mathbb{R}) &\to  L^\infty_{loc} \left(  \overline{\mathbb{C}^-} \backslash \{-1,0,1 \},C^1\big([x_0,\infty), \mathbb{C}^2\big) \cap W^{1,\infty}\big( [x_0,\infty), \mathbb{C}^2\big)   \right),\nonumber\\
    	\tanh(x)+L^{1,n}(\mathbb{R}) &\to  L^\infty_{loc} \left(  \overline{\mathbb{C}^-} \backslash \{-1,0,1 \}, C^1\big((-\infty, x_0], \mathbb{C}^2\big) \cap W^{1,\infty}\big((-\infty, x_0], \mathbb{C}^2\big)   \right).\nonumber
    	\end{align}
    	Similar results hold for $q(x) \to m_1^-(z;x)$ and for $q(x) \to m_2^+(z;x)$.
    	
    	In particular, there exists a function $F_n( \cdot )$, which is increasing and independent of $q$, such that
    	\begin{equation}
    	|\partial_z^n m^+_1(z;x)|\le F_n  \left( \left(1+|x|\right)^n ||q-1||_{L^{1,n}(x,\infty)} \right), \quad z \in \overline{\mathbb{C}^- }\backslash \{-1,0,1\}.
    	\end{equation}
    	Moreover, assume that $q(x)$ and $\tilde{q}(x)$ are sufficiently close and then we obtain
    	\begin{equation}
    	|\partial_z^n \left( m_1^+(z;x)-   \tilde{m}_1^+(z;x) \right)| \le  ||q-\tilde{q}||_{L^{1,n}(x,\infty)} F_n\left( \left( 1+|x| \right)^n||q-1||_{L^{1,n}(x,\infty)} \right).
    	\end{equation}
    	The other Jost functions also have similar statements.
    	
    \end{lemma}

    We know from Lemma \ref{sin1} that the Jost functions have singularities  at points $-1$, $0$ and $1$.
    The following lemma shows that the singularity at the point $z=0$ can not be removed, but the singularities at points $z=\pm 1$ can be removed by improving the attenuation of the initial data.
    %\begin{lemma}
    %	Fix $n \in \mathbb{N}_0$. Suppose $q(x) \in \tanh(x)+L^{1,n+1}(\mathbb{R})$ and $K$  be a compact support of $\{-1,1\}$ in $\overline{\mathbb{C}_+} / \{0 \}$. There exists a $c$ such that for $z\in K$ we obtain
    %  \begin{equation}
    %	\left|m^-_1(z;x)-\left(  \begin{array}{c}
    %		1\\
    %		1/z
    %	\end{array}\right)  \right|\le c \langle x^+ \rangle e^{c \int_{-\infty}^{x} \langle y-x \rangle |q(y)-1| \, \mathrm{d}y} ||q+1||_{L^{1,1}(-\infty,x)},
    %\end{equation}
    %	where $x^+= \text{max} \{x,0 \}$. Thus, we have the map $q \to m^-_1(z; \cdot)$ is locally Lipschitz continuous from
    %	\begin{equation}
    %		\tanh(x)+L^{1,1}(\mathbb{R}) \to L^\infty \Big( \overline{\mathbb{C}_+} / \{0\}, C^1\big((-\infty, x_0], \mathbb{C}^2\big) \cap W^{1,\infty}\big((-\infty, x_0]   , \mathbb{C} \big)  \Big).
    %	\end{equation}
    %    Similar results are obtained for other Jost functions and their derivatives.	
    %\end{lemma}
   \begin{lemma}
    Let $P$  be a compact support of $\{-1,1\}$ in $\overline{\mathbb{C}^-} \backslash \{0 \}$.	Then, for $q(x)\in L^{1,n+1}(\mathbb{R})+\tanh(x)$ where $n \in \mathbb{N}_0$ is fixed. There exists a constant $c$ such that for $z\in P$ we obtain
    \begin{equation}
    \left|m^+_1(z;x)-\left(  \begin{array}{c}
    1\\
    1/z
    \end{array}\right)  \right|\le c \langle x^- \rangle e^{c \int_{x}^{\infty} \langle y-x \rangle |q(y)-1| \, \mathrm{d}y} ||q-1||_{L^{1,1}(x,\infty)},
    \end{equation}
    where $x^-= \text{max} \{-x,0 \}$. Thus, we know the map $q \to m^+_1(z; \cdot)$ is locally Lipschitz continuous from
    \begin{equation}
    \tanh(x)+L^{1,1}(\mathbb{R}) \to L^\infty \Big( \overline{\mathbb{C}^-}  \backslash \{0\}, C^1\big( [x_0,\infty), \mathbb{C}^2\big) \cap W^{1,\infty}\big( [x_0,\infty)   , \mathbb{C} \big)  \Big).
    \end{equation}
    Furthermore, there exists a function $F_n(t)$, which is increasing and independent of $q$, such that
    \begin{equation}
    |	\partial_z m^+_1(z)| \le  F_n\left( (1+|x|)^{n+1} ||q-1||_{L^{1,n+1}(x,\infty)}   \right), \; z \in P.
    \end{equation}
   \end{lemma}

   The next lemma considers the asymptotics of the Jost functions as $z \to \infty$ and $z \to 0$.
    %  Let
    %  \begin{align}
    %   	C_+(x)&=||q-1||_{W^{2,1}{(x,\infty)}} \left( 1+ ||q-1||_{W^{2,1}{(x,\infty)}}  \right)^2 e^{||q-1||_{L^1(x,\infty)}},\\
    %  	C_-(x)&=||q+1||_{W^{2,1}{(-\infty,x)}} \left( 1+ ||q+1||_{W^{2,1}{(-\infty,x)}}  \right)^2 e^{||q+1||_{L^1(-\infty,x)}}.
    %    \end{align}
    \begin{lemma}
    	Assume that $q(x) \in \tanh (x)+ L^1(\mathbb{R})$ and $q'(x)\in W^{1,1}(\mathbb{R})$.
    	Then $m^\pm_1(z)$ and $m^\pm_2(z)$ have the following asymptotic behaviors:
    	\begin{align*}
    	m^\pm_1(z) &= e_1 +  \mathcal{O} \left(   z^{-1} \right); \    \  m^\pm_2(z) = e_2 +   \mathcal{O} \left(   z^{-2} \right),  z \to \infty,\nonumber\\
    	m^\pm_1(z) &= \pm \frac{e_2}{z} +\mathcal{O}(1);  \ \  m^\pm_2(z)  = \pm \frac{e_1}{z} +\mathcal{O}(1),  \  \; z\to 0.
    	\end{align*}
    	 There exists an increasing function $F_n(t)$ for $q-\tanh(x) \in L^{1,n}(\mathbb{R})$ such that
       \begin{equation}
        |	\partial_z^j m^+_1(z)| \le |z|^{-1} F_n\left( (1+|x|)^n ||q-1||_{L^{1,n}(x,\infty)}   \right), \; 0\le j \le n,\; z \to \infty.
        \end{equation}
       For any fixed two potential $q$ and $\tilde{q}$ which are sufficiently close to the other, we find
    	\begin{equation}
    	|\partial_z^n \left( m_1^+(z;x)-   \tilde{m}_1^+(z;x) \right)| \le |z|^{-1}||q-\tilde{q}||_{L^{1,n}(x,\infty)} F_n \left( \left( 1+|x| \right)^n||q-1||_{L^{1,n}(x,\infty)} \right),
    	\end{equation}
    	where $0 \le j \le n$.

    \end{lemma}

The Jost functions  have the following symmetry.
    \begin{lemma}\label{sym1}
    	Suppose that $q \in \tanh(x)+L^1(\mathbb{R})$. For $z\in \mathbb{C}  \backslash \{-1,0,1\}$, we find
    	\begin{equation}
    	\psi_\pm(z;x)=\sigma_1 \overline{\psi_\pm(\bar{z};x)}\sigma_1=\pm z^{-1} \psi^\pm(z^{-1};x)\sigma_1.
    	\end{equation}
    	The above symmetry is expanded as follows according to matrix columns
    	\begin{equation}
    	\psi^\pm_1(z)=\sigma_1 \overline{\psi^\pm_2(\bar{z})}=\pm z^{-1} \psi^\pm_2(z^{-1}),\;
    	\psi^\pm_2(z)=\sigma_1 \overline{\psi^\pm_1(\bar{z})}=\pm z^{-1} \psi^\pm_1(z^{-1}).
    	\end{equation}
    	
    \end{lemma}

	\section{A RH problem with  nonzero background }
	   \subsection{Scattering data and reflection coefficient}
	
	  Since $\psi_\pm(z)$ are two matrix solutions of Lax pairs (\ref{laxx})-(\ref{laxt}), there exists a  spectral  matrix
    $$S(z)= \left(\begin{array}{cc}
    s_{11}(z)& 	s_{12}(z)\\
    s_{21}(z)& 	s_{22}(z)
    \end{array} \right),$$
    such that
    \begin{equation}
    \psi_-(z)=\psi_+(z)S(z),
    \end{equation}
    where $s_{ij}(z)(i,j=1,2)$ are scattering data, by which   we define a reflection coefficient
	    \begin{equation}
	    	r(z) := \frac{s_{21}(z)}{s_{11}(z)}.
	    \end{equation}
	    According to the lemma (\ref{sym1}),  the scattering data   have  the following  properties \cite{CJ}
       \begin{lemma}\label{sds11s21}
       	Let $q \in \tanh(x)+L^1(\mathbb{R})$ and $z\in \mathbb{R}  \backslash \{-1,0,1\}$.
       	Then
       	\begin{itemize}
       		%\item The scattering matrix $S(z)$ follows
       		%\begin{equation}
       		%	S(z)=\sigma_1 \overline{S(z)} \sigma_1 = - \sigma_1 S(z^{-1}) \sigma_1.
       		%\end{equation}
       	    \item The scattering coefficients can be described by the Jost functions as
       	    \begin{equation}
       	    	s_{11}(z)=\frac{\det  \left[\psi_1^-(z;x), \psi_2^+(z;x)\right]  }{1-z^{-2}},\quad        	    	s_{21}(z)=\frac{\det  \left[\psi_1^+(z;x), \psi_1^-(z;x)\right]  }{1-z^{-2}}.
       	    \end{equation}
            \item $|s_{11}(z)|^2=1+|s_{21}(z)|^2 \ge 1$, $z\in \mathbb{R}  \backslash \{-1,0,1 \}$.
            \item $|r(z)|^2=1-|s_{11}(z)|^{-2}<1$, $z\in \mathbb{R}  \backslash \{-1,0,1 \}$.
            \item For $z \in \mathbb{C}^+$,
            %$s_{11}(z)=-\overline{s_{11} \left(\bar{z}^{-1}\right)},\; s_{21}(z)=-\overline{s_{21} \left(\bar{z}^{-1}\right)},\; r(z)=\overline{r\left( \bar{z}^{-1}   \right)}.$
            \begin{equation}
            	s_{11}(z)=-\overline{s_{11} \left(\bar{z}^{-1}\right)},\; s_{21}(z)=-\overline{s_{21} \left(\bar{z}^{-1}\right)},\; r(z)=\overline{r\left( \bar{z}^{-1}   \right)}.
            \end{equation}
            \item If we add a condition $q' \in W^{1,1}(\mathbb{R})$, then we can find, as $z \in \overline{\mathbb{C}^+}$,
            \begin{align}
            	&\lim_{z \to \infty} \left(s_{11}(z)-1 \right)z= i \int_{\mathbb{R}} \left( |q(x)|^2-1\right) \, \mathrm{d}x,\\
            	&\lim_{z \to 0} \left(s_{11}(z)+1 \right)z^{-1}= i \int_{\mathbb{R}} \left( |q(x)|^2-1\right) \, \mathrm{d}x,
            \end{align}
            and as $z \in \mathbb{R}$
       \begin{equation}
       	|s_{21}(z)|=\begin{cases}
       		\mathcal{O}\left( |z|^{-2}\right),\; |z|\to \infty, \\
       		\mathcal{O}\left( |z|^{2}\right),\; |z|\to 0,
       	\end{cases} \quad r(z) \sim \begin{cases}
       	z^{-2},\; |z|\to \infty, \\
       	0,\; |z|\to 0.
       \end{cases}
       \end{equation}
          \item  $s_{11}(z)$ and $s_{21}(z)$ both have simple poles at points $z=\pm 1$ and their residues at these points are proportional. Then, $z=\pm1$ are the removable poles of $r(z)$.
          \begin{equation}
          	s_{11}(z)=\frac{s_\pm}{z\mp1}+\mathcal{O}(1),\; s_{21}(z)=\mp \frac{s_\pm}{z\mp1}+\mathcal{O}(1),\; \lim_{z \to \pm 1}r(z)=\mp 1,
          \end{equation}
         where $s_\pm$=$\frac{1}{2} \det \left[ \psi_1^-(\pm 1;x), \psi_2^+(\pm 1;x)  \right]$.

  	\end{itemize}
         \end{lemma}

     %  	\begin{lemma}
     %  		If we give the conditions $q(x)\in \tanh(x)+L^{1,2}(\mathbb{R})$ and $q'(x)\in W^{1,1}(\mathbb{R})$, we have $r(z)\in H^{1}(\mathbb{R})$.
     %  	\end{lemma}
     %  \begin{lemma}
     %  	 If $q_0\in \tanh(x)+\Sigma_2$, the reflection coefficient meets
     %  	 \begin{equation}
     %  	 	||\log(1-|r|^2)||_{L^p(\mathbb{R})} < \infty, \quad p\ge1.
      % 	 \end{equation}  	
     %  	\end{lemma}

  We consider the  discrete spectrum relating to the initial value $q_0$, which is formed by a finite number of zeros of
  $s_{11}(z)$.    Suppose that $z_j \in\mathbb{C}^+(j=1,\cdots,N)$ are the zeros of  $s_{11}(z)$. Then,  from the symmetries,
   we know $\bar{z}_j \in \mathbb{C}^-(j=1,\cdots,N) $ are the corresponding zeros of  $s_{22}(z)$.
The correlation $\mathcal{D}=\left\{ r(z) ,\left(z_j,c_j\right)^{N}_{j=1} \right\}$ is called the scattering data associated with the initial value $q_0$.
 With the above lemmas, we can prove
      \begin{lemma}
      	Let $q(x) \in \tanh(x)+L^{1,2}(\mathbb{R})$ and $q'(x)\in W^{1,1}(\mathbb{R})$. Then we have
      	\begin{itemize}
            \item $r(z)\in H^{1}(\mathbb{R})$.
            \item  If $q_0\in \tanh(x)+H^{2,2}(\mathbb{R})$, the reflection coefficient meets
       	 \begin{equation}
       	 	||\log(1-|r|^2)||_{L^p(\mathbb{R})} < \infty, \quad p\ge1.
       	 \end{equation}
      		\item $s_{11}(z)$ has no spectral singularity on the real axis and its zeros are simple, finite and distributed on the unit circle. For convenience, we define the corresponding discrete spectrum set
      		\begin{align*}
      			&\mathcal{Z}^+ = \left\{z_j \in \mathbb{C}^+: s_{11}(z_j)=0, z_j=e^{i w_n}, 0 < w_j<\pi      \right\},\\
      			&\mathcal{Z}^- = \left\{\bar{z}_j \in \mathbb{C}^-: s_{22}(\bar{z}_j)=0, \bar{z}_j=e^{i w_j}, -\pi < w_j<0      \right\}.
      		\end{align*}
      	  %  \item For each $z_j(k=1,\cdots,n)$,  $\frac{\partial s_{11}}{\partial \lambda}(z_j)$ and $\gamma_k$ are pure imaginary which satisfies
      	%    \begin{equation*}
      	%    	\sgn(-i\gamma_k)=-\sgn \left(-i \frac{\partial s_{11}}{\partial \lambda}(z_j)   \right)
      	 %   \end{equation*}
      	%    where $\gamma_k$ are the connection coefficients between $\psi_2^+(z_j;x)$ and $\psi_1^-(z_j;x)$, i.e. $\psi_1^-(z_j;x)=\gamma_k\psi_2^+(z_j;x) $.
      	\end{itemize}
    	
      \end{lemma}

	 % \begin{remark}
	 % 	The proof of the above lemmas are given in the paper \cite{CJ}.
	%  \end{remark}

	  Moreover, we can get the trace formula
	  \begin{equation}
	  	s_{11}(z)= \prod_{j=1}^{N} \frac{z-z_j}{z-\bar{z}_j} \exp\left(-i \int_{\mathbb{R}} \frac{v(s)}{s-z} \, \mathrm{d}s\right),
	  \end{equation}
	where $z_j \in \mathcal{Z}^+ $ and
	  \begin{equation}
	  	v(s)=-\frac{1}{2 \pi} \log \left( 1-|r(s)|^2\right).
	  \end{equation}

	   \subsection{RH formulism   of the  initial value problem }

Define
	   \begin{equation}
	   	M(z)=M(z;x,t):=\begin{cases}
	   		\left(\frac{m_1^-(z;x,t)}{s_{11}(z)}, m_2^-(z;x,t)  \right), \; z\in \mathbb{C}^+,\\[4pt]
	   		\left(m_1^+(z;x,t), \frac{m_2^-(z;x,t)}{\overline{s_{11}(\bar{z})}}  \right), \; z\in \mathbb{C}^-,
	   	\end{cases}
	   \end{equation}
       then  it is easy to prove that $M(z)$ satisfies the following RH problem.

      \begin{figure}
      	\begin{center}
      		\begin{tikzpicture}[node distance=2cm]
      			\draw[->](-4,0)--(4,0)node[black,right]{Re$z$};
      			\foreach \x [count=\p] in {0,...,11} {
      		\node[shape=circle,fill=red, scale=0.25] (\p) at (-\x*30:2) {};};
      			\foreach \x [count=\p] in {0,...,5} {
      				\draw (-\x*60:2.4) ;
      				\draw (-30-\x*60:2.4) ;};
      			\node[shape=circle,fill=blue, scale=0.15]  at (0:2){0};
      			\node[shape=circle,fill=blue, scale=0.15]  at (-6*30:2){0} ;
      		\node[shape=circle,fill=blue,scale=0.15] at (0,0) {0};
      			\node[below] at (2.1,0) {$1$};
      			\node[below] at (0,0) {$0$};
      			\node[below] at (-2.2,0) {$-1$};
      			\node[right] at (1*30:2) {$z_j$};
      			\node[right] at (-1*30:2) {$\bar{z}_j$};
      		%	\node[left] at (5*30:2) {$\bar{z}_k$};
      		%	\node[left] at (-5*30:2) {$z_j$};
      			\draw [dashed, gray](1) arc (0:360:2);
      			
      		\end{tikzpicture}
      		\caption{ \small{The  jump contours and  poles for  $M(z)$.  The black solid line stands for the jump line. The red dots are poles $z_j, \ \bar z_j  \ (j=1,\cdots, N)$, and the blue dots are singular points $-1, 0 ,1$.}}
      		\label{mjumpp}
      	\end{center}
      \end{figure}
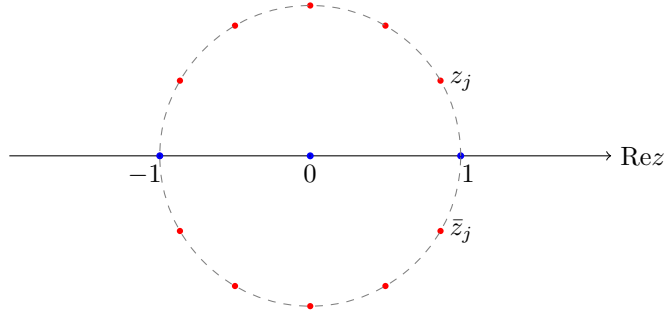

       \noindent\textbf{RHP1}.  Find a matrix-valued function $M(z)$ which satisfies
      \begin{itemize}
      	\item Analyticity: $M(z)$ is meromorphic in $\mathbb{C}\backslash \mathbb{R}$.
      	\item Symmetry: $M(z)=\sigma_1 \overline{M(\bar{z})} \sigma_1=z^{-1} M(z^{-1}) \sigma_1$.
      	\item Jump condition: $M(z)$ satisfies the jump condition
	$$M_+(z)=M_-(z)V(z), \; z \in \mathbb{R},$$
      	where
      	\begin{equation}\label{V0}
      		V(z)=\left(\begin{array}{cc}
      			1-|r(z)|^2 & -e^{-2it \theta(z)}\overline{r(z)}\\
      			e^{2it \theta(z)}r(z) & 1
      		\end{array}\right),
      	\end{equation}
      with  $ \theta(z)=\theta(z;x,t)=\zeta(z) [  {x}/{t}-2 \lambda(z)].$

      	\item Asymptotic behaviors:
      	\begin{align*}
      			&M(z)=I+\mathcal{O}(z^{-1}),	\quad  z \to  \infty,\\
      			&zM(z)=\sigma_1+\mathcal{O}(z), \quad z \to 0.
      	\end{align*}
      	\item Residue conditions: $M(z)$  has simple poles at each points $z_j$ in $\mathcal{Z}^+ \cup \mathcal{Z}^-$ with the following residue conditions:
      	\begin{align}
      		& \res_{z=z_j} M(z)=\lim_{z\to z_j} M(z) \left(\begin{array}{cc}
      			0 & 0\\
     			c_j e^{2it\theta(z_j)} & 0
      		\end{array} \right),\\
        	& \res_{z=\bar{z}_j} M(z)=\lim_{z\to \bar{z}_j} M(z) \left( \begin{array}{cc}
     		0 & \bar{c}_j e^{-2it\theta(\bar{z}_j)}\\
     		0 & 0
      	     \end{array} \right),
      	\end{align}
        where $c_j=\frac{s_{21}(z_j)}{s'_{11}(z_j)}=i z_j |c_j|.$	 The jump contours and poles of $M(z)$ can be seen in Figure \ref{mjumpp}.
	\end{itemize}
       Using the asymptotic property of $M(z)$ at infinity, we can obtain the reconstruction formula for the potential $q(x,t)$
       \begin{equation}\label{reconst}
       	q(x,t)= \lim_{z \to \infty} \left(z M(z;x,t) \right)_{21}.
       \end{equation}

	\section{Normalization of the RH problem}
	 In this section, we make factorizations of the jump matrix $V(z)$ and renormalize the RH problem of $M(z)$ so that it is well-behaved at infinity.
	   \subsection{Jump matrix  factorizations}

            \begin{figure}[htbp]
    	\centering
    	    	\subfigure[$ \xi>1$]{
    		\begin{minipage}[t]{0.32\linewidth}
    			\centering
    			\includegraphics[width=1.5in]{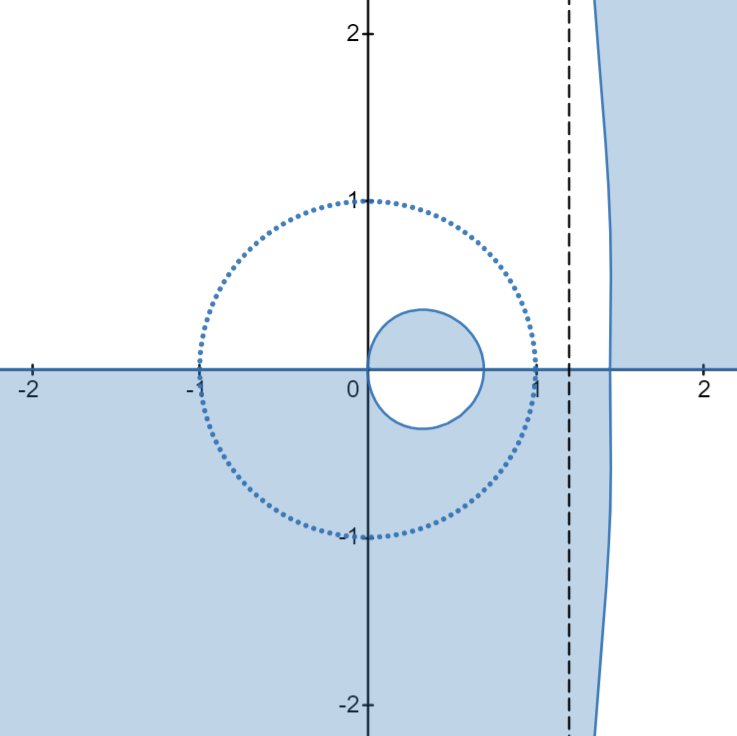}
    		\end{minipage}
    	}%
    	\subfigure[$\xi=1$]{
    		\begin{minipage}[t]{0.32\linewidth}
    			\centering
    			\includegraphics[width=1.5in]{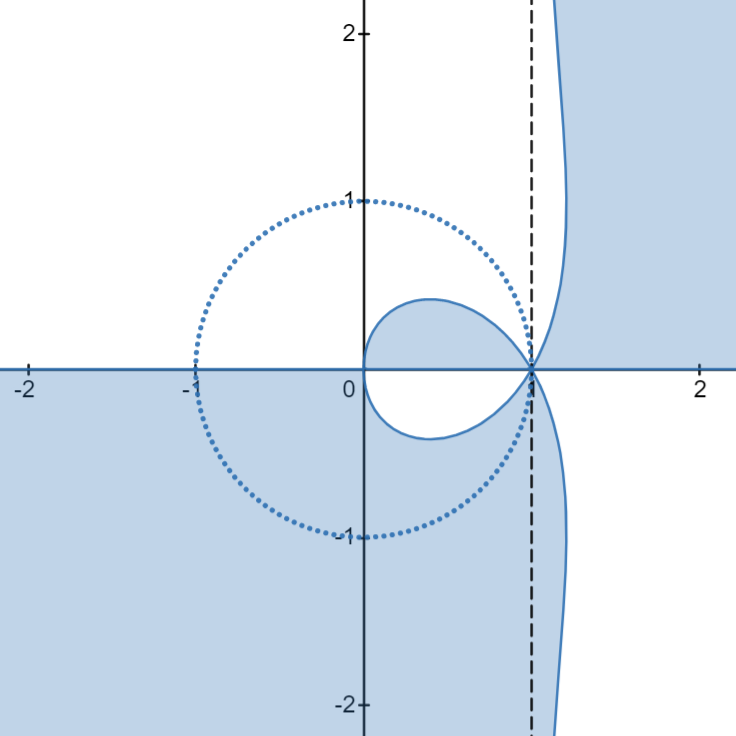}
    		\end{minipage}%
    	}%
    	\subfigure[$0<\xi<1$]{
	\begin{minipage}[t]{0.32\linewidth}
		\centering
		\includegraphics[width=1.5in]{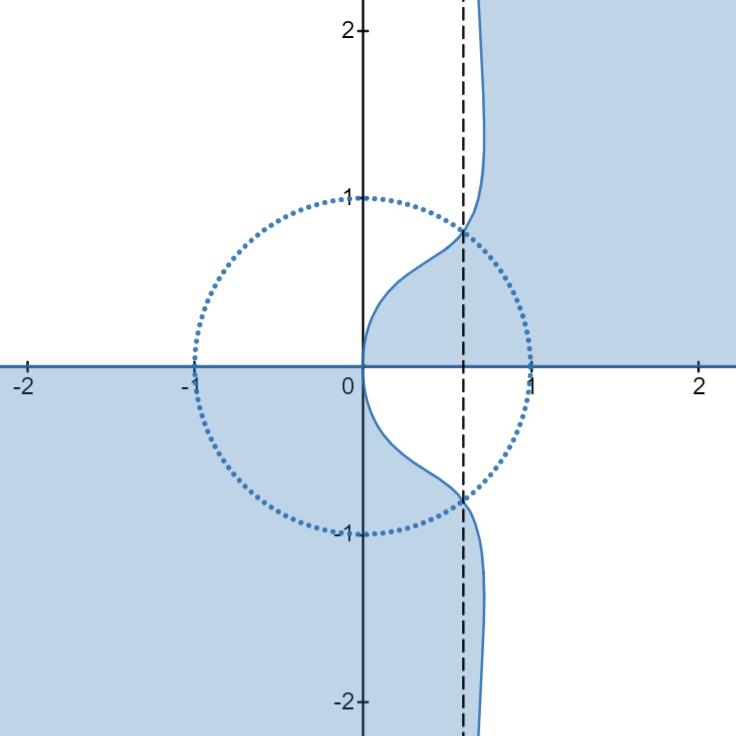}
	\end{minipage}%
}%

    \subfigure[$-1<\xi<0$]{
    	\begin{minipage}[t]{0.32\linewidth}
    		\centering
    		\includegraphics[width=1.5in]{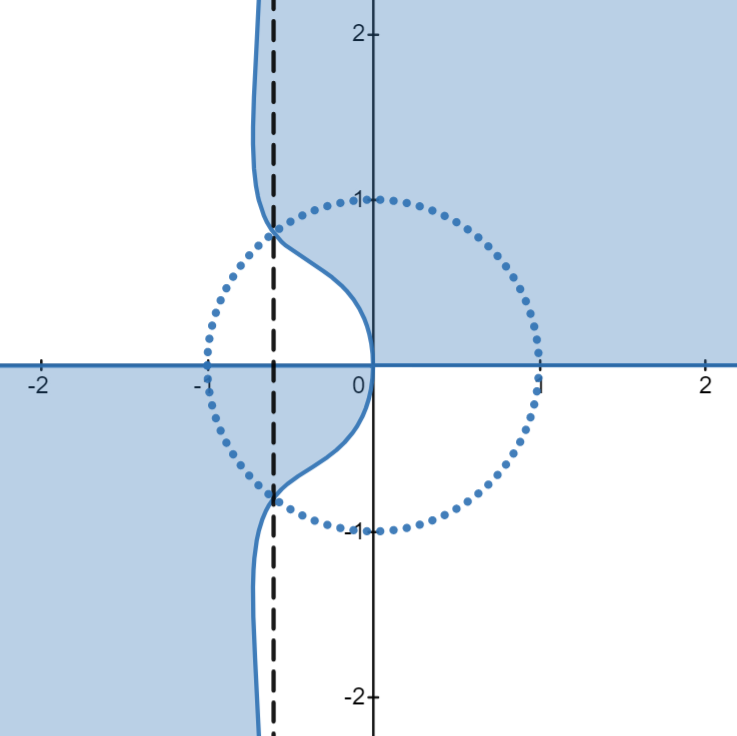}
    		%\caption{fig2}
    	\end{minipage}
    }%
    \subfigure[$\xi=-1$]{
    	\begin{minipage}[t]{0.32\linewidth}
    		\centering
    		\includegraphics[width=1.5in]{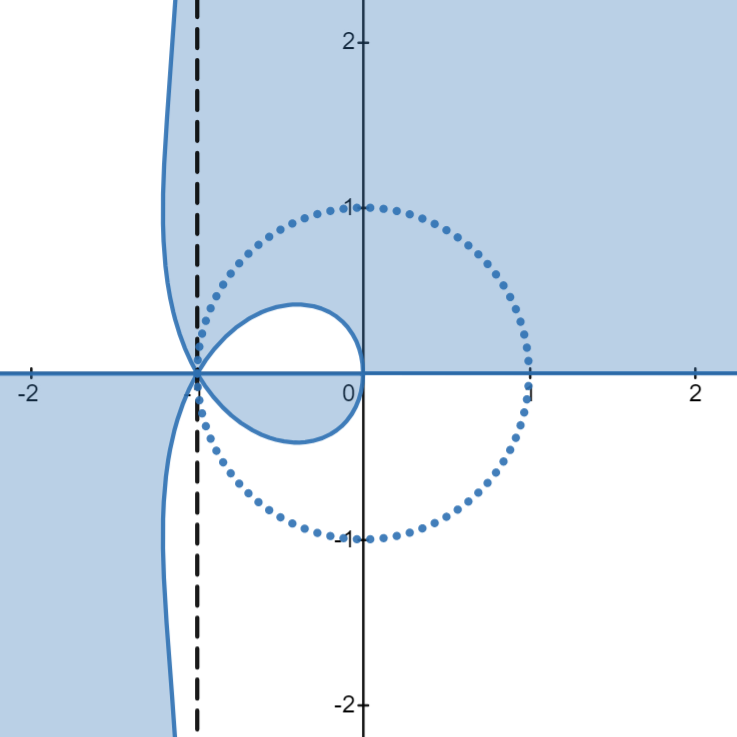}
    		%\caption{fig2}
    	\end{minipage}
    }%
    \subfigure[$\xi<-1$]{
	\begin{minipage}[t]{0.32\linewidth}
		\centering
		\includegraphics[width=1.5in]{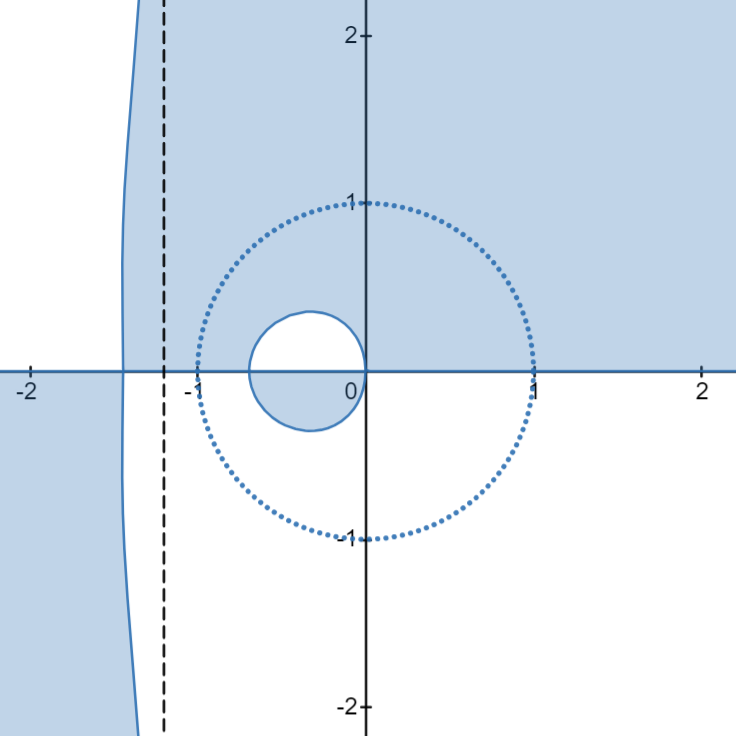}
		%\caption{fig2}
	\end{minipage}
}%
    	\centering
    	\caption{\small The signature table   of $\re (2i\theta)$.
    In the blue region, we have $\re (2i\theta)>0$, which implies that $e^{-2it\theta} \to 0$ as $t\to \infty$; In the white region,  $\re (2i\theta)<0$,
    which implies that $e^{2it\theta} \to 0$ as $t\to \infty$. }
    	\label{proptheta}
    \end{figure}

    The long-time asymptotic behavior of RH problems is affected by growth and decay  of the  oscillatory term $e^{\pm 2it\theta(z)}$ in the jump matrix $V(z)$.
Direct calculations show that
        \begin{align}
	     \re\left(2it\theta(z)\right) = \xi \im z \left( 1+\frac{1}{\re^2 z+ \im^2 z} \right)- \re z \im z \left( 1 + \frac{1}{\left(\re^2 z+ \im^2 z\right)^2} \right),
        \end{align}
where $\xi:=\frac{x}{2t}$.  It can be found that the sign  of $\re (2i\theta(z))$ changes  with $\xi$.
 The signature table   of $\re( 2 i\theta(z) )$ is shown   in Figure \ref{proptheta}.
\begin{itemize}
\item  For the case $|\xi|<1$, there is no phase point on $\mathbb{R}$ corresponding to the figures (c) and (d),
 which  were discussed by  Cuccagna and Jenkins   \cite{CJ};

\item  For the case  $|\xi|=1$, there is a phase point on $\mathbb{R}$ corresponding to
  the figures (b) and (e), which  are critical cases;

\item For the case $ |\xi|>1$, two phase points appear on $\mathbb{R}$, corresponding to the  figures (a)  and (f),
  which  will be considered in our paper.

\end{itemize}

Next we search for stationary phase points of the  function $\theta(z)$. Direct calculation gives
        \begin{align}
           % &2\theta(z)= \frac{x}{t}\left(z-z^{-1}\right)-(z^2-z^{-2}),\\
            &2\theta'(z)=\frac{x}{t}\left(1+z^{-2}\right)-(2z+2z^{-3})=-2z^{-1}l(k),\label{2theta}\\
            &2\theta''(z)=2z^{-2}l(s)-2z^{-1}l'(s),\label{theta2}
        \end{align}
        where $l(s)=s^2-\xi s-2,\quad s=z+z^{-1}$.  We find (\ref{2theta}) has two kinds of zeros on $\mathbb{R}$
         \begin{align}
         	&\xi_{k}(\xi)=  \frac{1}{2}\left|  \nu(\xi)+( -1)^k \sqrt{ \nu^2(\xi) -4} \right|,\quad  k=1, 2, \ \ {\rm for} \ \   \xi>1,\label{xi1}\\
         &\xi_{k}(\xi)= -\frac{1}{2}\left|  \nu(\xi) +( -1)^k \sqrt{ \nu^2(\xi) -4} \right|,\quad  k=1, 2, \ \ {\rm for} \ \  \xi<-1,\label{xi2}
          \end{align}
   where $\nu(\xi)=	\frac{1}{2} (|\xi| + \sqrt{\xi^2+8})$.  Based on the symmetry of zeros and the formula (\ref{xi1})-(\ref{xi2}),  we deduce that
        \begin{equation}
           \begin{cases}
            	0<\xi_1(\xi)<1<\xi_2(\xi)\, \quad \xi>1,\\
             	\xi_2(\xi)<-1<\xi_1(\xi)<0\, \quad \xi<-1,    	
          \end{cases}
 \end{equation}
with $\theta''(\xi_1(\xi))>0\, \text{and} \, \theta''(\xi_2(\xi))<0$.

       % \textcolor{red}{Pay attention to choose a way of expression!}
       % For $z=|z|e^{iw}$,
	   % \begin{equation}
	   % 	\re\left(2it\theta(z)\right) =  \sin(2w) \left( |z|+|z|^{-1}\right)^2 - \xi \sin(2w) \left( |z|+|z|^{-1}\right) \sec(w)-2\sin(2w).
	   % \end{equation}

 The jump matrix has the following two kinds of factorizations
	   \begin{equation}
	   	V(z)=        \begin{cases}
	   		\left(\begin{array}{cc} 1&-\bar{r}e^{-2it\theta }\\ 0&1\end{array}  \right)
	   		\left(\begin{array}{cc} 1&0\\   re^{2it\theta }&1\end{array}  \right),\\
	   		\left(\begin{array}{cc} 1 & 0 \\ \frac{r}{1-|r|^2}e^{2it\theta } & 1\end{array}  \right)
	   		\left(\begin{array}{cc} 1-|r|^2&0\\ 0&\frac{1}{1-|r|^2} \end{array}  \right) \left(\begin{array}{cc} 1 &
	   			\frac{-\bar{r}}{1-|r|^2}e^{-2it\theta } \\ 0 & 1\end{array}  \right).
	   	\end{cases}
	   \end{equation}
	   Moreover, according to the signal of $\re (2i\theta(z))$, we use different  factorization  forms of  the jump matrix  for different regions such that the oscillating factor $e^{\pm 2it\theta}$
       is decaying in the corresponding regions respectively.

       %In order to avoid complicated symbols,  we abbreviate $\xi_1=\xi_1^\pm$ and $\xi_2=\xi_2^\pm$, and we will indicate whether these zeros belong to the case $\xi>1$ or the case $\xi<-1$.
    \begin{remark}
        Here we only consider the zeros of the function $\theta(z)$ on $\mathbb{R}$, this is because for the zeros which on the complex plane but not on the junp curves $\Sigma$, they have no contribution to the solution of $M(z)$.  If $\xi_k$ is a zero point of $\theta(z)$, then
        \begin{itemize}
            \item   $\xi_k \notin \Sigma$: $M(z)$ is analytic at the point $\xi_k$. $\xi_k$ does not contribute to the solution of $M(z)$ and can therefore be ignored.
            \item   $\xi_k \in \Sigma$: $M(z)$ is not analytic at the point $\xi_k$ and the exponential oscillation terms $e^{\pm 2it\theta}$ in the jump matrix slow down or smooth out around $\xi_k$. The contribution of jumps near $\xi_k$ to the solution of $M(z)$ is dominant.
        \end{itemize}

    \end{remark}

       \subsection{Conjugation}
      We   introduce  a transformation to  deform the jumps on the real axis to a path on which the oscillatory term decays exponentially.
     % which renormalizes the RHP such that it is well conditioned for $t \to \infty$ with $\xi$ fixed.
      Denote
       \begin{equation}
       	 I(\xi)= \begin{cases}
       	 	(0,\xi_1) \cup \left(\xi_2,\infty\right), \quad \xi>1,\\
       	 	(0,\infty) \cup \left(\xi_2,\xi_1\right),\quad \xi<-1.
       	 \end{cases}
       \end{equation}
       We define the following function:
       \begin{equation}
        T(z)= \exp \left(  - i \int_{I(\xi)} v(s) \left(\frac{1}{s-z}- \frac{1}{2s} \right) \, \mathrm{d}s \right).\\
        \end{equation}
        \begin{proposition}
            The function $T(z)$ has the following properties:
            \begin{itemize} \label{prop1}
                \item Analyticity: $T(z)$ is analytical  in $\mathbb{C} \backslash I(\xi)$.
                \item Symmetry: $\overline{T(\bar{z})}=T(z)^{-1}=T(z^{-1})$.
                \item Jump condition:\begin{equation}
                    T_+(z)=T_-(z)(1-|r(z)|^2), \quad z\in I(\xi).
                \end{equation}
                \item Asymptotic behavior: Let
                \begin{equation}
                    T(\infty) := \lim_{z \to \infty} T(z)=\exp \left(i \int_{I(\xi)} \frac{v(s)}{2s} \, \mathrm{d}s \right).
                \end{equation}
                Then,  $|T(\infty)|=1$ and the asymptotic expansion at infinity is
                \begin{equation}
                    T(z)=T(\infty)\left( I + \frac{1}{2\pi i z} \int_{I(\xi)} \log (1-|r(s)|^2) \, \mathrm{d}s +\mathcal{O}\left(z^{-2}\right)       \right).
                \end{equation}
            \item Boundedness: The ratio $\frac{s_{11}(z)}{T(z)}$ is holomorphic in $\mathbb{C}^+$ and $\left|\frac{s_{11}(z)}{T(z)} \right|$ is bounded for $z\in \mathbb{C}^+$.
               \item Local properties: For $k=1,2$,
                \begin{equation}\label{Test}
         	|T(z)-T_k(\xi_k)(z-\xi_k)^{ \epsilon_k v(\xi_k) i}|\le c ||r||_{H^1} |z-\xi_k|^{1/2}, \ \epsilon_k=(-1)^{k+1},
                \end{equation}
                 where $z=\xi_k +  \mathbb{R}^+ e^{i\varphi_k}, \, |\varphi_k|<\pi,$ and
                \begin{align}\label{betai}
                	 &T_k(z) =T(\infty)e^{i \beta_k(z;\xi_k)},\; \beta_k(z;\xi_k)=  \epsilon_k  v(\xi_k) \ln (z-\xi_k)+\int_{I(\xi)} \frac{v(s)}{s-z}\, \mathrm{d}s.
               \end{align}

                % \begin{align}
                %    &|T(z)-T_1(\xi_1)(z-\xi_1)^{-iv(
                %        \xi_1)}|\le c ||r||_{H^1} |z-\xi_1|^{1/2}, \quad z\in \epsilon_1,\\
                %    &|T(z)-T_2(\xi_2)(z-\xi_2)^{iv(\xi_2)}|\le c ||r||_{H^1} |z-\xi_2|^{1/2}, \quad z\in \epsilon_2,\label{T2est}
               % \end{align}
               % where $\epsilon_i$ is a sufficiently small neighborhood of $\xi_i(i=1,2)$,
             %    \begin{align}
            %       T_1(\xi_1) &=T(\infty)e^{-i \beta_1(z;\xi_1)}, \text{and}\quad T_2(\xi_2)=T(\infty)e^{-i \beta_2(z;\xi_2)},
             %      \end{align}
            %   with
            %  \begin{equation}
            % 	\begin{split}
           %  		\beta_1(z;\xi_1)&= -v(\xi_1)\log(z-\xi_1+1)+\int_{-\infty}^{\xi_1} \frac{v(s)-\chi(s)v(\xi_1)}{s-z} \, \mathrm{d}s \\
           %  		&+ \int_{\xi_2}^{+\infty} \frac{v(s)}{s-z} \, \mathrm{d}s -\int_{-\infty}^0 \frac{v(s)}{s-z} \, \mathrm{d}s,\\
           %  	\end{split}
           %  \end{equation}
           %  \begin{equation}\label{betaxi2}
            % 	 \beta_2(z;\xi_2)= %v(\xi_2)\log(z-\xi_2-1)+\int_{\xi_2}^{\infty} \frac{v(s)-\chi(s)v(\xi_2)}{s-z} \, \mathrm{d}s + \int_0^{\xi_1} \frac{v(s)}{s-z} \, \mathrm{d}s.
           %  \end{equation}

            \end{itemize}
        \end{proposition}

        \begin{proof}
        	The  first five properties of $T(z)$ are obvious, so we only give the proof for the local properties.  We first prove the local property at the neighborhood of $\xi_2$ for $\xi>1$, and other cases can be proved in similar way.
        	 For  $z=\xi_2 + \mathbb{R}^+ e^{i\varphi_2} $ ($|\varphi_2|<\pi$),  $T(z)$ can be  rewritten as
        	 \begin{align}
        	 	T(z)&=T(\infty) \exp \left(-i\int_{\xi_2}^{+\infty} \frac{v(s)}{s-z}\, \mathrm{d}s\right) \exp \left(-i\int_{0}^{\xi_1} \frac{v(s)}{s-z}\, \mathrm{d}s\right)\nonumber \\
        	 	&= T(\infty) \left( z-\xi_2\right)^{iv_2} \exp \left(i \beta_2(z;\xi_2)\right),
        	 \end{align}
        	where $\beta_2(z;\xi_2)$ is defined by (\ref{betai}).
        	Then we estimate the error between $\beta_2(\xi_2;\xi_2)$ and $\beta_2(z;\xi_2)$.
        	\begin{equation}\label{betazxi2}
        		\begin{split}
       	& \left|\beta_2(z;\xi_2)- \beta_2(\xi_2;\xi_2) \right| \le \left|v(\xi_2) \ln (\xi_2+1-z) \right| + \left|\int_{\xi_2 +1}^{+\infty}\frac{v(s)}{s-z}-  \frac{v(s)}{s-\xi_2}\,\mathrm{d}s    \right|	 \\
       	& +\left| \int_{\xi_2}^{\xi_2+1} \frac{v(s)-v(\xi_2)}{s-z} -  \frac{v(s)-v(\xi_2)}{s-\xi_2}        \,\mathrm{d}s         \right| +\left|\int_{0}^{\xi_1}\frac{v(s)}{s-z}-  \frac{v(s)}{s-\xi_2}\,\mathrm{d}s    \right|.	
        		\end{split}
        	\end{equation}
        Through simple calculations, we know
        \begin{equation}\label{est1}
        v(\xi_2) \ln (\xi_2+1-z) \sim -v(\xi_2)(z-\xi_2) + \mathcal{O}\left( (z-\xi_2)^2 \right),
        \end{equation}
         and
         \begin{equation}\label{est234}
         \left|\int_{\xi_2 +1}^{+\infty}\frac{v(s)}{s-z}-  \frac{v(s)}{s-\xi_2}\,\mathrm{d}s    \right| \lesssim ||r||_{H^1} |z-\xi_2|^{-1/2}.
         \end{equation}
    Then, we have
      $$\left| \int_{\xi_2}^{\xi_2+1} \frac{v(s)-v(\xi_2)}{s-z} -  \frac{v(s)-v(\xi_2)}{s-\xi_2}        \,\mathrm{d}s         \right|, \ \left|\int_{0}^{\xi_1}\frac{v(s)}{s-z}-  \frac{v(s)}{s-\xi_2}\,\mathrm{d}s    \right| \lesssim ||r||_{H^1} |z-\xi_2|^{-1/2}.$$
      By combining (\ref{est1}) and (\ref{est234}), the inequality (\ref{Test})  can be obtained.	
        \end{proof}

	   \subsection{Constructing interpolation functions}

 For all  poles $z_j\in \mathcal{Z}^+$ on the unit circle $|z|=1$,   we define
     $$\rho=\frac{1}{2}   \min \left(
        \operatorname*{min}\limits_{z_j\in \mathcal{Z}^+ }|{\rm Im} z_j|, \  \operatorname*{min}\limits_{z_k, z_j\in \mathcal{Z}^+ }|  z_j-z_k|,  \operatorname*{min}\limits_{n=0, \pm 1 } \{ |\xi_1-n|, |\xi_2-n|   \}  \right).$$
If we  make a small circle  with each   $z_k$ as its center point and  $\rho$ as a radius respectively, then
 they are  disjoint each other and real axis.  See Figure \ref{fig5}.
 For convenience, we define a  directed path
       \begin{equation}
              \Sigma^{pole  }= \cup_{j=1}^N  \left\{ z\in \mathbb{C}:|z-z_j|=\rho \;or\; |z-\bar{z}_j|=\rho    \right\}.
       \end{equation}

To  convert  the  residues at   $z_k$ into the corresponding  jumps  on   circles
 such that they  further decay on  the new  jumps,   we introduce an interpolation function
  as follows.\\
   For $\xi>1$,
      \begin{equation}
           G(z) = \begin{cases}
            \left(\begin{array}{cc} 1&0 \\ -\displaystyle {\frac{c_j e^{2it\theta(z_j)}}{z-z_j}}&1\end{array}  \right),\; |z-z_j|<\rho, \\[4pt]
            \left(\begin{array}{cc} 1& - \displaystyle {\frac{\bar{c}_j e^{-2it\theta(\bar{z}_j) }}{z-\bar{z}_j} }\\ 0&1\end{array}  \right),\; |z-\bar{z}_j|<\rho,\\
            \left(\begin{array}{cc} 1&0 \\ 0&1\end{array}  \right),\; \text{as}\; z\; \text{in elsewhere};
           \end{cases}
       \end{equation}
       for $\xi<-1$,
       \begin{equation}
        G(z) = \begin{cases}
         \left(\begin{array}{cc} 1& -\displaystyle {\frac{z-z_j}{c_j e^{2it\theta(z_j)}}} \\ 0&1\end{array}  \right),\; |z-z_j|<\rho, \\[4pt]
         \left(\begin{array}{cc} 1&0\\ -\displaystyle {\frac{z-\bar{z}_j}{\bar{c}_j e^{-2it\theta(\bar{z}_j) }}} &1\end{array}  \right),\; |z-\bar{z}_j|<\rho,\\
         \left(\begin{array}{cc} 1&0 \\ 0&1\end{array}  \right),\; \text{as}\; z\; \text{in elsewhere},
        \end{cases}
    \end{equation}
        where $z_j \in \mathcal{Z}^+$ and the corresponding $\bar{z}_j \in \mathcal{Z}^-$.
       Then we make the following transformation which can renormalize the RH problem $M(z)$
       \begin{equation}\label{mtom1}
        M^{(1)}(z)=T(\infty)^{-\sigma_3} M(z) G(z)T(z)^{\sigma_3},
       \end{equation}
       which satisfies the following RH problem.

       \noindent\textbf{RHP2}.  Find a matrix-valued function $M^{(1)}(z)=M^{(1)}(z;x,t)$  satisfying
       \begin{itemize}
        \item Analyticity: $M^{(1)}(z)$ is analytic  in $\mathbb{C} \backslash \Sigma^{(1)}$  where $\Sigma^{(1)} = \mathbb{R} \cup \Sigma^{pole}$.
        \item Symmetry: $M^{(1)}(z)=\sigma_1 \overline{M^{(1)}(\bar{z})}\sigma_1 =z^{-1}M^{(1)}(z^{-1})\sigma_1$.
        \item Jump condition:  $  M^{(1)}_+(z)=M^{(1)}_-(z)V^{(1)}(z),$
        where for $z \in \mathbb{R}$
        \begin{equation}
            V^{(1)}(z)=        \begin{cases}
                \left(\begin{array}{cc} 1&-\bar{r} T^{-2}e^{-2it\theta }\\ 0&1\end{array}  \right)
                \left(\begin{array}{cc} 1&0\\   r T^2 e^{2it\theta }&1\end{array}  \right), \quad z\in \mathbb{R} \backslash I(\xi),\\
                \left(\begin{array}{cc} 1 & 0 \\ \frac{r}{1-|r|^2} T_-^{2} e^{2it\theta } & 1\end{array}  \right)
                \left(\begin{array}{cc} 1 & \frac{-\bar{r}}{1-|r|^2}T_+^{-2}e^{-2it\theta } \\ 0 & 1\end{array}  \right), \quad  z\in I(\xi),
            \end{cases}
        \end{equation}
       while jump matrices  on  $z\in \Sigma^{pole}$ are  given as follows.\\
       For $   \xi >1,$
        \begin{align}
           & V^{(1)}(z) = \begin{cases}
                \left(\begin{array}{cc} 1&0 \\ -\displaystyle {\frac{c_j T^{2}(z)e^{2it\theta(z_j)}}{z-z_j}}&1\end{array}  \right),\; |z-z_j|<\rho, \\[4pt]
                \left(\begin{array}{cc} 1& \displaystyle {\frac{\bar{c}_j T^{-2}(z)e^{-2it\theta(\bar{z}_j) }}{z-\bar{z}_j}} \\ 0&1\end{array}  \right),\; |z-\bar{z}_j|<\rho.
               \end{cases}
               \end{align}
        For $  \xi <-1,$
           \begin{align}
           & V^{(1)}(z) = \begin{cases}
                \left(\begin{array}{cc} 1& -\displaystyle {\frac{z-z_j}{c_j T^{2}(z)e^{2it\theta(z_j)}}}\\ 0&1\end{array}  \right),\; |z-z_j|<\rho, \\[4pt]
                \left(\begin{array}{cc} 1&0\\ \displaystyle {\frac{z-\bar{z}_j}{\bar{c}_j T^{-2}(z) e^{-2it\theta(\bar{z}_j) }}} &1\end{array}  \right),\; |z-\bar{z}_j|<\rho.
               \end{cases}
        \end{align}
        \item Asymptotic behaviors:
        \begin{align*}
                &M^{(1)}(z)=I+\mathcal{O}(z^{-1}),	\quad  z \to  \infty,\\
                &zM^{(1)}(z)=\sigma_1+\mathcal{O}(z), \quad z \to 0.
        \end{align*}
%        \item Residue conditions: $M^{(1)}(z)$  has simple poles at each points $z_n$ in $\mathcal{Z}^+ \cup \mathcal{Z}^-$ with the following residue conditions:
%        \begin{align}
%            & \res_{z=z_n} M^{(1)}(z)=\lim_{z\to z_n} M^{(1)}(z) \left(\begin{array}{cc}
%                0 & 0\\
%               c_n T^2(z_n)e^{2it\theta(z_n)} & 0
%            \end{array} \right),\\
%          & \res_{z=\bar{z}_n} M^{(1)}(z)=\lim_{z\to \bar{z}_n} M^{(1)}(z) \left( \begin{array}{cc}
%           0 & \bar{c}_n \overline{T}^2(z_n)e^{-2it\theta(\bar{z}_n)}\\
%           0 & 0
%             \end{array} \right).
%        \end{align}

    \end{itemize}

    \section{Transition to  a hybrid $\bar{\partial}$-RH problem}

	   \subsection{Opening the $\bar{\partial}$-lenses}

       Fix a sufficiently small angle $0<\phi(\xi) <\arctan \left( \frac{|\xi_2(\xi)|}{\min | \im z_j|} \right), j =1,\cdots,N $,   such that all regions which are touched by opening the jump contour $\mathbb{R}$   do not intersect  any  pole point $z_j$.
   See Figure \ref{fig5}.

   \begin{figure}
        \begin{center}
         \subfigure[$ \xi>1$]{
            \begin{tikzpicture}[scale=0.9]
                %[node distance=2cm]
    %\draw [yellow!10, fill=yellow!10] (5.6,0)--(5.6,2.6)--(3.3,2.6)--(3.7,0);
    % \draw [yellow!10, fill=yellow!10] (-4,0)--(-4,-2.6)--(3.3,-2.6)--(3.7,0);
\draw [dashed  ] (3.3,2.6) to [out=-75,in=90] (3.75,0);
\draw [dashed  ] (3.3,-2.6) to [out=75,in=-95] (3.75,0);
   %\filldraw[white](0,0)--(1.2,0) arc (0:-180:0.6);
                \draw[dashed,->](-4.2,0)--(6,0)node[black,right]{Re$z$};
                \foreach \x [count=\p] in {1,...,5} {
            \node[shape=circle,fill=red, scale=0.15] (\p) at (\x*30:2) {};
            \draw[red,line width=0.5] (\x*30:2) circle (0.2);
            \draw[red, ->] (\x*30:2)+(0.2,0) arc (0:90:0.2);
            \node[shape=circle,fill=blue, scale=0.15] (\p) at (180+\x*30:2) {};
            \draw[blue,line width=0.5] (180+\x*30:2) circle (0.2);
            \draw[blue, ->] (180+\x*30:2)+(-0.2,0) arc (180:90:0.2);
            };
         %   \draw[red,line width = 0.5](-1.8,0) arc(0:180:0.2);
         %   \draw[red,->] (-1.8,0) arc(0:90:0.2);
         %   \node[shape=circle,fill=black, scale=0.15] at (-2,0){};
        %    \draw[red,line width = 0.5](2.2,0) arc(0:180:0.2);
         %   \draw[red,->] (2.2,0) arc(0:90:0.2);
         %  \node[shape=circle,fill=black, scale=0.15] at (2,0){};

         %   \draw[blue,line width = 0.5](-2.15,0) arc(180:360:0.15);
        %    \draw[blue,line width = 0.5](1.85,0) arc(180:360:0.15);

           % \draw[red,line width=0.2] (\p) arc (0:360:0.2)
            %\foreach \x [count=\p] in {0,...,5} {
           %         \draw (-\x*60:2.4);
            %        \draw (-30-\x*60:2.4) ;};
              %  \node[shape=circle,fill=blue, scale=0.15]  at (0:2){0};
              %  \node[shape=circle,fill=blue, scale=0.15]  at (-6*30:2){0} ;
              %  \node[shape=circle,fill=blue,scale=0.15] at (-0.05,0) {0};
                %\node[below] at (2.15,0.05) {\footnotesize $1$};
                \node[below] at (-0.1,0) {\footnotesize $0$};
                \node[below] at (-2.3,0) {\footnotesize $-1$};
                \node[right] at (1.2,1.8) {$z_j$};
            %    \draw[red,line width=0.2] (1,1.8) arc (0:360:0.2);
                \node[right] at (1.2,-1.8) {$\bar{z}_j$};
            %	\node[left] at (5*30:2) {$\bar{z}_k$};
            %	\node[left] at (-5*30:2) {$z_k$};
              \draw [dashed, gray] (0,0) circle (2);
            %    \draw [dashed, gray](1) arc (0:360:2);
               \draw [dashed] (0.8,0) circle [radius=0.8];
                                 \draw [ blue] (-4, -1.5)--(0.8,0.35);
                                 \draw [ blue] (0.8,0.35 )--(2.7,-0.5);
                                 \draw [-latex,blue] (0.8,0.35 )--(2.2,-0.265);
                           \draw [ blue] (2.7,-0.5)--(5.5, 0.75);
                           \draw [-latex, blue] (2.7,-0.5)--(4.6, 0.35);
   \draw [-latex,blue] (-4, -1.5)--(-1.25,-0.44 );
   \draw[](0.8,-0.35)--(0.8,0.35);
   \draw[-latex](0.8,0)--(0.8,0.16);
   \draw[-latex](0.8,-0.35)--(0.8,-0.1);
   \draw[](2.7,-0.5)--(2.7,0.5);
   \draw[-latex](2.7,0.5)--(2.7,0.2);
   \draw[-latex](2.7,0)--(2.7,-0.2);
                             \draw [ red] (0.8,-0.35)--(2.7, 0.5);
                                \draw [-latex,red] (0.8,-0.35 )--(2.2,0.265);
                                \draw [ red] (2.7,0.5)--(5.5, -0.75);
                                \draw [ red] (-4, 1.5)--(0.8,-0.35 );
                                 \draw [-latex,red] (2.7, 0.5)--(4.6, -0.35);
                                       \draw [-latex,red] (-4,  1.5)--(-1.25, 0.44 );
                                  \node[blue,thick]  at (4.5,1.2) {\bf  $+$};
                                     \node[red,thick]  at (4.5,-1.2) {\bf  $-$};
                                       \node[red,thick]  at (2.6,1.2) {\bf  $-$};
                                     \node[blue,thick]  at (2.6,-1.2) {\bf  $+$};
                                     \node[red,thick]  at (0,1.2) {\small\bf $-$};
                                     \node[blue,thick]  at (0,-1.2) {\small\bf $+$};
                                      \node[blue,thick]  at (0.8,0.5) {\tiny\bf $+$};
                                         \node[red,thick]  at (0.8,-0.5) {\tiny\bf $-$};
                                         \node[blue,thick]  at (-3,-1.3) {\bf  $+$};
                                         \node[red,thick]  at (-3,1.3) {\bf  $-$};

                                                 \node[below] at (3.9,0) {\footnotesize $\xi_2$};
                                         \node[below] at (1.73,0) {\footnotesize $\xi_1$};
                                         \node  at (5.6,1) {\footnotesize $\Sigma_{21}$};
                                         \node  at (5.6,-1) {\footnotesize $\Sigma_{24}$};
                                          \node  at (3.3,0.6) {\footnotesize $\Sigma_{22}$};
                                         \node  at (3.3,-0.6) {\footnotesize $\Sigma_{23}$};
                                          \node  at (2.2,0.6) {\footnotesize $\Sigma_{12}$};
                                         \node  at (2.3,-0.6) {\footnotesize $\Sigma_{13}$};
                                            \node  at (-3.6,1) {\footnotesize $\Sigma_{02}$};
                                         \node  at (-3.6,-1) {\footnotesize $\Sigma_{03}$};
                                           \node  at (1.3,0.3) {\tiny $\Sigma_{11}$};
                                             \node  at (0.35,0.3) {\tiny $\Sigma_{01}$};
                                             \node  at (1.23,-0.4) {\tiny $\Sigma_{14}$};
                                             \node  at (0.35,-0.4) {\tiny $\Sigma_{04}$};
            \end{tikzpicture}
         }%

         \subfigure[$ \xi<-1$]{
            \begin{tikzpicture}[scale=0.9]
                %[node distance=2cm]
    %\draw [yellow!10, fill=yellow!10] (5.6,0)--(5.6,2.6)--(3.3,2.6)--(3.7,0);
    % \draw [yellow!10, fill=yellow!10] (-4,0)--(-4,-2.6)--(3.3,-2.6)--(3.7,0);
\draw [dashed  ] (-3.3,2.6) to [out=75,in=90] (-3.8,0);
\draw [dashed  ] (-3.3,-2.6) to [out=-75,in=-95] (-3.8,0);
   %\filldraw[white](0,0)--(1.2,0) arc (0:-180:0.6);
                \draw[dashed,->](-6,0)--(4.5,0)node[black,right]{Re$z$};
                \foreach \x [count=\p] in {1,...,5} {
                    \node[shape=circle,fill=blue, scale=0.15] (\p) at (\x*30:2) {};
                    \draw[blue,line width=0.5] (\x*30:2) circle (0.2);
                    \draw[blue, ->] (\x*30:2)+(-0.2,0) arc (180:90:0.2);
                    \node[shape=circle,fill=red, scale=0.15] (\p) at (180+\x*30:2) {};
                    \draw[red,line width=0.5] (180+\x*30:2) circle (0.2);
                    \draw[red, ->] (180+\x*30:2)+(0.2,0) arc (0:90:0.2);
                    };

            %    \foreach \x [count=\p] in {0,...,11} {
            %\node[shape=circle,fill=red, scale=0.25] (\p) at (-\x*30:2) {};};
             %   \foreach \x [count=\p] in {0,...,5} {
             %       \draw (-\x*60:2.4);
             %       \draw (-30-\x*60:2.4) ;};
            %    \node[shape=circle,fill=blue, scale=0.15]  at (0:2){0};
            %    \node[shape=circle,fill=blue, scale=0.15]  at (-6*30:2){0} ;
            %\node[shape=circle,fill=blue,scale=0.15] at (0.05,0) {0};
                \node[below] at (2.1,0.05) {\footnotesize $1$};
                \node[below] at (0.1,0) {\footnotesize $0$};
                % \node[below] at (-2.3,0) {\footnotesize $-1$};
                \node[right] at (1.2,1.8) {$z_j$};
                \node[right] at (1.2,-1.8) {$\bar{z}_j$};
            %	\node[left] at (5*30:2) {$\bar{z}_k$};
            %	\node[left] at (-5*30:2) {$z_k$};
               \draw [dashed, gray] (0,0) circle (2);
               \draw [dashed] (-0.8,0) circle [radius=0.8];

                                 \draw [ red] (4, -1.5)--(-0.8,0.35);
                                 \draw [ red] (-0.8,0.35 )--(-2.7,-0.5);
                                 \draw [-latex,red] (-2.7,-0.5)--(-2,-0.175);
                           \draw [ red] (-2.7,-0.5)--(-5.5, 0.75);
                           \draw [-latex, red] (-5.5, 0.75)--(-4.6, 0.35);
   \draw [-latex,red] (-0.8,0.35 )--(1.25,-0.44 );
   \draw[](-0.8,-0.35)--(-0.8,0.35);
   \draw[-latex](-0.8,0)--(-0.8,0.16);
   \draw[-latex](-0.8,-0.35)--(0-.8,-0.1);
   \draw[](-2.7,-0.5)--(-2.7,0.5);
   \draw[-latex](-2.7,0.5)--(-2.7,0.2);
   \draw[-latex](-2.7,0)--(-2.7,-0.2);
                             \draw [ blue] (-0.8,-0.35)--(-2.7, 0.5);
                                \draw [-latex,blue] (-2.7, 0.5)--(-2,0.175);
                                \draw [blue] (-2.7,0.5)--(-5.5, -0.75);
                                \draw [blue] (4, 1.5)--(-0.8,-0.35 );
                                 \draw [-latex,blue] (-5.5, -0.75)--(-4.6, -0.35);
                                   \draw [-latex,blue] (-0.8,-0.35 )--(1.25, 0.44 );

                                  \node[red,thick]  at (-4.5,1.2) {\bf  $-$};
                                     \node[blue,thick]  at (-4.5,-1.2) {\bf  $+$};
                                       \node[blue,thick]  at (2.6,1.2) {\bf  $+$};
                                     \node[red,thick]  at (2.6,-1.2) {\bf  $-$};
                                     \node[blue,thick]  at (0,1.2) {\small\bf $+$};
                                     \node[red,thick]  at (0,-1.2) {\small\bf $-$};
                                      \node[red,thick]  at (-0.8,0.5) {\tiny\bf $-$};
                                         \node[blue,thick]  at (-0.8,-0.5) {\tiny\bf $+$};
                                         \node[red,thick]  at (-3,-1.2) {\bf  $-$};
                                         \node[blue,thick]  at (-3,1.2) {\bf  $+$};

                                                 \node[below] at (-3.9,0) {\footnotesize $\xi_2$};
                                         \node[below] at (-1.73,0) {\footnotesize $\xi_1$};
                                         \node  at (-5.6,1) {\footnotesize $\Sigma_{22}$};
                                         \node  at (-5.6,-1) {\footnotesize $\Sigma_{23}$};
                                          \node  at (-3.2,0.6) {\footnotesize $\Sigma_{21}$};
                                         \node  at (-3.2,-0.6) {\footnotesize $\Sigma_{24}$};
                                          \node  at (-2.3,0.6) {\footnotesize $\Sigma_{11}$};
                                         \node  at (-2.3,-0.65) {\footnotesize $\Sigma_{14}$};
                                            \node  at (3.6,1) {\footnotesize $\Sigma_{01}$};
                                         \node  at (3.6,-1) {\footnotesize $\Sigma_{04}$};
                                           \node  at (-1.3,0.3) {\tiny $\Sigma_{12}$};
                                             \node  at (-0.35,0.3) {\tiny $\Sigma_{02}$};
                                             \node  at (-1.23,-0.4) {\tiny $\Sigma_{13}$};
                                             \node  at (-0.35,-0.4) {\tiny $\Sigma_{03}$};
            \end{tikzpicture}
         }%
        \end{center}
        \caption{ \small{ The jump contours and the signal of  $\re \left(2i\theta \right)$ in different regions for the two  cases $ \xi>1$ and $ \xi<-1$:
        Open  the jump contour $\mathbb{R}$ in a sufficient small fixed angle, such that  the regions enclosed between the red and blue straight lines which contains the real lines does not touch any part of the circles of the poles. The signal $+$ means the real part $\re \left(2i\theta\right) >0$ in these regions, and  $-$ stands for $\re \left(2i\theta \right)<0$   in the  corresponding regions. }}
        \label{fig5}
    \end{figure}
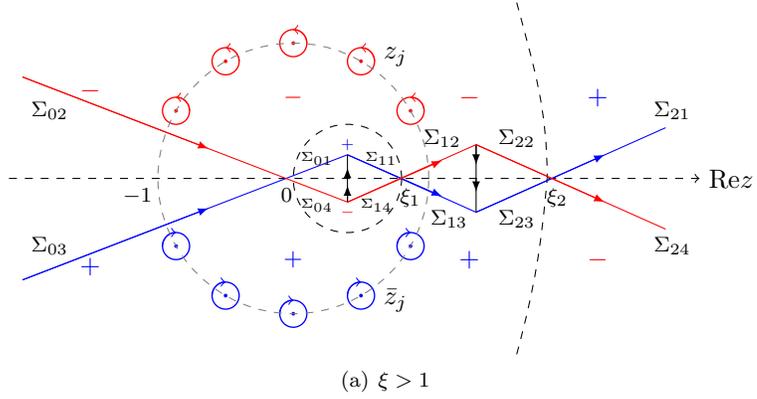
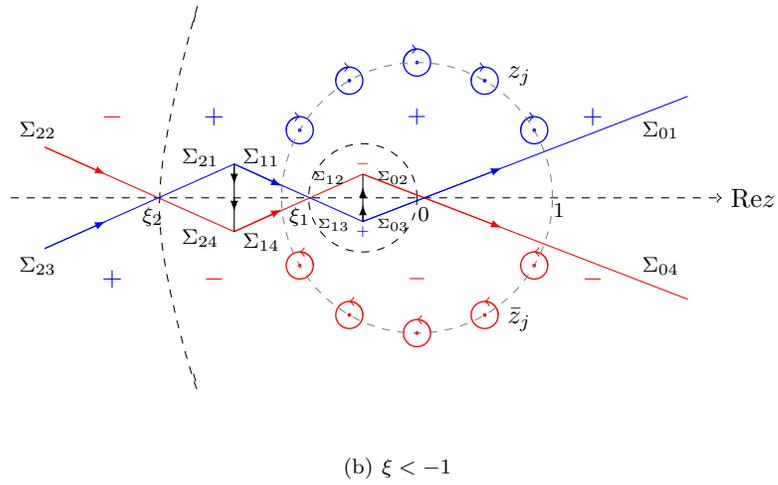

      Denote $\xi_{0}=0$, $\xi_{0,1}=(\xi_{0}+\xi_{1})/2$, $\xi_{1,2}=(\xi_{1}+\xi_{2})/2$, and real intervals
    \begin{align*}
   & l_1\in \left(0,  |\xi_{0,1}|\sec \phi(\xi)\right), \ \  l_2\in \left(0,  |\xi_{1,2}-\xi_1|\sec \phi(\xi)\right), \\
    & \tilde{l}_1\in \left(0,  |\xi_{0,1}|\tan \phi(\xi)\right), \ \  \tilde{l}_2\in \left(0,  |\xi_{1,2}-\xi_1|\tan \phi(\xi)\right).
      \end{align*}
   Then we define the  boundaries $\Sigma_{kj}, \ k=0,1,2, \ j=1,2,3,4$,  produced   after opening the real axis  $\mathbb{R}$ at  $0, \xi_1, \xi_2$:
 \begin{align*}
     & \Sigma_{k1}=\xi_k+e^{i\phi(\xi)} \mathbb{R}^+,\quad \Sigma_{k2}=\xi_k+e^{i\left(\pi - \phi(\xi)\right)} l_2,\quad \Sigma_{k3}=\overline{\Sigma}_{k2},\quad \Sigma_{k4}=\overline{\Sigma}_{k1}, \ k=0, 2,\\
      &\Sigma_{11}=\xi_1+e^{i\left(\pi - \phi(\xi)\right)} l_1,\quad
     \Sigma_{12}=\xi_1+e^{i\phi(\xi)} l_2,\;  \quad
     \Sigma_{13}=\overline{\Sigma}_{12},\quad \, \Sigma_{14}=\overline{\Sigma}_{11},\\
     &\Sigma'=\cup_{j=1}^4\Sigma'_j, \ \  \Sigma'_1=\xi_{0,1}+ e^{i\pi/2} \tilde{l}_1, \quad \Sigma'_2=\overline{\Sigma}'_1,\quad
   	\Sigma'_3=\xi_{1,2}+ e^{i\pi/2} \tilde{l}_2, \quad \Sigma'_4=\overline{\Sigma}'_3.
     \end{align*}
     Denote by $\Omega_{kj}, k=0,1,2, j=1,2,3,4 $, the twelve  sectors enclosed between the boundaries $\Sigma_{kj}$ and the real line as shown in Figure \ref{rj1}. Denote the  cone  $\Omega=\cup_{k=0}^{2}\cup_{j=1}^{4}\Omega_{kj}$, and the intervals
   \begin{align*}	
   	  	& I_1=\begin{cases}
        (-\infty,0), \text{for}\; \xi>1,\\
        (-\infty,\xi_2), \text{for}\; \xi<-1,
     \end{cases}\;    I_2=\begin{cases}
        (0,\xi_1), \text{for}\; \xi>1,\\
        (\xi_2,\xi_1), \text{for}\; \xi<-1,
     \end{cases}\\
    & I_3=\begin{cases}
        (\xi_1,\xi_2), \text{for}\; \xi>1,\\
        (\xi_1,0), \text{for} \;\xi<-1,
     \end{cases}\quad \;
     I_4=\begin{cases}
        (\xi_2, +\infty), \text{for}\; \xi>1,\\
        (0,+\infty), \text{for}\; \xi<-1.
     \end{cases}
    \end{align*}

\begin{figure}
    \begin{center}
     \subfigure[$ \xi>1$]{
         \begin{tikzpicture}

            \draw[Blue!10,fill=LightSteelBlue!] (2,0)--(3,1)--(3,0);
\draw[Blue!10,fill=LightSteelBlue!] (2,0)--(1,-1)--(1,0);
\draw[CadetBlue!20,fill=LightSteelBlue!20] (0,0)--(1,1)--(1,0);	
\draw[CadetBlue!20,fill=LightSteelBlue!20] (0,0)--(-1,-1)--(-1,0);
\draw[Blue!10,fill=LightSteelBlue!] (-2,0)--(-1,1)--(-1,0);	
\draw[Blue!10,fill=LightSteelBlue!] (-2,0)--(-3,-1)--(-3,0);

\draw[CadetBlue!20,fill=LightSteelBlue!20] (2,0)--(1,1)--(1,0);	\draw[CadetBlue!20,fill=LightSteelBlue!20] (2,0)--(3,-1)--(3,0);
\draw[Blue!10,fill=LightSteelBlue!] (0,0)--(1,-1)--(1,0);	\draw[Blue!10,fill=LightSteelBlue!] (0,0)--(-1,1)--(-1,0);
\draw[CadetBlue!20,fill=LightSteelBlue!20] (-2,0)--(-1,-1)--(-1,0);	\draw[CadetBlue!20,fill=LightSteelBlue!20] (-2,0)--(-3,1)--(-3,0);
             \draw[](-3.4,0)--(3.4,0);
             \draw[](0,0)--(0.5,0.5);
             \draw[](0.5,0.5)--(1,1);
             \node[below, scale=0.6] at (0,0) {$\xi_1$};
             \node[below, scale=0.6]   at (-2,0) {$0$};
             \node[below, scale=0.6]  at (2,0) {$\xi_2$};
             \node[below, scale=0.6]  at (-1,0) {$\xi_{0,1}$};
             \node[below, scale=0.6]  at (1,0) {$\xi_{1,2}$};
             \draw[](-2,0)--(-1.5,0.5);
             \draw[](-1.5,0.5)--(-1,1);
             \draw[](-3,1)--(-2.5,0.5);
             \draw[](-2.5,0.5)--(-2,0);
             \draw[](-3,-1)--(-2.5,-0.5);
             \draw[](-2.5,-0.5)--(-2,0);
             \draw[](-2,0)--(-1.5,-0.5);
             \draw[](-1.5,-0.5)--(-1,-1);	
             \draw[](-1,-1)--(-0.5,-0.5);
             \draw[](-0.5,-0.5)--(0,0);
             \draw[](0,0)--(-0.5,0.5);
             \draw[](-1,1)--(-0.5,0.5);
             \draw[](0,0)--(0.5,-0.5);
             \draw[]	(0.5,-0.5)--(1,-1);
             \draw[](1.5,0.5)--(2,0);
             \draw[](1,1)--(1.5,0.5);
             \draw[](2,0)--(1.5,-0.5);
             \draw[](1,-1)--(1.5,-0.5);
             \draw[](2,0)--(2.5,0.5);
             \draw[](3,1)--(2.5,0.5);
             \draw[](2.5,-0.5)--(3,-1);
             \draw[](2,0)--(2.5,-0.5);
             \draw[](1,0)--(1,0.5);
             \draw[](1,1)--(1,0.4);
             \draw[](1,0)--(1,-0.5);
             \draw[](1,0.5)--(1,-1);

             \node[scale=0.6] at (2.8,0.38) {$\Omega_{21}$};
         %	\node[scale=0.5] at (2.38,0.78) {$\Sigma_{21}$};
             \node[scale=0.6] at (2.8,-0.38) {$\Omega_{24}$};
         %	\node[scale=0.7] at (2.38,-0.78) {$\Sigma_{24}$};
             \node[scale=0.6] at (1.3,0.3) {$\Omega_{22}$};
         %	\node[scale=0.7] at (1.6,0.78) {$\Sigma_{22}$};
             \node[scale=0.6] at (1.3,-0.32) {$\Omega_{23}$};
         %	\node[scale=0.7] at (1.6,-0.78) {$\Sigma_{23}$};
             \node[scale=0.6] at (0.72,0.3) {$\Omega_{12}$};
         %	\node[scale=0.7] at (0.45,0.8) {$\Sigma_{12}$};
             \node[scale=0.6] at (0.72,-0.32) {$\Omega_{13}$};
         %	\node[scale=0.7] at (0.45,-0.8) {$\Sigma_{13}$};
             \node[scale=0.6] at (-0.7,0.3) {$\Omega_{11}$};
         %	\node[scale=0.7] at (-0.45,0.8) {$\Sigma_{11}$};
             \node[scale=0.6] at (-0.68,-0.32) {$\Omega_{14}$};
         %	\node[scale=0.7] at (-0.4,-0.8) {$\Sigma_{14}$};
             \node[scale=0.6] at (-1.3,0.3) {$\Omega_{01}$};
         %	\node[scale=0.7] at (-1.6,0.8) {$\Sigma_{01}$};
             \node[scale=0.6] at (-1.4,-0.3) {$\Omega_{04}$};
         %	\node[scale=0.7] at (-1.6,-0.8) {$\Sigma_{04}$};
             \node[scale=0.6] at (-2.8,0.3) {$\Omega_{02}$};
         %	\node[scale=0.7] at (-3.1,0.8) {$\Sigma_{02}$};
             \node[scale=0.6] at (-2.8,-0.25) {$\Omega_{03}$};
         %	\node[scale=0.7] at (-3.2,-0.8) {$\Sigma_{03}$};
             \draw[](-1,0)--(-1,0.5);
             \draw[](-1,0.5)--(-1,1);
             \draw[](-1,-1)--(-1,-0.4);
             \draw[](-1,0)--(-1,-0.5);
              \draw[ ->](-4,0)--(4,0)node[black,right]{Re$z$};
         \end{tikzpicture}
     }%

     \subfigure[$ \xi<-1$]{
         \begin{tikzpicture}
            \draw[Blue!10,fill=LightSteelBlue!] (2,0)--(3,1)--(3,0);
\draw[Blue!10,fill=LightSteelBlue!] (2,0)--(1,-1)--(1,0);
\draw[CadetBlue!20,fill=LightSteelBlue!20] (0,0)--(1,1)--(1,0);	
\draw[CadetBlue!20,fill=LightSteelBlue!20] (0,0)--(-1,-1)--(-1,0);
\draw[Blue!10,fill=LightSteelBlue!] (-2,0)--(-1,1)--(-1,0);	
\draw[Blue!10,fill=LightSteelBlue!] (-2,0)--(-3,-1)--(-3,0);

\draw[CadetBlue!20,fill=LightSteelBlue!20] (2,0)--(1,1)--(1,0);	\draw[CadetBlue!20,fill=LightSteelBlue!20] (2,0)--(3,-1)--(3,0);
\draw[Blue!10,fill=LightSteelBlue!] (0,0)--(1,-1)--(1,0);	\draw[Blue!10,fill=LightSteelBlue!] (0,0)--(-1,1)--(-1,0);
\draw[CadetBlue!20,fill=LightSteelBlue!20] (-2,0)--(-1,-1)--(-1,0);	\draw[CadetBlue!20,fill=LightSteelBlue!20] (-2,0)--(-3,1)--(-3,0);	\draw[CadetBlue!20,fill=LightSteelBlue!20] (-2,0)--(-3,1)--(-3,0);
         \draw[](-3.4,0)--(3.4,0);
         \draw[](0,0)--(0.5,0.5);
         \draw[](0.5,0.5)--(1,1);
         \node[below, scale=0.6] at (0,0) {$\xi_1$};
         \node[below, scale=0.6]   at (-2,0) {$\xi_2$};
         \node[below, scale=0.6]  at (2,0) {$0$};
         \node[below, scale=0.6]  at (-1,0) {$\xi_{0,1}$};
         \node[below, scale=0.6]  at (1,0) {$\xi_{1,2}$};
         \draw[](-2,0)--(-1.5,0.5);
         \draw[](-1.5,0.5)--(-1,1);
         \draw[](-3,1)--(-2.5,0.5);
         \draw[](-2.5,0.5)--(-2,0);
         \draw[](-3,-1)--(-2.5,-0.5);
         \draw[](-2.5,-0.5)--(-2,0);
         \draw[](-2,0)--(-1.5,-0.5);
         \draw[](-1.5,-0.5)--(-1,-1);	
         \draw[](-1,-1)--(-0.5,-0.5);
         \draw[](-0.5,-0.5)--(0,0);
         \draw[](0,0)--(-0.5,0.5);
         \draw[](-1,1)--(-0.5,0.5);
         \draw[](0,0)--(0.5,-0.5);
         \draw[]	(0.5,-0.5)--(1,-1);
         \draw[](1.5,0.5)--(2,0);
         \draw[](1,1)--(1.5,0.5);
         \draw[](2,0)--(1.5,-0.5);
         \draw[](1,-1)--(1.5,-0.5);
         \draw[](2,0)--(2.5,0.5);
         \draw[](3,1)--(2.5,0.5);
         \draw[](2.5,-0.5)--(3,-1);
         \draw[](2,0)--(2.5,-0.5);
         \draw[](1,0)--(1,0.5);
         \draw[](1,1)--(1,0.4);
         \draw[](1,0)--(1,-0.5);
         \draw[](1,0.5)--(1,-1);

         \node[scale=0.6] at (2.8,0.38) {$\Omega_{01}$};
         %	\node[scale=0.5] at (2.38,0.78) {$\Sigma_{21}$};
         \node[scale=0.6] at (2.8,-0.38) {$\Omega_{04}$};
         %	\node[scale=0.7] at (2.38,-0.78) {$\Sigma_{24}$};
         \node[scale=0.6] at (1.3,0.3) {$\Omega_{02}$};
         %	\node[scale=0.7] at (1.6,0.78) {$\Sigma_{22}$};
         \node[scale=0.6] at (1.3,-0.32) {$\Omega_{03}$};
         %	\node[scale=0.7] at (1.6,-0.78) {$\Sigma_{23}$};
         \node[scale=0.6] at (0.72,0.3) {$\Omega_{12}$};
         %	\node[scale=0.7] at (0.45,0.8) {$\Sigma_{12}$};
         \node[scale=0.6] at (0.72,-0.32) {$\Omega_{13}$};
         %	\node[scale=0.7] at (0.45,-0.8) {$\Sigma_{13}$};
         \node[scale=0.6] at (-0.7,0.3) {$\Omega_{11}$};
         %	\node[scale=0.7] at (-0.45,0.8) {$\Sigma_{11}$};
         \node[scale=0.6] at (-0.68,-0.32) {$\Omega_{14}$};
         %	\node[scale=0.7] at (-0.4,-0.8) {$\Sigma_{14}$};
         \node[scale=0.6] at (-1.3,0.3) {$\Omega_{21}$};
         %	\node[scale=0.7] at (-1.6,0.8) {$\Sigma_{01}$};
         \node[scale=0.6] at (-1.4,-0.3) {$\Omega_{24}$};
         %	\node[scale=0.7] at (-1.6,-0.8) {$\Sigma_{04}$};
         \node[scale=0.6] at (-2.8,0.3) {$\Omega_{22}$};
         %	\node[scale=0.7] at (-3.1,0.8) {$\Sigma_{02}$};
         \node[scale=0.6] at (-2.8,-0.25) {$\Omega_{23}$};
         %	\node[scale=0.7] at (-3.2,-0.8) {$\Sigma_{03}$};
         \draw[](-1,0)--(-1,0.5);
         \draw[](-1,0.5)--(-1,1);
         \draw[](-1,-1)--(-1,-0.4);
         \draw[](-1,0)--(-1,-0.5);
          \draw[](-1,0)--(-1,-0.5);
           \draw[ ->](-4,0)--(4,0)node[black,right]{Re$z$};
     \end{tikzpicture}
     }%
    \end{center}
    \caption{ \small The continuous extension regions.   ${\rm Re} (2i\theta)>0$   in the blue regions and ${\rm Re} (2i\theta)<0$  in the gray regions.}
    \label{rj1}
    \end{figure}
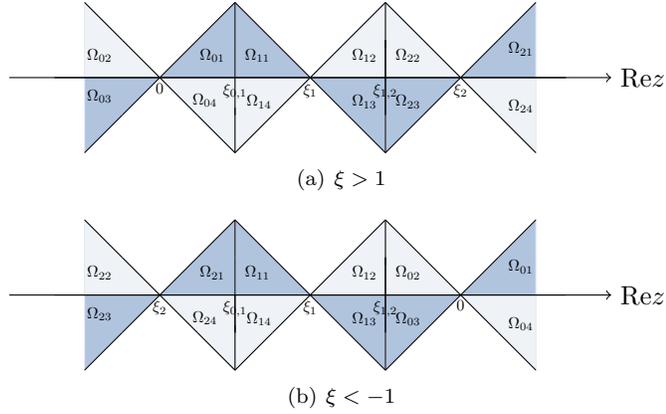

\begin{proposition} \label{prop2}  For $ |\xi|>1, \ \xi=\mathcal{O}(1)$, we have for $z = \xi_k + le^{i w}:=\xi_k + u +i v, k=0,1,2$,
    \begin{align}
        &\re(2i\theta(z)) \ge c(\xi,\xi_k) v >0,\quad z\in \Omega_{k1} \cup \Omega_{k3},\\
       &\re(2i\theta(z)) \le -c(\xi,\xi_k) v<0,\quad z\in \Omega_{k2} \cup \Omega_{k4},
    \end{align}
  where $c(\xi,\xi_k)$ is a constant.
    %\begin{align}
    %    \re(2i\theta(z))>0,\quad z\in \Omega_{k1} \cup \Omega_{k3},\\
    %    \re(2i\theta(z))<0,\quad z\in \Omega_{k2} \cup \Omega_{k4},
    %\end{align}
   % where $k=0,1,2.$
   \end{proposition}
   \begin{proof}
    We take the case for $\Omega_{01}$ and $\Omega_{11}$ of $ \xi>1,   \ \xi=\mathcal{O}(1)$ as an example to prove the above proposition. Other cases can be proven in the similar method.
    For $z \in \Omega_{01}$, assume that $z=|z|e^{iw}$,
    \begin{align}
        \re\left(2it\theta(z)\right) &=  \sin(2w) \left( |z|+|z|^{-1}\right)^2 - \xi \sin(2w) \left( |z|+|z|^{-1}\right) \sec(w)-2\sin(2w)\nonumber \label{F11}\\
         &=  G(|z|) \sin(2w),\nonumber
    \end{align}
    where $G(|z|)=F(|z|)^2 - \xi \sec(w)F(|z|)-2$ and $F(z)=z+z^{-1}$.
   The two zeros of $G(|z|)$ are
   \begin{equation}
        F_1(|z|)= \frac{\xi \sec(w)- \sqrt{\xi^2 \sec(w)^2+8} }{2}, \; F_2(|z|)= \frac{\xi \sec(w)+ \sqrt{\xi^2 \sec(w)^2+8} }{2}.\nonumber
   \end{equation}
    Since $F(|z|) \ge 2$, we only consider the zero $F_2(|z|)$ and define the corresponding $|z|=F_2^{-1}\left( \frac{\xi \sec(w)+ \sqrt{\xi^2 \sec(w)^2+8} }{2}\right)$.
    A simple calculation shows  that  as the angle $w< \pi/4$, then  we have
    $ \xi \sec(w)+ \sqrt{\xi^2 \sec(w)^2+8}>4$
   and $G(z)\ge G\left(F_2^{-1}\left( \frac{\xi \sec(w)+ \sqrt{\xi^2 \sec(w)^2+8} }{2}\right)\right)> 0$.
  % In addition, we can  verify that if $w \le \frac{\pi}{4}$, $G(z)>0$ is satisfied.
  % In our case here, $w \le \phi(\xi) <\frac{\pi}{4}$.
  Thus, $ \re\left(2it\theta(z)\right)$ can be estimated by
  \begin{align*}
    \re\left(2it\theta(z)\right)  \ge G\left(F_2^{-1}\left( \frac{\xi \sec(w)+ \sqrt{\xi^2 \sec(w)^2+8} }{2}\right)\right)>0.
  \end{align*}

   For $z \in \Omega_{11}$, assume that $z = \xi_1 + l e^{iw} := \xi_1 +u+iv$,
   \begin{align}
    \re\left(2it\theta(z)\right) = v \left( \xi + \frac{\xi}{|z|^2}- u \left(1+ \frac{1}{|z|^4}\right)   \right) > v \left( \xi -1 + \frac{\xi}{|z|^2}-  \frac{1}{|z|^4}\right).
   \end{align}
   Let $\tau = |z|^2$ and
   \begin{align}
        h(\tau) = \xi-1+ \xi \tau^{-1}-  \tau^{-2}.
   \end{align}
   $h'(\tau) > 0$ as $z \in (\xi_{0,1},\xi_1)$, so we have $h(\tau)\ge h(|(\xi_1-\xi_{0,1})\sec(w)|^2)>0$. Finally, we obtain
   \begin{align*}
    \re\left(2it\theta(z)\right) \ge  \left( \xi - 1 + \xi \cos(w)^2 |\xi_1 - \xi_{1,0}|^{-2} - \cos(w)^4 |\xi_1 - \xi_{1,0}|^{-4} \right) v >0.
   \end{align*}
   %  \begin{equation}
%       G(z)> 0.
%   \end{equation}
%      For any $|\xi|>1$, the above equality is true, then we find $G(z)>0$.
   \end{proof}
%\begin{remark}
%    In fact, we can find that for $z = \xi_k + le^{i\varphi}:=\xi_k + u +i v, k=0,1,2$,
%    \begin{align}
%        &\re(2i\theta(z)) \ge P(v)>0,\quad z\in \Omega_{k1} \cup \Omega_{k3},\\
%       &\re(2i\theta(z)) \le -P(v)<0,\quad z\in \Omega_{k2} \cup \Omega_{k4}.
%    \end{align}
%    Here $P(v)=P(v,\xi_k)$ is a polynomial of $v$, which is $P(v)=A(\xi_k) v^2 + B(\xi_k) v + C(\xi_k)$ where $A(\xi_k),B(\xi_k), C(\xi_k)$ are constants and $A(\xi_k)$ can be equal to $0$.
%\end{remark}

	   \subsection{The hybrid $\bar{\partial}$-RH problem and its decompositions }
       We make continuous extensions of the jump matrix $V^{(1)}(z) $  to remove the jump from $\mathbb{R}$.

       \begin{proposition} \label{prop3}
       We define functions $R_{kj}: \bar{\Omega} \to \mathbb{C}, k=0,1,2, j=1,2,3,4$, which have the following boundary values:
       \begin{align}
        &R_{k1}= \begin{cases}
            \frac{\overline{r(z)}T_+(z)^{-2}}{1-|r(z)|^2}, \quad z\in I_2 \cup I_4,\\
           f_{k1}(z)= \frac{\overline{r(\xi_k)}T_{k}(\xi_k)^{-2}}{1-|r(\xi_k)|^2}\left(z-\xi_k\right)^{2iv(\xi_k) \epsilon_k},\quad z \in \Sigma_{k1},
           \end{cases}\\
        &R_{k2}= \begin{cases}
             r(z)T(z)^2, \quad z\in I_1 \cup I_3,\\
             f_{k2}(z)=r(\xi_k)T_{k}(\xi_k)^2 \left(z-\xi_k\right)^{-2iv(\xi_k)\epsilon_k},\quad z \in \Sigma_{k2},
            \end{cases}\\
        &R_{k3}= \begin{cases}
                \overline{r(z)}T(z)^{-2}, \quad z\in I_1 \cup I_3,\\
                f_{k3}(z)=\overline{r(\xi_k)}T_{k}(\xi_k)^{-2} \left(z-\xi_k\right)^{2iv(\xi_k) \epsilon_k},\quad z \in \Sigma_{k3},
               \end{cases}\\
        &R_{k4}= \begin{cases}
                \frac{r(z)T_-(z)^{2}}{1-|r(z)|^2}, \quad z\in I_2\cup I_4,\\
                f_{k4}(z)=\frac{r(\xi_k)T_{k}(\xi_k)^{2}}{1-|r(\xi_k)|^2}\left(z-\xi_k\right)^{-2iv(\xi_k) \epsilon_k},\quad z \in \Sigma_{k4},
               \end{cases}
       \end{align}
       where $r(\xi_0)=r(0)=0$. Then  there exists a constant $c_1 $ such that for $ \xi>1 $, %$i=0,1,2$ and $j=1,2,3,4$,
       \begin{equation}\label{DbarR}
	|\bar{\partial}R_{kj}| \le
c_1\left(|r'(\re(z))|+|z-\xi_k|^{-1/2}\right), \quad \text{for all $z\in \Omega_{kj}$};
\end{equation}
for $ \xi<-1$ with $k=0,2$,
      \begin{align}\label{dbarrij}
        |\bar{\partial}R_{kj}| \le
        \begin{cases}
            c_1\left(|\varphi_k(\re(z))|+|r'(\re(z))|+|z-\xi_k|^{-1/2}\right),\quad \text{for $z\in \Omega_{kj}, j=1,4$},\\
            c_1\left(|r'(\re(z))|+|z-\xi_k|^{-1/2}\right),\quad \text{for $z\in \Omega_{kj}, j=2,3$},\\
            c_1 |z+1|,\quad  {\rm near } \ z=-1,\\
            c_1|z-1|,\quad {\rm near }  \ z=1,
            \end{cases}
       %|\bar{\partial}R_{ij}| \le
  %      \begin{cases}
  %          c_1 |z+1|,\quad \text{for all $z\in \Omega_{2j}$ in a small fixed neighborhood of $-1$},\\
  %          c_1|z-1|,\quad \text{for all $z\in \Omega_{0j}$ in a small fixed neighborhood of $1$},
  %          \end{cases}
    \end{align}
    where $\varphi_k(z)= \begin{cases} \varphi^{(-1)}(z),\; k=2, \\ \varphi^{(1)}(z),\; k=0, \end{cases}$ and $\varphi^{(-1)}, \varphi^{(1)} \in C_0^\infty \left( \mathbb{R}, \left[0,1\right] \right)$ with small support near $-1$ and $1$ respectively,
    and for  $ \xi<-1$ with $k=1$,
      \begin{equation}
        |\bar{\partial}R_{1j}| \le
    c_1\left(|r'(\re(z))|+|z-\xi_1|^{-1/2}\right), \quad \text{for all $z\in \Omega_{1j}$}.
    \end{equation}

        \end{proposition}
   \begin{proof}
   	We only provide the detailed proof for $R_{21}$  under the assumption that $\xi<-1$ and $\xi=\mathcal{O}(1)$, other cases can be done in a similar way.
   	For convenience, we define the function
   	\begin{equation*}
   		g_1(z)= f_{21}T(z)^2, \quad z\in \bar{\Omega}_{21}.
   	\end{equation*}
   Observing the lemma (\ref{sds11s21}), we find that $s_{11}(z)$ and $s_{21}(z)$ have singularities at $z=\pm 1$, and $\lim_{z\to\pm1}r(z)=\mp 1$.
   This implies that $R_{21}$ is singular at $z=-1$, but the singularity can be balanced by the factor $T(z)^{-2}$. In fact, we can rewrite
  \begin{equation}
    \frac{\overline{r(z)}T_+(z)^{-2}}{1-|r(z)|^2} = \frac{\overline{s_{21}(z)}}{s_{11}(z)} \left(  \frac{s_{11}(z)}{T_+(z)}            \right)^2 =\frac{\overline{J_{21}(z)}}{J_{11}(z)} \left(  \frac{s_{11}(z)}{T_+(z)}            \right)^2
  \end{equation}
  where $J_{11}(z)=\det  \left[\psi_1^-(z;x), \psi_2^+(z;x)\right]$ and  $J_{21}(z)=\det  \left[\psi_1^+(z;x), \psi_1^-(z;x)\right]$.
  Denote $\chi_0, \chi_{-1} \in C^\infty_0 \left( \mathbb{R}, \left[0,1\right]   \right)$ with a small support near $0$ and $-1$ respectively.
  Let $z-\xi_2=s_2e^{i\varphi_2}$ with $s_2>0$. For $z \in \bar{\Omega}_{21}$, the extension is given as follows:
   \begin{align}
   	&R_{21}(z)= \hat{R}_{21}(z)+\tilde{R}_{21}(z),\nonumber\\
   	&\hat{R}_{21}(z)=\left( g_1(z)+ \left( \frac{\overline{r(\re z)}}{1-|r(\re z)|^2}-g_1(z) \right) \cos(a_0\varphi_2)  \right) T(z)^{-2}\big(1-\chi_{-1}(\re z)\big),\label{R111}\\
   	&\tilde{R}_{21}(z)=k(\re z)\frac{s^2_{11}(z)}{T^2(z)} \cos(a_0 \varphi_2)+\frac{i|z-\xi_2|}{a_0} \chi_0\left( \frac{\varphi_2}{\delta_0}\right) k'(\re z) \frac{s^2_{11}(z)}{T^2(z)} \sin(a_0\varphi_2),\label{R112}
   \end{align}
   where $k(z):=\chi_{-1}(z) \frac{\overline{J_{21}(z)}}{J_{11}(z)}$, $a_0:=\frac{\pi}{2\phi(\xi)}$ and $\delta_0$ is a small positive constant.

%where $h(z)= r\left(\re(z)\right)\cos(2\varphi_1) +g_1(z) \left(1-\cos(2\varphi_1)\right)$  and $k(z)=\chi_{1}(\re(z))r(\re(z))$.
 %\begin{equation}\label{DR11}
 %	R_{21}(z)=\Big( g_1(z)+ \big(r(\re z)-g_1(z)\cos(2\varphi_2) \big) \Big) T(z)^{-2}.
 %\end{equation}
% During the process, we have used the fact $\lim_{z\to\pm1}r(z)=\mp 1$ and $T(z)$ is analytic and bounded in the neighborhood of the point $1$.
    Next, we calculate the $\bar{\partial}$-derivatives of (\ref{R111})-(\ref{R112}). Since $\bar{\partial}=\frac{e^{i\varphi_2}}{2} \left(\partial_ {s_2} + i s_2^{-1}  \partial_{\varphi_2}\right)$, we have
  % \begin{equation}
  % 	 \begin{split}
  % 	\bar{\partial}\hat{R}_{11}&=\bar{\partial}h(z)T(z)^2 \left(1-\chi_{\mathcal{Z}}(z)\right)\left(1-\chi_{1}(\re(z))\right)\\
  % 	&-h(z)T(z)^2\bar{\partial}\chi_{\mathcal{Z}}(z)\left(1-\chi_{1}(\re(z))\right)\\
  % 	&-h(z)T(z)^2\left(1-\chi_{\mathcal{Z}}(z)\right) \bar{\partial}\chi_{1}(\re(z)).
  % 	\end{split}
  % \end{equation}
  \begin{equation}\label{dbarR12}
    \begin{split}
  	& \bar{\partial} \hat{R}_{21}= - \frac{\bar{\partial} \chi_{-1} (\re z)}{T(z)^2}\left( g_1(z)+ \left( \frac{\overline{r(\re z)}}{1-|r(\re z)|^2}-g_1(z)  \right) \cos(a_0\varphi_2) \right)T(z)^{-2} \\
    & + \bar{\partial} \left( g_1(z)+ \left( \frac{\overline{r(\re z)}}{1-|r(\re z)|^2}-g_1(z) \right) \cos(a_0\varphi_2)\right)T(z)^{-2}\big(1-\chi_{-1}(\re z)\big).
    \end{split}
\end{equation}
     We can prove that
   \begin{equation}
   \left|\frac{\overline{r(\re z)}}{1-|r(\re z)|^2}-g_1(z)\right| \lesssim |z-\xi_2|^{1/2},
   \end{equation}
   where we have used the Cauchy-Schwarz inequality.
  Let $\varphi^{(-1)}(z) \in C_{0}^\infty \left( \mathbb{R},\left[0,1\right]\right)$,
   \begin{equation}
    \varphi^{(-1)}(z)=\begin{cases}1,\; z \in \text{supp}\chi_{-1},\\ 0, \;\text{otherwise}. \end{cases}
   \end{equation}
 Then, the first term in (\ref{dbarR12}) can be bounded by $\varphi^{(-1)}(\re z)$.
% Then, if follows that $\bar{\partial}\hat{R}_{11}$ satisfies (\ref{DbarR}).
 %  As for $R_{12}$,
  % \begin{equation}
 %  	\begin{split}
 %  	\bar{\partial}R_{12} &=\frac{1}{2}e^{i\varphi_1}T(z)^2 \Big[ \cos(2\varphi_1) k'(z) \chi_1\left(1+\varphi_1\right)-\frac{2ik(z)}{s}\sin(2\varphi_1)\\
  % 	&+\frac{i}{2}\left(sk(z)\right)' \sin(2\varphi_1) \chi_1(1+\varphi_1)-\frac{1}{2}k'(z)\sin(2\varphi_1)\chi'_1(1+\varphi_1) \Big].
  % 	 \end{split}
  % \end{equation}
 %  We can easily to find that there exists a constant $c_1$ such that
 %  $|\bar{\partial} R_{12}| < c_1 |\chi_{1}(\re(z))||z-\xi_1|^{-1/2}$.
%   Finally, for $z \sim 1$, $|\bar{\partial}R_{12}|\le %c \left[ %|\sin(2\varphi_1)+\left(1-\chi_1(1+\varphi_1)\right)|\r%ight]=\mathcal{O}(\varphi_1)$.
   Therefore, the inequality
\begin{equation}
    |\bar{\partial}R_{21}| \lesssim  \varphi^{(-1)}(\re z)+|r'(\re z)|+|z-\xi_2|^{-1/2}, \quad \text{for}\; z \in \Omega_{21},
\end{equation}
follows immediately.

As for $\tilde{R}_{21}(z)$,
\begin{equation}
    \begin{split}
    \bar{\partial}\tilde{R}_{21}(z) &=\frac{1}{2}e^{i\varphi_2} s^2_{11}(z)T(z)^{-2} \Big[ \cos(a_0 \varphi_2) k'(s_2) \left(1- \chi_{0}\left( \frac{\varphi_2}{\delta_0} \right) \right)-\frac{ia_0 k(s_2)}{s_2}\sin(a_0\varphi_2)\\
    &+\frac{i}{a_0}\left(s_2k'(s_2)\right)' \sin(a_0\varphi_2) \chi_0\left(\frac{\varphi_2}{\delta_0}\right)-\frac{1}{a_0 \delta_0}k'(s_2)\sin(a_0\varphi_2)\chi'_0 \left(\frac{\varphi_2}{\delta_0}\right) \Big].\nonumber
     \end{split}
\end{equation}
We can easily to find that $|\bar{\partial} \tilde{R}_{21}(z)| \lesssim \varphi^{(-1)}(\re z)$. Thus, the result (\ref{dbarrij}) can be obtained.
%there exists a constant $c_1$ such that
%$|\bar{\partial} R_{21}| \le c_1 |\chi_{- 1}(\re(z))||z-\xi_2|^{-1/2}$.
   \end{proof}

    Finally, we use $R^{(2)}(z)$ to make a new transformation
    \begin{equation}\label{m1tom2}
        M^{(2)}(z)=M^{(1)}(z)R^{(2)}(z),
    \end{equation}
    where for $z\in \Omega_{kj}(k=0,1,2)$,
    \begin{align}
    	R^{(2)}(z)=\begin{cases}
      		\left(\begin{array}{cc}
    	1& (-1)^\tau R_{kj}e^{-2it\theta}\\0&1
    \end{array} \right), \quad j=1,3,\\
    \left( \begin{array}{cc}
    	1& 0\\(-1)^\tau R_{kj}e^{2it\theta}&1
    \end{array} \right), \quad  j=2,4,
    	\end{cases}
    \end{align}
    with $\tau=\begin{cases}
   	0,\; j=1,4,\\
    	1,\; j=2,3,
    \end{cases}$
 and for $z$ belongs to other regions, $R^{(2)}=I$.
    $M^{(2)}$ satisfies the following  hybrid  $\bar{\partial}$-problem.

    \noindent\textbf{$\bar{\partial}$-RHP}.  Find a matrix-valued function $M^{(2)}(z)=M^{(2)}(z;x,t)$ which satisfies
    \begin{itemize}
        \item Analyticity: $M^{(2)}(z)$ is continuous in $\mathbb{C}\setminus   \Sigma^{(2)} $, where
        $$\Sigma^{(2)}=\cup_{j=1}^4  \left( \left(\cup_{k=1}^2 \Sigma_{kj} \right) \cup \Sigma'_{j}\right)   \cup \Sigma^{pole}.$$
       % \item Symmetry: $M^{(2)}(z)=\sigma_1 \overline{M^{(2)}(\bar{z})}\sigma_1 =z^{-1}M^{(2)}(z^{-1})\sigma_1$.
        \item Jump condition:\begin{equation}
            M^{(2)}_+(z)=M^{(2)}_-(z)V^{(2)}(z),
        \end{equation}
        where
       \begin{align}\label{V2}
       	V^{(2)}(z)=\begin{cases}
       	\left(	\begin{array}{cc}
       			1& -f_{kj}e^{-2it\theta}\\
       			0 & 1
       		\end{array}\right),\quad z\in \Sigma_{kj},\, j=1,3,\\
       		\left( \begin{array}{cc}
       		1& 0\\
       		f_{kj}e^{2it\theta} & 1
       	\end{array}\right),\quad z\in  \Sigma_{kj},\,j=2,4,\\
       \left(		\begin{array}{cc}
       	1& (f_{(k-1)j}-f_{kj})e^{-2it\theta}\\
       	0 & 1
       \end{array}\right),\quad z\in \Sigma'_{j},\, j=1,4,\\
     	\left(	\begin{array}{cc}
     	1& 0\\
     	(f_{(k-1)j}-f_{ij})e^{2it\theta} & 1
     \end{array}\right),\quad z\in \Sigma'_{j},\, j=2,3,\\[3pt]
        V^{(1)}(z),\; z \in \Sigma^{{pole}},
       	\end{cases}
       \end{align}
       with $k=1,2$.

        \item Asymptotic behaviors:
        \begin{align*}
                &M^{(2)}(z)=I+\mathcal{O}(z^{-1}),	\quad  z \to  \infty,\\
                &zM^{(2)}(z)=\sigma_1+\mathcal{O}(z), \quad z \to 0.
        \end{align*}
        \item $\bar{\partial}$-derivative: For $z\in \mathbb{C}\setminus   \Sigma^{(2)}$, we have
        \begin{equation}
            \bar{\partial}M^{(2)}(z)= M^{(2)}(z) \bar{\partial}R^{(2)}(z),
        \end{equation}
        where
        \begin{equation}\label{parR2}
            \bar{\partial}R^{(2)}(z)= \begin{cases}
            	
            		\left(\begin{array}{cc}
            			1&(-1)^\tau \bar{\partial}R_{kj}e^{-2it\theta}\\0&1
            		\end{array} \right), \quad  j=1,3,\\
            		\left( \begin{array}{cc}
            			1& 0\\(-1)^\tau \bar{\partial}R_{kj}e^{2it\theta}&1
            		\end{array} \right), \quad  j=2,4,
            	\end{cases}
           \end{equation}
        with $k=0,1,2$.
%        \item Residue conditions: $M^{(2)}(z)$  has simple poles at each points $z_n$ in $\mathcal{Z}^+ \cup \mathcal{Z}^-$ with
%        \begin{align}\label{resm2}
%            \begin{split}
%                & \res_{z=z_n} M^{(2)}(z)=\lim_{z\to z_n} M^{(2)}(z) \left(\begin{array}{cc}
%                    0 & 0\\
%                    c_n T^2(z_n)e^{2it\theta(z_n)} & 0
%                \end{array} \right),\\
%                & \res_{z=\bar{z}_n} M^{(2)}(z)=\lim_{z\to \bar{z}_n} M^{(2)}(z) \left( \begin{array}{cc}
%                    0 & \bar{c}_n \bar{T}^2(z_n)e^{-2it\theta(\bar{z}_n)}\\
%                    0 & 0
%                \end{array} \right).
%            \end{split}
%\end{align}
    \end{itemize}
 %   \begin{remark}
 %	Besides, the new RH problem $M^{(2)}(z)$ takes advantage of the properties of the decay/growth of $e^{\pm 2it\theta(z)}$.
 %  \end{remark}

To solve $M^{(2)}(z)$, we decompose it into a pure RH problem  $M^{(2)}_{RHP}(z)$  with $\bar{\partial}R^{(2)}(z)=0$ and a pure $\bar{\partial}$-problem
$M^{(3)}(z)$ with
\begin{align*}
&\bar{\partial} M^{(3)}(z)= M^{(3)}(z) W^{(3)}(z),\\
&W^{(3)}(z) =M^{(2)}_{RHP}(z) \bar{\partial}R^{(2)}(z) M^{(2)}_{RHP}(z)^{-1}.
 \end{align*}
	Next we analyze the two problems obtained by decomposition respectively.

	\section{Contribution from a pure RH problem }

	 % \subsection{Analysis on the   pure RH problem  }
	
	  We first consider the following  pure RH problem.
	
	  \noindent\textbf{RHP3}.  Find a matrix-valued function $M^{(2)}_{RHP}(z)=M^{(2)}_{RHP}(z;x,t)$ which satisfies
	  \begin{itemize}
	  	\item Analyticity: $M^{(2)}_{RHP}(z)$ is analytic in $\mathbb{C}\backslash  \Sigma^{(2)}$.
	  %	\item Symmetry: $M^{(2)}_{RHP}(z)=\sigma_1 \overline{M^{(2)}_{RHP}(\bar{z})}\sigma_1 =z^{-1}M^{(2)}_{RHP}(z^{-1})\sigma_1$.
	  	\item Jump condition:\begin{equation}
	  		M^{(2)}_{RHP+}(z)=M^{(2)}_{RHP-}(z)V^{(2)}(z),
	  	\end{equation}
	  	where $V^{(2)}(z)$ is given by (\ref{V2}).
	  	\item Asymptotic behaviors: $M^{(2)}_{RHP}(z)$  has the same asymptotic behaviors with  $M^{(2)}(z)$.
	  	\item $\bar{\partial}$-derivative:  $\bar{\partial}R^{(2)}(z)=0, \quad z \in  \mathbb{C}\backslash   \Sigma^{(2)}.$

%	  	\item Residue conditions:  $M^{(2)}_{RHP}(z)$  satisfies  the residue relations in (\ref{resm2}) with  $M^{(2)}_{RHP}(z)$  replacing $M^{(2)}(z)$.
	  	
	  \end{itemize}

To  separate out  poles  from the pure RH problem  $M^{(2)}_{RHP}(z)$, we  define
     $$ \mathcal{U}_{\xi}=  \mathcal{U}_{\xi_1}\cup  \mathcal{U}_{\xi_2}, \ \ \   \mathcal{U}_{\xi_k}=\{z: |z-\xi_k|< \rho \}, \ \ k=1,2. $$

   In addition,  the jump matrix $V^{(2)}(z) $ has the following estimation.
     \begin{proposition} \label{Estv2}
     	There exists a positive constant $c_p$   such that for $1 \le p \le \infty$,
     	\begin{equation}\label{EV2}
     	||V^{(2)}(z)-I||_{L^p(\Sigma^{(2)} \backslash \mathcal{U}_{\xi})} =
     		\mathcal{O}\left(c_p e^{-c_p t}\right),\quad t  \rightarrow \infty.
        \end{equation}
     \end{proposition}
    \begin{proof}
    	We prove  three cases for $z \in \Sigma_{21} \backslash  \mathcal{U}_{\xi}$, $z \in \Sigma'_3$ and $z \in \{z\in \mathbb{C}:|z-z_1|=\rho\}$ for $\xi >1$, and the other cases can be proved similarly.
        From the definition of $V^{(2)}$ and $R_{21}$, we have for  $z \in \Sigma_{21} \backslash  \mathcal{U}_{\xi}$ and $1 \le p <\infty$,
    	\begin{equation*}
    		||V^{(2)}(z)-I||_{L^p(\Sigma_{21} \backslash  \mathcal{U}_{\xi})}= ||R_{21} e^{-2it\theta(z)}||_{L^p(\Sigma_{21} \backslash  \mathcal{U}_{\xi})}\lesssim ||e^{-2it\theta(z)}||_{L^p(\Sigma_{21} \backslash  \mathcal{U}_{\xi})}.
    	\end{equation*}
        Denote $z = \xi_2 + l e^{i\varphi}$, $l\in (\rho,\infty)$ for $z \in \Sigma_{21} \backslash  \mathcal{U}_{\xi}$. Then the proposition \ref{prop2} tells us $ ||e^{-2it\theta(z)}||^p_{L^p(\Sigma_{21} \backslash  \mathcal{U}_{\xi})} \lesssim t^{-1}e^{-c_p t} $.
      %  \begin{equation*}
      %      ||e^{-2it\theta(z)}||^p_{L^p(\Sigma_{21} \backslash  \mathcal{U}_{\xi})}  \lesssim \int_{\rho}^\infty e^{-pctl}\mathrm{d} l \lesssim t^{-1}e^{-c_p t}.
      %  \end{equation*}
%    	Moreover, we can prove if we make a Taylor expansion of $\theta(z)$ at the point $\xi_2$, which is
%    $$\theta(z)= \theta(\xi_2)+\frac{\theta''(\xi_2)}{2}(z-\xi_2)^2+G_2(z;\xi_2).$$
%     Then, we find as $t\to\infty$, $|e^{-itG_2(z;\xi_2)}|\to 1$, so we obtain
%     $$|e^{-it\theta(z)}|\lesssim |e^{\theta''(\xi_2)s_2^2 \sin(2\varphi_2)/2 }|,$$
%     where $z-\xi_2=s_2 e^{i \varphi_2}$ and $\sin(2\varphi_2)>0$ for $z\in \Sigma_{21} \backslash  \mathcal{U}_{\xi}$. The detailed proof can be found in the subsection (\ref{locmod}). Therefore,
%    	\begin{equation*}
%    		|V^{(2)}(z)-I|  \lesssim e^{ \theta''(\xi_2)  \sin(2\varphi_2)\rho^2 t}.
%    	\end{equation*}
        For $z \in \Sigma'_3$,  then there exists a positive constant $l$ such that
        \begin{equation*}
        	||V^{(2)}(z)-I||_{L^p(\Sigma'_3)}= ||(R_{22}-R_{12}) e^{2it\theta(z)}||_{L^p(\Sigma'_3)} \lesssim ||e^{2it\theta(z)}||_{L^p(\Sigma'_3)}\lesssim  t^{-1/p}e^{- c_p t}.
        \end{equation*}
        For $z \in \{z\in \mathbb{C}: |z- z_1|=\rho \}$  under $\xi >1$,
        \begin{align*}
            &||V^{(2)}(z)-I||_{L^p(\{z\in \mathbb{C}: |z- z_1|=\rho \})} = ||c_1 T^{2}(z)e^{2it\theta(z_1)} (z-z_1)^{-1}||_{L^p(\{z\in \mathbb{C}: |z- z_1|=\rho \})} \\
           & \lesssim C(\rho,p) ||e^{2it\theta(z_1)}||_{L^p(\{z\in \mathbb{C}: |z- z_1|=\rho \})} \lesssim C(\rho,p) e^{-c_p t},
        \end{align*}
        where $C(\rho,p)$ is a positive constant.
        It is obvious that the estimation (\ref{EV2}) is right when $p=\infty$ in the above cases.
    %   For $z\in \Sigma_{21} \cap \mathcal{U}_{\xi}$, we have
    %   \begin{equation}
    %   	|V^{(2)}(z)-I| \lesssim |z-\xi_2|^{-1} t^{-1/2}.
    %   \end{equation}
    	The   other cases can be shown in a  similar way.

    \end{proof}

      The proposition (\ref{Estv2}) tells us for $z\in \Sigma^{(2)} \backslash  \mathcal{U}_{\xi}$ the jump matrix $V^{(2)}(z)$  uniformly  goes to identity.  Thus outside the  $\mathcal{U}_{\xi}$ there is only exponentially small error by completely ignoring the jump condition of $M^{(2)}_{RHP}(z)$.
      Then, the pure RH problem $M_{RHP}^{(2)}(z)$ can be decomposed into two parts:
      \begin{equation}\label{m2rhp}
          M^{(2)}_{RHP}(z)=\begin{cases}
              E(z)M^{out}(z),\quad z \in \mathbb{C} \backslash \mathcal{U}_{\xi},\\
              E(z)M^{out}(z)M^{lo}(z), \quad z \in \mathcal{U}_{\xi},
          \end{cases}
      \end{equation}
       where $M^{out}(z)$ is a solution by ignoring the  jump conditions of $M^{(2)}_{RHP}(z)$ which has no discrete spectrum since we have transformed  the information of poles to the jump condition,
       $M^{lo} (z)$ is a local model for phase points which matches with the parabolic cylinder model problem, and $E(z)$ is the error function which satisfies a
        small-norm RH problem.

	      \subsection{The outer model}\label{outmod}
	      In this subsection,  we consider the outer model $M^{out}(z)$.
          Since we have converted the poles to the jumps, there has no discrete spectrum  in  $M^{out}(z)$.
          In fact,  in our situation, all discrete spectrums $z_j \ (j=1,\cdots,N)$ are away from the critical line $\re z = \xi$  , so the soliton contribution to the problem is exponentially small.
          This is because  when the soliton is close to the critical line it will make the exponential term not decay, which is $e^{\pm 2it\theta(z_j)} = \mathcal{O}(1)$, otherwise the exponential term will decay in the corresponding region.
          $M^{out}(z)$  satisfies the following  RH problem.
    %      we build a reflectionless case of RHP1 to show that its
	%      solution can approximated by $M^{out}(z)$.
	%      $M^{out}(z)$ is the soliton solution of $M^{(2)}(z)$ in the soliton region, which satisfies the RH problem as follows:

  \noindent\textbf{RHP4}.  Find a matrix-valued function
	      $M^{out}(z)=M^{out}(z;x,t)$ which satisfies
	      \begin{itemize}
	      	\item Analyticity: $M^{out}(z)$ is analytic in $\mathbb{C} \backslash \{0\}$.
	      	\item	Symmetry: $M^{out}(z)=\sigma_1 \overline{M^{out}(\bar{z})}\sigma_1 =z^{-1}M^{out}(z^{-1})\sigma_1$.
	      	\item Asymptotic behaviors:  	
	      	\begin{align*}
	      		&M^{out}(z)=I+\mathcal{O}(z^{-1}),	\quad  z \to  \infty,\\
	      		&zM^{out}(z)=\sigma_1+\mathcal{O}(z), \quad z \to 0.
	      	\end{align*}
	      %	\item $\bar{\partial}$-Derivative:
	      %	\begin{equation}
	      %		\bar{\partial}R^{(2)}=0, \quad z \in  \mathbb{C}\setminus  \left(   \mathcal{Z}^+ \cup  \mathcal{Z}^- \right).
	      %	\end{equation}
	      %	\item Residue conditions:  $M^{out}(z)$ has the same form of residue conditions with $M^{(2)}(z)$.
	      \end{itemize}
	
	     % Furthermore, a reflectionless case of RHP1 $M^{rl}(z|\sigma_d)$ follows the RH problem:
	
	     % \noindent\textbf{RHP6}.  Find a matrix-valued function
	     % $M^{(rl)}(z|\sigma_d)$ which satisfies
	     % \begin{itemize}
	     % 	\item Analyticity: $M^{(rl)}(z|\sigma_d)$ is analytic in $\mathbb{C}\backslash  \left( \mathcal{Z}^+ \cup  \mathcal{Z}^- \right)$.
	     % 	\item
	     % 	Symmetry: $M^{(rl)}(z|\sigma_d)=\sigma_1 \overline{M^{(rl)}(\bar{z}|\sigma_d)}\sigma_1 =z^{-1}M^{(rl)}(z^{-1}|\sigma_d)\sigma_1$.
	     % 	\item Asymptotic behaviors:  	
	     % 	\begin{align*}
	     % 		&M^{(rl)}(z|\sigma_d)=I+\mathcal{O}(z^{-1}),	\quad  z \to  \infty,\\
	     % 		&zM^{(rl)}(z|\sigma_d)=\sigma_1+\mathcal{O}(z), \quad z \to 0.
	     % 	\end{align*}
	     % 	\item Residue conditions:  $M^{(rl)}(z|\sigma_d)$  has  the same form of residue conditions with   $M^{(1)}(z)$.
	     % \end{itemize}
	
	      \begin{proposition}
	      %	Given the discrete data $\sigma_d=\left\{r(z), \left(z_n,c_n\right)_{n=1}^N	\right\}$, the above RHP6 exists a unique solution.
	      	The RHP4  exists  a unique solution
	      	\begin{equation}	
	      		M^{out}(z)=I + z^{-1} \sigma_1.
	      	\end{equation}
	      %	where  $\sigma_d^{(out)}:=\left\{0,  \left(z_n,\tilde{c}_n\right)_{n=1}^N \right\}$ is the modified scattering data with
	      %	\begin{equation}
	      % 		\tilde{c}_n=c_n \exp\left(-2i \int_{I(\xi)} v(s) \left(\frac{1}{s-z}-\frac{1}{2s}\right) \, \mathrm{d}s \right).
	      %	\end{equation}
      	 %In addition, we have
      	 %\begin{equation}
      	 %	 q_{sol,N}(x,t)=\lim_{z\to\infty} \left(z M^{out}(z)\right)_{21}=1,
      	 %\end{equation}
	    % where $q_{sol,N}(x,t)$ is the $N$-soliton solution of the NLS equation.
        %with the modified scattering data $\sigma_d^{(out)}$.
      	\end{proposition}
      \begin{proof}
      	 %The proof is similar to the proof of Lemma 6.6 in \cite{CJ} and Appendix B in \cite{fNLS} through a direct explicit construction.
    	  $M^{out}(z)$ is analytic  in $\mathbb{C}$ except for   $z = 0$.  Making a  transformation
        \begin{equation}
             \tilde{M}(z) =  M^{out}(z) \left( I + z^{-1} \sigma_1\right)^{-1}.
        \end{equation}
        Notice that
        \begin{equation}
            \left( I +  z^{-1} \sigma_1 \right)^{-1} = \left(1-z^{-2}\right)^{-1} \sigma_2  \left( I + z^{-1} \sigma_1\right)^{T} \sigma_2,
        \end{equation}
      then we find that
         \begin{align*}
             \lim_{z \to 0}  \tilde{M}(z) & = \lim_{z \to 0} M^{out}(z) \left( I +z^{-1} \sigma_1\right)^{-1} = I, \\
            \lim_{z\to\infty}  \tilde{M}(z)& = \lim_{z\to\infty}  M^{out}(z) \left( I + z^{-1} \sigma_1\right)^{-1} = I.
       \end{align*}
        $\tilde{M}(z)$ is bounded and analytic in the complex plane, so it is a constant matrix, which is $\tilde{M}(z)=I$.
        Finally, we obtain $M^{out}(z) = I +z^{-1}\sigma_1$.

      	\end{proof}

        \begin{remark}
            The result here is a special case of the $m^{sol}_{j_0}$ in the paper \cite{CJ} with $j_0 = -1$ where  the critical line $\re z = \xi$ is not passing through any neighborhood of the  discrete spectra $z_j\in \mathcal{Z}$.
        \end{remark}

      Outside the region $\mathcal{U}_{\xi}$, the error between $M^{(2)}_{RHP}(z)$ and $M^{out}(z)$ mainly comes from the contribution of neglecting the  jump line and  we will study the error in more detail in subsection  \ref{emod}. Before we consider the error function, we construct local models for the neighborhood of the phase points $\xi_k$ ($k=1,2$).
	
          \subsection{The local model}\label{locmod}

          Denote  local jump  contours
          $$\Sigma^{lo}:=\Sigma^{(2)}\cap  \mathcal{U}_{\xi }    =  \Sigma_{1}\cup \Sigma_{2}, $$
          where
           $\Sigma_{k}:= \cup_{j=1}^4 \Sigma_{kj} \cap \mathcal{U}_{\xi_k}, \, k=1,2$. See Figure \ref{local}.
          Consider the  following  local RH problem:

    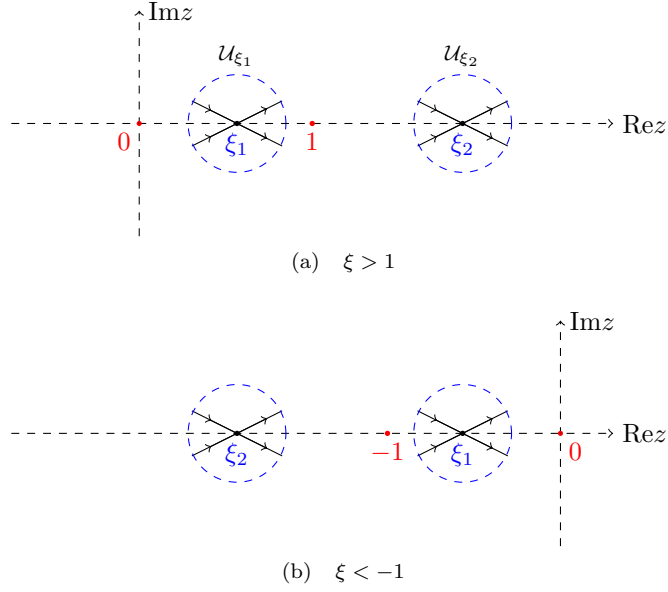
\begin{figure}[h]
	\begin{center}
		\subfigure[\footnotesize\ \ $ \xi>1$]{
			\begin{tikzpicture}
			\draw(-1,0)--(-0.4,0.3);
			\draw[->](-1,0)--(-0.6,0.2);
			\draw(-1,0)--(-1.6,0.3);
			\draw[-<](-1,0)--(-1.4,-0.2);
			\draw(-1,0)--(-0.4,-0.3);
			\draw[-<](-1,0)--(-1.4,0.2);
			\draw(-1,0)--(-1.6,-0.3);
			\draw[->](-1,0)--(-0.6,-0.2);
			\draw[dashed,blue](-1,0) circle (0.65);
			\node at    ( 2,0.9) {\footnotesize $ \mathcal{U}_{\xi_2}$};
			\node at    (-1,0.9) {\footnotesize $ \mathcal{U}_{\xi_1}$};
			\draw[->,dashed](-4,0)--(4,0)node[right]{ \textcolor{black}{Re$z$}};
			\draw[->,dashed](-2.3,-1.5)--(-2.3,1.5)node[right]{\textcolor{black}{Im$z$}};
			\draw(2,0)--(1.4,0.3);
			\draw[-<](2,0)--(1.6,0.2);
			\draw(2,0)--(1.4,-0.3);
			\draw[->](2,0)--(2.4,-0.2);
			\draw(2,0)--(2.6,0.3);
			\draw[->](2,0)--(2.4,0.2);
			\draw(2,0)--(2.6,-0.3);
			\draw[-<](2,0)--(1.6,-0.2);
			\draw[dashed,blue](2,0) circle (0.65);
			\coordinate (I) at (0,0);
			\fill[red] (I) circle (1pt) node[below] {$1$};
			\coordinate (b) at (-1,0);
			\fill (b) circle (1pt) node[blue,below] {$\xi_1$};
			\coordinate (f) at (2,0);
			\fill (f) circle (1pt) node[blue,below] {$\xi_2$};
			\node[shape=circle,fill=red,scale=0.11] at (-2.3,0) {0};
			\node[below,red] at (-2.5,0) {$0$};
			\end{tikzpicture}
			\label{si1}}\\
		\subfigure[\footnotesize \ \ $ \xi<-1$]{
			\begin{tikzpicture}
			\draw(-1,0)--(-0.4,0.3);
			\draw[->](-1,0)--(-0.6,0.2);
			\draw(-1,0)--(-1.6,0.3);
			\draw[-<](-1,0)--(-1.4,-0.2);
			\draw(-1,0)--(-0.4,-0.3);
			\draw[-<](-1,0)--(-1.4,0.2);
			\draw(-1,0)--(-1.6,-0.3);
			\draw[->](-1,0)--(-.6,-0.2);
			
			\draw[->,dashed](-4,0)--(4,0)node[right]{ \textcolor{black}{Re$z$}};
			\draw[->,dashed,black](3.3,-1.5)--(3.3,1.5)node[right]{\textcolor{black}{Im$z$}};
			
			\draw(2,0)--(1.4,0.3);
			\draw[-<](2,0)--(1.6,0.2);
			\draw(2,0)--(2.6,-0.3);
			\draw[->](2,0)--(2.4,-0.2);
			\draw(2,0)--(2.6,0.3);
			\draw[->](2,0)--(2.4,0.2);
			\draw(2,0)--(1.4,-0.3);
			\draw[-<](2,0)--(1.6,-0.2);
			\coordinate (I) at (1,0);
			\fill[red] (I) circle (1pt) node[below] {$-1$};
			%	\coordinate (A) at (-4,0);
			%		\fill (A) circle (1pt) node[blue,below] {$\xi_4$};
			\coordinate (b) at (-1,0);
			\fill (b) circle (1pt) node[blue,below] {$\xi_2$};
			%		\coordinate (e) at (4,0);
			%	\fill (e) circle (1pt) node[blue,below] {$\xi_1$};
			\coordinate (f) at (2,0);
			\fill (f) circle (1pt) node[blue,below] {$\xi_1$};
			%	\coordinate (c) at (2,0);
			%	\fill[red] (c) circle (1pt) node[below] {$0$};
			\node[shape=circle,fill=red,scale=0.11] at (3.3,0) {0};
			\node[below,red] at (3.5,0) {$0$};
			\draw[dashed,blue](2,0) circle (0.65);
			\draw[dashed,blue](-1,0) circle (0.65);
			\end{tikzpicture}
			\label{si2}}
		\caption{ \footnotesize Jump contours $\Sigma^{lo}$ for  $ M^{lo}(z)$ corresponding to the   cases    $ \xi>1$  and   $ \xi<-1$ respectively.}
		\label{local}
	\end{center}
\end{figure}

            \noindent\textbf{RHP5}.  Find a matrix-valued function  $M^{lo}(z)=M^{lo}(z;x,t)$ such that
              \begin{itemize}
          	\item Analyticity: $M^{lo}(z)$ is analytic in $\mathbb{C}\setminus \Sigma^{lo}$.
          %	\item Symmetry: $M^{lo}(z)=\sigma_1 \overline{M^{lo}(\bar{z})}\sigma_1 =z^{-1}M^{lo}(z^{-1})\sigma_1$.
            \item Symmetry: $M^{lo}(z)=\sigma_1 \overline{M^{lo}(\bar{z})} \sigma_1 $.
          	\item Jump condition: $M^{lo}_+(z)=M^{lo}_-(z)V^{(2)}(z), \quad z\in \Sigma^{lo}.$
          	\item Asymptotic behavior:  $	M^{lo}(z)=I+\mathcal{O}(z^{-1}),	\quad  z \to  \infty.$
          \end{itemize}
          This  local  RH problem, which consists of two  local models on $\Sigma_1$ and $\Sigma_2$,   has the jump condition and no poles.
First, we show as $t \to \infty$, the interaction between $\Sigma_1$ and $\Sigma_2$  reduces to  0 to higher order and the contribution to the solution of $M^{lo}(z)$ is simply the sum of the separate contributions from $\Sigma_1$  and $\Sigma_2$.

     We consider the trivial decomposition of the jump matrix
           \begin{align*}
           &V^{(2)}(z)=b_-^{-1} b_+, \ \  \; b_-=I,  \ \ \; b_+=V^{(2)}(z),\\
          	&w_-=0,\ \   w_+= V^{(2)}(z)-I,\ \  w=V^{(2)}(z)-I,\\
           &w=w_1+w_2, \; w_1=0, z\in \Sigma_2,\;w_2=0, z\in \Sigma_1,
          \end{align*}
          then
          \begin{align*}
           &C_wf=C_-(fw_+)+C_+(fw_-)=C_-(f(V^{(2)}-I)),
          \end{align*}
          where $C_w=C_{w_1}+C_{w_2}$ and $C_\pm$ is the Cauchy operator on $\Sigma^{(2)}$ defined by
       \begin{equation}
          	C_\pm f(z)=\lim_{z'\to z \in \Sigma^{(2)}} \frac{1}{2\pi i} \int_{\Sigma^{(2)}} \frac{f(s)}{s-z'} \, \mathrm{d}s.
          \end{equation}
          Then we have the following lemma.
          \begin{proposition}\label{sum1}
          	\begin{align}
          			&||C_{w_1} C_{w_2}||_{L^2(\Sigma^{lo})}=||C_{w_2} C_{w_1}||_{L^2(\Sigma^{lo})} \lesssim t^{-1},\label{cw1cw21}\\
          			&||C_{w_1}C_{w_2}||_{L^\infty(\Sigma^{lo}) \to L^2(\Sigma^{lo}) }, 	||C_{w_2}C_{w_1}||_{L^\infty(\Sigma^{lo}) \to L^2(\Sigma^{lo}) } \lesssim t^{-1}\label{cw1cw22}.
          	\end{align}
          \end{proposition}
          \begin{proof}
          	From the definition of the operators, we have for any $f \in L^{\infty}\cap L^{2} \left(\Sigma^{lo}\right)$,
          	\begin{equation}
          		C_{w_1}C_{w_2} f =C_+ \left( C_- (f w_{2+}) w_{1-}\right) + C_- \left( C_+ (f w_{2-}) w_{1+}\right).
          	\end{equation}
          	Then,
          	\begin{align}
          		||C_- \left(C_+ (f w_{2-}) w_{1+}\right)(\cdot)||_{L^2\left(\Sigma^{lo}\right)}&=\Big|\Big|\int_{\Sigma_1} \left( \int_{\Sigma_2} \frac{f(\eta)w_{2-}(\eta)}{(\eta-\kappa)_+} \, \mathrm{d}\eta  \right)\frac{ w_{1+}(\kappa)}{(\kappa-\cdot)_-} \, \mathrm{d}\kappa \Big|\Big|_{L^2\left(\Sigma^{lo}\right)} \nonumber\\
          		& \lesssim ||w_{1+}||_{L^2(\Sigma_1)} \sup_{\kappa \in \Sigma_1} \Big|\int_{\Sigma_2} \frac{f(\eta)w_{2-}(\eta)}{\eta-\kappa}\, \mathrm{d}\eta \Big|.\nonumber
          	\end{align}
            Furthermore, we have $||w||_{L^2\left(\Sigma^{lo}\right)} \lesssim t^{-1/2}$. Combining the above two equations we can obtain the estimations (\ref{cw1cw21})-(\ref{cw1cw22})  after direct calculations.
          \end{proof}

 Following the idea and  step in  \cite{sdmRHp},  we can derive the proposition:
         \begin{proposition}\label{sum}
            \begin{equation}
            \int_{\Sigma^{lo}} \frac{ (I-C_w)^{-1}Iw}{s-z}\, \mathrm{d}s=\sum_{j=1}^{2} \int_{\Sigma_{j}} \frac{ (I-C_{w_j})^{-1}Iw_j}{s-z}\, \mathrm{d}s + \mathcal{O}(t^{-3/2}).
            \end{equation}
        \end{proposition}

        We consider the Taylor expansion in the neighborhood of $\xi_k, \   k=1,2$,
        \begin{equation}
            \theta(z)=\theta(\xi_k)+\frac{\theta''(\xi_k)}{2} (z-\xi_k)^2 + G_k(z;\xi_k),
        \end{equation}
        where $G_k(z;\xi_k)=\mathcal{O}\left((z-\xi_k)^3\right)$.
        %Looking at the picture \textcolor{red}{Figure}, we find that $\theta''(\zeta_1)>0$ and $\theta''(\zeta_2)<0$.
      %  \begin{equation} \label{vth}
      %      \vartheta=\begin{cases}
      %  1, & z \in U_{\xi_1},\\
      %  -1,& z \in U_{\xi_2}.
     %   \end{cases}
     %  \end{equation}
 Then we obtain the following proposition.
          \begin{proposition}\label{sum2}
           Let $\xi=\mathcal{O}(1)$, define  operators $N_k$ ($k=1,2$)
        \begin{equation}
            N_k: g(z) \to (N_kg)(z)=g\left(\frac{s}{\sqrt{ 2t \epsilon_k\theta''(\xi_k)\epsilon_k}}+\xi_k\right),
        \end{equation}
        where $s=u\xi_k e^{\pm i\varphi_k}, \ |u|<  \rho, \ k=1,2 $.
    Then  we have
              \begin{equation}\label{G}
                      \Big|\exp \left\{-it G_k\left(\frac{s}{\sqrt{ 2t\epsilon_k\theta''(\xi_k)  }}+\xi_k;\xi_k \right)\right\}\Big|\to 1, \quad \text{as}\; t\to \infty.
              \end{equation}

          \end{proposition}
         \begin{proof}
            We only give the proof for $\xi_1$ and the proof for $\xi_2$ can be given similarly.
              The operator $N_1$ acting on $T^{-1}(z)e^{-it\theta(z)}$  in the neighborhood of $\xi_1$ gives
             \begin{equation}
                 N_1\left(T^{-1}e^{-it\theta(z)} \right) = T^{-1} \left( \frac{s}{\sqrt{2t\theta''(\xi_1)}}+\xi_1 \right)  \exp\left( -it \theta \left(\frac{s}{\sqrt{2t\theta''(\xi_1)}}+\xi_1 \right)\right).\nonumber
             \end{equation}
             Direct calculations show that
             \begin{align}
                 &N_1\theta(z)=\theta(\xi_1)+\frac{s^2}{4t}+G_1\left(\frac{s}{\sqrt{2t\theta''(\xi_1)}}+\xi_1 ;\xi_1\right),\nonumber\\
                 &N_1T^{-1}(z)= \left(2t\theta''(\xi_1) \right)^{-\frac{v(\xi_1)}{2}}e^{w \left(   \frac{s}{\sqrt{2t\theta''(\xi_1)}}+\xi_1  \right)} s^{iv(\xi_1)}  \left(\frac{s}{\sqrt{2t\theta''(\xi_1)}}+\xi_1 -\xi_2\right)^{-iv(\xi_2)}, \nonumber
             \end{align}
             where
             \begin{align}
                % &v_1=v(\xi_1),\quad v_2=v(\xi_2),\\
                 &w(z)=-\frac{1}{2\pi i} \int_{I(\xi)} \ln|s-z|\, \mathrm{d}\ln(1-|r|^2).
                % &h_2(z)=-\frac{1}{2\pi i} \int_{0}^{\xi_1} \ln|s-z|\, \mathrm{d}\ln(1-|r|^2).
             \end{align}
             For $s=u\xi_1e^{\pm i\varphi_1}$, we can obtain the result (\ref{G})  using  direct calculations.
         \end{proof}
      % \textcolor{red}{Then following the idea of Lemma 3.35 in \cite{PAX} or section 8.2 in \cite{PP} for the phase reduction, we found  the error between using $\theta(z)$ and using $\theta(\xi_k)+\frac{\theta''(\xi_k)(z-\xi_k)^2}{2}$ can be bounded by the order $\mathcal{O}(t^{-1})$.}

          Second, following the lemma  \ref{sum}, we introduce two local models $M^{lo,k}$  on the  jump contour  $\Sigma_{k}$  $(k=1, 2)$
           whose solutions can be given explicitly in terms of parabolic cylinder functions on the two contours respectively.
          %We describe the details of constructing the RHP $M_{\zeta_1}$ and $M_{\zeta_2}$  can be constructed in a similar way.
          We consider the following RH problems:

           \noindent\textbf{RHP6}.  Find a matrix-valued function  $M^{lo,k}(z)=M^{lo,k}(z;\xi_k,x,t)$ such that
           \begin{itemize}
          	\item Analyticity: $M^{lo,k}(z)$ is analytic in $\mathbb{C}\backslash \Sigma_{k}$.
             \item Symmetry: $M^{lo,k}(z)=\sigma_1 \overline{M^{lo,k}(\bar{z})} \sigma_1 $.
          	\item Jump condition:
          	\begin{align*}
          		M^{lo,k}_+(z)=M^{lo,k}_-(z)V^{lo,k}(z), \quad z\in \Sigma_k,
          	\end{align*}
          where
       \begin{align}
      	V^{lo,k}(z)= \begin{cases}
      		\left(\begin{array}{cc}
      			1 & -f_{kj}e^{-2it\theta}  \\
      			0 &1
      		\end{array} \right),\, z\in \Sigma_{kj}, j=1,3,\\
      		\left(\begin{array}{cc}
      			1 & 0 \\
      			 f_{kj}e^{2it\theta}&1
      		\end{array} \right),\, z\in \Sigma_{kj},j=2,4.
      	\end{cases}
      \end{align}
\item Asymptotic behavior:  	$M^{lo,k}(z)=I+\mathcal{O}(z^{-1}),	\quad  z \to  \infty.$

          \end{itemize}

          Here we take the construction of the model $M^{lo,1} (z)$ as an example, and the other cases can be considered similarly.
          In order to motivate this model, we define the rescaled local variable
          \begin{align}
          s =s(z;\xi_1)=\sqrt{2t\theta''(\xi_1)} (z-\xi_1),\label{scle}
          \end{align}
          and we choose the branches of the logarithm with $-\pi < \arg s <\pi$.
          %for $\xi>1$ and $0<\arg s<2\pi$ for $\xi<-1$.
           We set
           \begin{align}
           	&r_{\xi_1}=-r(\xi_1)T_1(\xi_1)^2 \exp \big( 2it \theta(\xi_1)  +  iv(\xi_1)  \log\left(2t\theta''(\xi_1) \right) \big),\label{rdd}\\
            &\tilde{\rho}=\tilde{\rho}(\xi_1)=\sqrt{2t\theta''(\xi_1) }\rho,\quad \tilde{\mathcal{U}}_{\xi_1}= \{s\in \mathbb{C}: |s|=\tilde{\rho}\},
        \end{align}
          where $|r_{\xi_1}|=|r(\xi_1)|$.
          %, $\tilde{\rho}=\tilde{\rho}(\xi_1)=\sqrt{2t\theta''(\xi_1) }\rho$ and
          % $\tilde{\mathcal{U}}_{\xi}= \{s\in \mathbb{C}: |s|=\tilde{\rho}\}$.
         Making scaling transformation
         $$ M^{pc,1} (s;r_{\xi_1}) =  N_1  M^{lo,1}(z), $$
%where $s$ satisfies the relations  (\ref{scle}).
where $s$ and $	r_{\xi_1}$  satisfy the relations  (\ref{scle}) and (\ref{rdd}).
            After changing the variable we obtain the parabolic cylinder model problems.

             \noindent\textbf{RHP7}.  Find a matrix-valued function  $M^{pc,1}(s;r_{\xi_1})=M^{pc,1}(s;r_{\xi_1}, x,t)$ such that
           \begin{itemize}
           	\item Analyticity: $M^{pc,1}(s;r_{\xi_1})$ is analytic in $\mathbb{C}\backslash \Sigma^{pc,1}$ where $\Sigma^{pc,1}= \cup_{j=1}^4  \Sigma^{pc,1}_j$ with
            \begin{equation*}
                \Sigma^{pc,1}_1= \mathbb{R}^+ e^{(\pi-\varphi)i} \cap \tilde{\mathcal{U}}_{\xi_1},\;\Sigma^{pc,1}_2= \mathbb{R}^+ e^{\varphi i} \cap \tilde{\mathcal{U}}_{\xi_1},\;\Sigma^{pc,1}_3=\overline{\Sigma}^{pc,1}_2, \;\Sigma^{pc,1}_4=\overline{\Sigma}^{pc,1}_1.
            \end{equation*}
           	\item Jump condition:
           	\begin{align*}
           		M^{pc,1}_+(s;r_{\xi_1})=M^{pc,1}_-(s;r_{\xi_1})V^{pc,1}(s;r_{\xi_1}), \quad s\in \Sigma^{pc,1},
           	\end{align*}
           	where
               \begin{align}
           		V^{pc,1}(s;r_{\xi_1})= \begin{cases}
                    \left(\begin{array}{cc}
                        1 & 		\frac{-\bar{r}_{\xi_1}}{1-|r_{\xi_1}|^2}s^{2iv(\xi_1)}e^{- is^2/2} \\
                        0&1
                    \end{array} \right),\, s\in \Sigma^{pc,1}_1,\\
           			\left(\begin{array}{cc}
           				1 & 0 \\
           				r_{\xi_1} s^{- 2iv(\xi_1) }e^{is^2/2}&1
           			\end{array} \right),\, s\in \Sigma^{pc,1}_2,\\
           			\left(\begin{array}{cc}
           				1 & -\overline{r}_{\xi_1}s^{2iv(\xi_1)}e^{-is^2/2}  \\
           				0 &1
           			\end{array} \right),\, s\in \Sigma^{pc,1}_3,\\
           		    \left(\begin{array}{cc}
           			1 & 0  \\
           			\frac{r_{\xi_1}}{1-|r_{\xi_1}|^2}s^{-2iv(\xi_1) }e^{is^2/2} &1
           		\end{array} \right),\, s\in \Sigma^{pc,1}_4.
           		\end{cases}
           	\end{align}
  %          with
  %          \begin{align}
  %              &\Sigma^{(pc,\xi_1)}_1=\mathbb{R}^+ e^{(\pi-\varphi)i},\;\Sigma^{(pc,\xi_1)}_2= \mathbb{R}^+ e^{\varphi i},\;
  %               \Sigma^{(pc,\xi_1)}_3=\overline{\Sigma}^{(pc,\xi_1)}_2,\; \Sigma^{(pc,\xi_1)}_4=\overline{\Sigma}^{(pc,\xi_1)}_1,\nonumber\\
  %              &\Sigma^{(pc,\xi_2)}_1= \mathbb{R}^+ e^{\varphi i},\;
  %              \Sigma^{(pc,\xi_2)}_2=\mathbb{R}^+ e^{(\pi-\varphi)i},\; \Sigma^{(pc,\xi_2)}_3=\mathbb{R}^+ e^{(\pi+\varphi)i},\;
  %              \Sigma^{(pc,\xi_2)}_4=\mathbb{R}^+ e^{(2\pi-\varphi) i}.\nonumber
  %          \end{align}

      	\item Asymptotic behavior:  	
           	\begin{align*}
           		M^{pc,1}(s;r_{\xi_1})=I+\frac{M^{pc,1}_1(x,t)}{s}+
                \mathcal{O}\left( s^{-2}\right),	\quad  s \to  \infty,
           	\end{align*}
             where $M^{pc,1}_1(x,t)$ is the coefficient of the term $s^{-1}$ in $M^{pc,1}(s;r_{\xi_1})$.
           \end{itemize}

        Further, $M^{pc,1}(s;r_{\xi_1})$ can be constructed by  the following transformation:
           \begin{equation}
           	M^{pc,1}(s;r_{\xi_1})= \Psi(s;r_{\xi_1}) P(s;r_{\xi_1})  e^{is^2 \sigma_3 /4}s^{-iv(\xi_1)\sigma_3}, \label{925}
           \end{equation}
          where
          \begin{equation}\nonumber
          	P(s;r_{\xi_1}) =\begin{cases}
          		\left(\begin{array}{cc}
          			1&0\\
          			 - r_{\xi_1} & 1
          		\end{array} \right),\quad \arg s \in (0,\varphi),\\
          		\left(\begin{array}{cc}
          		1& - \bar{r}_{\xi_1}\\
          		0 & 1
          	\end{array} \right),\quad \arg s \in (0,-\varphi),\\
             	\left(\begin{array}{cc}
             	1& \frac{\bar{r}_{\xi_1}}{1-|r_{\xi_1}|^2}\\
             	0 & 1
             \end{array} \right),\quad \arg s \in (\pi - \varphi,\pi),\\
         \left(\begin{array}{cc}
         	1& 0\\
         	\frac{r_{\xi_1}}{1-|r_{\xi_1}|^2} & 1
         \end{array} \right),\quad \arg s \in (-\pi ,-\pi+\varphi),
	\end{cases}
          \end{equation}
  and $\Psi(s,r_{\xi_1})$ satisfies the following RH problem:

           \noindent\textbf{RHP8}.  Find a matrix-valued function  $\Psi(s;r_{\xi_1})$ which has the following properties:
          \begin{itemize}
          	\item Analyticity: $\Psi(s;r_{\xi_1})$ is analytic in $\mathbb{C}\backslash  \mathbb{R}$.
          	\item Jump condition:
           \begin{align*}
              \Psi_{+}(s;r_{\xi_1})=\Psi_{-}(s;r_{\xi_1})V^{\Psi}(0), \quad s\in \mathbb{R},
           \end{align*}
          	where \begin{align}
          		V^{\Psi}(0)=
          			\left(\begin{array}{cc}
          				1-|r_{\xi_1}|^2 &  \bar{r}_{\xi_1} \\
          				-r_{\xi_1} &1
          			\end{array} \right).
          	\end{align}

          	\item Asymptotic behavior:  	
          	\begin{align}
          		\Psi(s;r_{\xi_1})=\left(I+\frac{M^{pc,1}_1}{s}
          		+\mathcal{O}(s^{-2}) \right)  s^{iv(\xi_1)\sigma_3}e^{-is^2\sigma_3/4},	\quad  s \to  \infty. \nonumber
          	\end{align}

          \end{itemize}
  In a similar way \cite{fNLS,LJQ},    this RH problem can be changed into a Weber equation to  obtain its solution  $\Psi(s;r_{\xi_1})=(\Psi_{jl})_{j,l=1}^2$
in terms of   parabolic cylinder functions.
Then using (\ref{925}), we get the asymptotic solution of the RHP7. Further,  we can obtain the  asymptotic solution  of $M^{pc,2}(s_;r_{\xi_2})$ similarly.  The asymptotic results are shown in the following formula
         	\begin{align*}
           		M^{pc,k}(s_;r_{\xi_k})=I+\frac{M^{pc,k}_1}{s}+
           		\mathcal{O}(s^{-2}),	\quad  s \to  \infty,
           	\end{align*}
where
    \begin{align}
        M^{pc,k}_1=
            \left( \begin{array}{cc}
                0 &  -i\epsilon_k \beta_{12}^{ \xi_k}\\
              i \epsilon_k\beta_{21}^{ \xi_k}&  0
            \end{array}\right),\label{pc1}
    \end{align}
with
          \begin{align}
            &\beta_{12}^{ \xi_k} =\frac{(2\pi)^{
                \frac{1}{2}
            } e^{\frac{(2k-1) \pi i }{4}}   e^{\frac{ -\pi \epsilon_k v(\xi_k) }{2}}}{-r_{\xi_k} \Gamma(-i\epsilon_k v(\xi_k))}, \  \; \beta_{12}^{ \xi_k}\beta_{21}^{ \xi_k}=  v(\xi_k),\; \  \epsilon_k = (-1)^{k+1},\nonumber\\
            & \arg \beta_{12}^{ \xi_k}= \frac{(2k-1)\pi}{4}  -\arg r_{\xi_k} + \arg \Gamma(i\epsilon_k v(\xi_k)).\nonumber
          \end{align}
\begin{remark}

From calculations above, we see that this asymptotic result (\ref{pc1})  is independent of the opening fixed  angle $\phi(\xi)$ satisfying $|\phi(\xi)|<\pi/4$.

\end{remark}
\begin{remark} Noting that for the original RHP1,  $M(z)$   admits the circular symmetry
\begin{align}
 M(z)=z^{-1}M(z^{-1})\sigma_1. \label{ere}
\end{align}
However for the local model RHP5, the matrix function $M^{lo}(z)$ no longer admits  such  circular symmetry as  (\ref{ere}).
So we cannot obtain the local model $M^{pc,1}(z)$  at  the phase  point $\xi_1$ through the local model $M^{pc,2}(z)$ at the  symmetric
 phase  point $\xi_2$
by means of circular symmetry. Therefore, we  have to solve   two local  models separately  since
  the local model cannot keep  the circular symmetry   like  the  symmetry with respect to the imaginary axis $i\mathbb{R}$.

\end{remark}

          \subsection{The small-norm RH problem}\label{emod}

           Define  the error function $E(z)$,  which satisfies the following RH problem.
\begin{equation}
      E(z)= \begin{cases}
        M^{(2)}_{RHP}(z) (M^{out})^{-1}(z), \quad z\in \mathbb{C} \backslash  \mathcal{U}_{\xi },\\[3pt]
        M^{(2)}_{RHP}(z) (M^{lo} (z))^{-1}  (M^{out})^{-1}(z), \quad z \in \mathcal{U}_{\xi}.
       \end{cases}
\end{equation}

\begin{figure}
    \begin{center}
     \subfigure[The \ region $ \xi>1$]{
        \begin{tikzpicture}[scale=0.9]

%\draw [dashed  ] (3.3,2.6) to [out=-75,in=90] (3.75,0);
%\draw [dashed  ] (3.3,-2.6) to [out=75,in=-95] (3.75,0);
            \draw[dashed,->](-4.2,0)--(6,0)node[black,right]{Re$z$};
            \foreach \x [count=\p] in {1,...,5} {
        \node[shape=circle,fill=black, scale=0.15] (\p) at (\x*30:2) {};
        \draw[line width=0.5] (\x*30:2) circle (0.2);
        \draw[ ->] (\x*30:2)+(0.2,0) arc (0:90:0.2);
        \node[shape=circle,fill=black, scale=0.15] (\p) at (180+\x*30:2) {};
        \draw[line width=0.5] (180+\x*30:2) circle (0.2);
        \draw[ ->] (180+\x*30:2)+(-0.2,0) arc (180:90:0.2);
        };

            \node[below] at (-0.1,0) {\footnotesize $0$};
            \node[below] at (-2.3,0) {\footnotesize $-1$};
            \node[right] at (1.2,1.8) {$z_j$};
            \node[right] at (1.2,-1.8) {$\bar{z}_j$};

            \draw [dashed, gray] (0,0) circle (2);
            \draw [red,line width=0.5] (1.65,0) circle (0.35);
            \node[shape=circle,fill=black, scale=0.2]  at (1.7,0) {};
            \node[below] at (1.73,0) {\tiny $\xi_1$};

            \draw [red,line width=0.5] (3.8,0) circle (0.35);
            \node[shape=circle,fill=black, scale=0.2]  at (3.8,0) {};
            \node[below] at (3.8,0) {\tiny $\xi_2$};

                     \draw  (-4, -1.5)--(0.8,0.35);
                     \draw  (2,-0.14 )--(2.7,-0.5);
                     \draw  (0.8,0.35 )--(1.3,0.12);
                     \draw [-latex] (2,-0.14 )--(2.2,-0.265);
                     \draw  (2.7,-0.5)--(3.5, -0.09);
                     \draw (2.7,0.5)--(3.52, 0.12);
                    \draw  (4.1,0.09)--(5.5, 0.75);
                    \draw [-latex]  (4.1,0.09)--(4.7, 0.373);
                    %\draw [-latex] (2.7,-0.5)--(4.6, 0.35);

                    \draw (0.8,-0.35)--(1.3, -0.12);
                    \draw [-latex] (2,0.14 )--(2.2,0.265);
                    \draw  (2,0.14 )--(2.7,0.5);
                    \draw (4.1,-0.09)--(5.5, -0.818);
                    \draw [-latex] (4.1, -0.09)--(4.6, -0.35);

\draw [-latex] (-4, -1.5)--(-1.25,-0.44 );
\draw[](0.8,-0.35)--(0.8,0.35);
\draw[-latex](0.8,0)--(0.8,0.16);
\draw[-latex](0.8,-0.35)--(0.8,-0.1);
\draw[](2.7,-0.5)--(2.7,0.5);
\draw[-latex](2.7,0.5)--(2.7,0.2);
\draw[-latex](2.7,0)--(2.7,-0.2);

\draw (-4, 1.5)--(0.8,-0.35 );
\draw [-latex] (-4,  1.5)--(-1.25, 0.44 );
                             % \node[blue,thick]  at (4.5,1.2) {\bf  $+$};
                             %    \node[red,thick]  at (4.5,-1.2) {\bf  $-$};
                             %      \node[red,thick]  at (2.6,1.2) {\bf  $-$};
                             %    \node[blue,thick]  at (2.6,-1.2) {\bf  $+$};
                             %    \node[red,thick]  at (0,1.2) {\small\bf $-$};
                             %    \node[blue,thick]  at (0,-1.2) {\small\bf $+$};
                             %     \node[blue,thick]  at (0.8,0.5) {\tiny\bf $+$};
                             %        \node[red,thick]  at (0.8,-0.5) {\tiny\bf $-$};
                             %        \node[blue,thick]  at (-3,-1.3) {\bf  $+$};
                             %        \node[red,thick]  at (-3,1.3) {\bf  $-$};

                                     \node  at (5.6,1) {\footnotesize $\Sigma_{21}$};
                                     \node  at (5.6,-1) {\footnotesize $\Sigma_{24}$};
                                      \node  at (3.3,0.6) {\footnotesize $\Sigma_{22}$};
                                     \node  at (3.3,-0.6) {\footnotesize $\Sigma_{23}$};
                                      \node  at (2.2,0.6) {\footnotesize $\Sigma_{12}$};
                                     \node  at (2.3,-0.6) {\footnotesize $\Sigma_{13}$};
                                        \node  at (-3.6,1) {\footnotesize $\Sigma_{02}$};
                                     \node  at (-3.6,-1) {\footnotesize $\Sigma_{03}$};

                                       \node  at (1.2,0.4) {\tiny $\Sigma_{11}$};
                                         \node  at (0.35,0.4) {\tiny $\Sigma_{01}$};
                                         \node  at (1.23,-0.4) {\tiny $\Sigma_{14}$};
                                         \node  at (0.35,-0.4) {\tiny $\Sigma_{04}$};
        \end{tikzpicture}
     }%

     \subfigure[The \ region  $ \xi<-1$]{
        \begin{tikzpicture}[scale=0.9]
%\draw [dashed  ] (-3.3,2.6) to [out=75,in=90] (-3.8,0);
%\draw [dashed  ] (-3.3,-2.6) to [out=-75,in=-95] (-3.8,0);

            \draw[dashed,->](-6,0)--(4.5,0)node[black,right]{Re$z$};
            \foreach \x [count=\p] in {1,...,5} {
                 \node[shape=circle,fill=black, scale=0.15] (\p) at (\x*30:2) {};
                 \draw[line width=0.5] (\x*30:2) circle (0.2);
                 \draw[ ->] (\x*30:2)+(-0.2,0) arc (180:90:0.2);
                 \node[shape=circle,fill=black, scale=0.15] (\p) at (180+\x*30:2) {};
                 \draw[line width=0.5] (180+\x*30:2) circle (0.2);
                 \draw[ ->] (180+\x*30:2)+(0.2,0) arc (0:90:0.2);
                 };

            \node[below] at (2.1,0.05) {\footnotesize $1$};
            \node[below] at (0.1,0) {\footnotesize $0$};
            \node[right] at (1.2,1.8) {$z_j$};
            \node[right] at (1.2,-1.8) {$\bar{z}_j$};
           \draw [dashed, gray] (0,0) circle (2);
%           \draw [dashed] (-0.8,0) circle [radius=0.8];

%                             \draw [ red] (4, -1.5)--(-0.8,0.35);
%                             \draw [ red] (-0.8,0.35 )--(-2.7,-0.5);
%                             \draw [-latex,red] (-2.7,-0.5)--(-2,-0.175);
%                       \draw [ red] (-2.7,-0.5)--(-5.5, 0.75);
%                       \draw [-latex, red] (-5.5, 0.75)--(-4.6, 0.35);
%\draw [-latex,red] (-0.8,0.35 )--(1.25,-0.44 );
%\draw[](-0.8,-0.35)--(-0.8,0.35);
%\draw[->](-0.8,0)--(-0.8,0.16);
%\draw[->](-0.8,-0.35)--(0-.8,-0.1);
%\draw[](-2.7,-0.5)--(-2.7,0.5);
%\draw[->](-2.7,0.5)--(-2.7,0.2);
%\draw[->](-2.7,0)--(-2.7,-0.2);
%                         \draw [ blue] (-0.8,-0.35)--(-2.7, 0.5);
%                            \draw [-latex,blue] (-2.7, 0.5)--(-2,0.175);
%                            \draw [blue] (-2.7,0.5)--(-5.5, -0.75);
%                            \draw [blue] (4, 1.5)--(-0.8,-0.35 );
%                             \draw [-latex,blue] (-5.5, -0.75)--(-4.6, -0.35);
%                               \draw [-latex,blue] (-0.8,-0.35 )--(1.25, 0.44 );

\draw  (4, -1.5)--(-0.8,0.35);
\draw[-latex] (-0.8,0.35)--(1.25,-0.44 );

\draw (4, 1.5)--(-0.8,-0.35 );
\draw [-latex] (-0.8,-0.35 )--(1.25, 0.44 );

\draw  (-4.1,0.09)--(-5.5, 0.75);
\draw [-latex]  (-5.5, 0.75)--(-4.7, 0.373);

\draw (-4.1,-0.09)--(-5.5, -0.818);
\draw [-latex] (-5.5, -0.818)--(-4.6, -0.35);

\draw  (-2,-0.14 )--(-2.7,-0.5);
\draw [-latex] (-2.7,-0.5)--(-2.2,-0.265);
\draw  (-0.8,0.35 )--(-1.3,0.12);

\draw  (-2.7,-0.5)--(-3.5, -0.09);
\draw (-2.7,0.5)--(-3.52, 0.12);
\draw (-0.8,-0.35)--(-1.3, -0.12);
\draw [-latex] (-2.7,0.5)--(-2.2,0.265);

\draw  (-2,0.14 )--(-2.7,0.5);

\draw[](-0.8,-0.35)--(-0.8,0.35);
\draw[-latex](-0.8,0)--(-0.8,0.16);
\draw[-latex](-0.8,-0.35)--(-0.8,-0.1);
\draw[](-2.7,-0.5)--(-2.7,0.5);
\draw[-latex](-2.7,0.5)--(-2.7,0.2);
\draw[-latex](-2.7,0)--(-2.7,-0.2);

                          %    \node[red,thick]  at (-4.5,1.2) {\bf  $-$};
                         %        \node[blue,thick]  at (-4.5,-1.2) {\bf  $+$};
                           %        \node[blue,thick]  at (2.6,1.2) {\bf  $+$};
                          %       \node[red,thick]  at (2.6,-1.2) {\bf  $-$};
                          %       \node[blue,thick]  at (0,1.2) {\small\bf $+$};
                          %       \node[red,thick]  at (0,-1.2) {\small\bf $-$};
                          %        \node[red,thick]  at (-0.8,0.5) {\tiny\bf $-$};
                         %            \node[blue,thick]  at (-0.8,-0.5) {\tiny\bf $+$};
                          %           \node[red,thick]  at (-3,-1.2) {\bf  $-$};
                          %           \node[blue,thick]  at (-3,1.2) {\bf  $+$};

                          \draw [red,line width=0.5] (-1.65,0) circle (0.35);
                          \node[shape=circle,fill=black, scale=0.2]  at (-1.7,0) {};
                          \node[below] at (-1.68,0) {\tiny $\xi_1$};

                          \draw [red,line width=0.5] (-3.8,0) circle (0.35);
                          \node[shape=circle,fill=black, scale=0.2]  at (-3.8,0) {};
                          \node[below] at (-3.8,0) {\tiny $\xi_2$};

                                %     \node[below] at (-3.9,0) {\footnotesize $\xi_2$};
                                %     \node[below] at (-1.73,0) {\footnotesize $\xi_1$};
                                     \node  at (-5.6,1) {\footnotesize $\Sigma_{22}$};
                                     \node  at (-5.6,-1) {\footnotesize $\Sigma_{23}$};
                                      \node  at (-3.2,0.6) {\footnotesize $\Sigma_{21}$};
                                     \node  at (-3.2,-0.6) {\footnotesize $\Sigma_{24}$};
                                      \node  at (-2.3,0.6) {\footnotesize $\Sigma_{11}$};
                                     \node  at (-2.3,-0.65) {\footnotesize $\Sigma_{14}$};
                                        \node  at (3.6,1) {\footnotesize $\Sigma_{01}$};
                                     \node  at (3.6,-1) {\footnotesize $\Sigma_{04}$};
                                       \node  at (-1.2,0.4) {\tiny $\Sigma_{12}$};
                                         \node  at (-0.35,0.4) {\tiny $\Sigma_{02}$};
                                         \node  at (-1.2,-0.4) {\tiny $\Sigma_{13}$};
                                         \node  at (-0.35,-0.4) {\tiny $\Sigma_{03}$};
        \end{tikzpicture}
     }%
    \end{center}
    \caption{\small{The jump contour $\Sigma^{E}$ for the $E(z;\xi)$. The blue circles are $\mathcal{U}_{\xi_k}(k=1,2)$.} }
    \label{figE}
\end{figure}
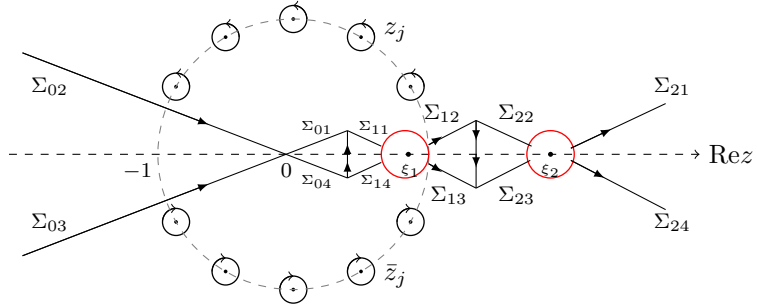
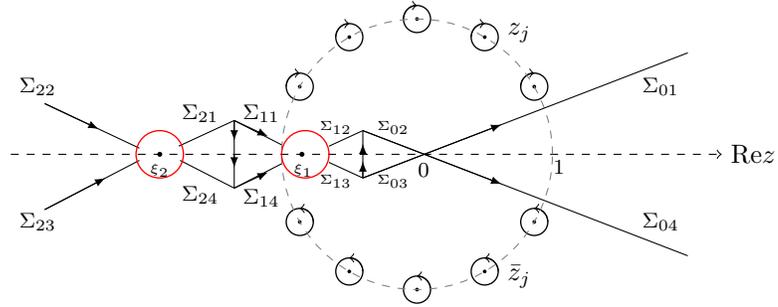

           \noindent\textbf{RHP9}.
            Find a matrix-valued function  $E(z)$ with the properties as follows:
          \begin{itemize}
          	\item Analyticity: $E(z)$ is analytic in $\mathbb{C}\backslash  \Sigma^{E}$, \  where  $\Sigma^{E}= \partial \mathcal{U}_{\xi} \cup \left( \Sigma^{(2)} \setminus \mathcal{U}_{\xi}\right).$
          	\item Jump condition:
          	\begin{align*}
          		E_{+}(z)=E_-(z)V^{E}(z), \quad z\in \Sigma^{E},
          	\end{align*}
          	where the jump matrix is given by
   \begin{align}\label{VE}
                V^{E}(z)= \begin{cases}
                    M^{out}(z)V^{(2)}(z)(M^{out})^{-1}(z),\quad z \in \Sigma^{(2)} \backslash \mathcal{U}_{\xi },\\[3pt]
                    M^{out}(z) M^{lo}(z) (M^{out})^{-1}(z), \quad z \in \partial \mathcal{U}_{\xi}.
                \end{cases}
            \end{align}
        See Figure \ref{figE}.
          	\item Asymptotic behavior:  $ E(z)=I+\mathcal{O}(z^{-1}),	\quad  z\to  \infty.$

          \end{itemize}

        Then we estimate the jump matrix of $E(z)$.
        \begin{proposition}
            \begin{equation}
                |V^{E}(z)-I|=\begin{cases}
                    \mathcal{O}(e^{-c_\rho t}),\quad z\in \Sigma^{(2)} \backslash   \mathcal{U}_{\xi},   \\
                    \mathcal{O}(t^{-1/2}), \quad z\in \partial \mathcal{U}_{\xi }.
                \end{cases}
            \end{equation}
        \end{proposition}
        \begin{proof}
            Since $M^{out}$ is bounded, we have for $z \in  \Sigma^{(2)} \backslash  \mathcal{U}_{\xi }   $,
          \begin{equation}
              |V^{E}(z)-I| \lesssim  | V^{(2)}(z)-I| \lesssim e^{-c_\rho t},
          \end{equation}
          and for $z\in \partial \mathcal{U}_{\xi_k}$,
          \begin{equation}
            |V^{E}(z)-I| \lesssim |M^{lo,k}(z)-I| \lesssim t^{-1/2}.\nonumber
        \end{equation}

        \end{proof}
        According to   Beal-Cofiman theory, we decompose $V^{E}(z)=(b_-)^{-1}b_+, $ with
        \begin{align}
            &   b_-=I,\;  b_+=V^{E}, \ \
         w_-=0, \; w_+ =V^{E}-I.\nonumber
        \end{align}
       Let $C_{w }$ be an integral operator: $L^{2}(\Sigma^{E}) \to L^{2}(\Sigma^{E})$
        \begin{align}
           C_{w }f=C_-\left( f \left( V^{E}-I \right) \right),
        \end{align}
         where $C_-$ is the Cauchy projection operator on $\Sigma^{E}$ defined by
         \begin{equation}
            C_-f(z)=\lim_{z'\to z  \in \Sigma^{E}} \frac{1}{2\pi i} \int_{\Sigma^{E}} \frac{f(s)}{s-z'}\, \mathrm{d}s.
         \end{equation}
  Then     $||C_{w }||_{L^2(\Sigma^{E})}=\mathcal{O}(t^{-1/2})$   and  the {RHP9} exists a unique solution
         \begin{equation}
            E(z)=I +\frac{1}{2\pi i} \int_{\Sigma^{E}} \frac{\mu (s) \left( V^{E}-I\right)}{s-z}\, \mathrm{d}s, \nonumber
         \end{equation}
        where $\mu  \in L^2\left( \Sigma^{E}\right)$  satisfies $\left(I-C_{w }\right)\mu =I$.
        Moreover, we can derive
       \begin{equation}
          ||\mu -I||_{L^2(\Sigma^{E})}=\mathcal{O}(t^{-1/2}), \quad ||V^{E}-I||_{L^2(\Sigma^{E})}=\mathcal{O}(t^{-1/2}).
       \end{equation}

         In order to reconstruct the solution $q(x,t)$, we need the large-$z$ behavior of $E(z)$. We consider
        \begin{equation}
            E(z)=I +\frac{E_1}{z} +\mathcal{O}\left(z^{-1}\right),
        \end{equation}
     where
     \begin{equation}\label{E1}
        E_1=-\frac{1}{2\pi i} \int_{\Sigma^{E}} \mu (s)\left( V^{E}(s)-I \right)\, \mathrm{d}s.
     \end{equation}
        Specifically, $E_1$ can be expressed as the following formula after combining (\ref{VE}) to (\ref{E1}).
        \begin{proposition} \label{sum5}
            \begin{equation}
            E_1(x,t)= \sum_{k=1}^2 \frac{M^{out}M^{pc,k}_1(s;r_{\xi_k})(M^{out})^{-1}}{\sqrt{  2t\theta''(\xi_k) \epsilon_k}} + \mathcal{O}(t^{-1}).
        \end{equation}
        \end{proposition}

        \section{ Contribution from  a  pure $\bar{\partial}$-problem }
             Now we turn to consider the pure $\bar{\partial}$-problem $M^{(3)}$ and its large-$z$ behavior.
             From the analysis at the beginning of Section 6, we define the function
            \begin{equation}\label{m3m2m2rhp}
                M^{(3)}(z)=M^{(2)}(z)\left(M^{(2)}_{RHP}(z)\right)^{-1},
            \end{equation}
            which  satisfies the following  RH problem.

            \noindent\textbf{$\bar{\partial}$ Problem}.  Find a matrix-valued function  $M^{(3)}(z)$ which satisfies
            \begin{itemize}
                \item Analyticity: $M^{(3)}(z)$ is continuous and has sectionally continuous first partial derivatives in $\mathbb{C}  \backslash \left(\mathbb{R} \cup \Sigma^{(2)}\right)$.
                \item Asymptotic behavior:  $M^{(3)}(z)=I+\mathcal{O}(z^{-1})$,   $z\to  \infty$; 	
                \item $ \bar{\partial}$-derivative:   $\bar{\partial} M^{(3)}(z) = M^{(3)}(z) W^{(3)}(z), \ \ z\in \mathbb{C},$
                where
                $$W^{(3)}(z):=M^{(2)}_{RHP}(z)  \bar{\partial} R^{(2)}(z) \left(M^{(2)}_{RHP}(z)\right)^{-1}$$
                 and $ \bar{\partial}  R^{(2)}(z)$ has been given in (\ref{parR2}).
            \end{itemize}
            \begin{proof}
                The  asymptotic behavior and $ \bar{\partial}$-derivative of $M^{(3)}$ can be obtained by the properties of $M^{(2)}$ and $M^{(2)}_{RHP}(z)$. Since $M^{(2)}$ and $M^{(2)}_{RHP}(z)$ have the same jump, we find $M^{(3)}$ has  no jump.
                With reference to the proof of $ \bar{\partial}$ Problem $6.1$ in the paper \cite{CJ}, we can show $M^{(3)}$ has no isolated singularities.

            \end{proof}
            Then the solution of $M^{(3)}$ can be given by the following integral equation
            \begin{equation} \label{Im3}
                M^{(3)}(z)=I-\frac{1}{\pi}  \iint_\mathbb{C} \frac{ M^{(3)}(s) W^{(3)}(s)}{s-z} \, \mathrm{d}A(s),
            \end{equation}
             where $\mathrm{d} A(s)$ is Lebesgue measure on the $\mathbb{C}$. Let $S$ denotes the solid Cauchy operator,
            \begin{equation}\label{Sm3}
                Sf(z)=\frac{1}{\pi} \iint \frac{f(s)W^{(3)}(s)}{s-z} \, \mathrm{d}A(s).
            \end{equation}
            So the equation (\ref{Im3}) can be written in the operator form
           \begin{equation}\label{Me}
            (I-S) M^{(3)}(z)=I.
           \end{equation}
           To show the existence and uniqueness of the solution of $M^{(3)}(z)$, we derive a proposition.
           \begin{proposition}\label{ests}
             For sufficiently large $t$, the operator (\ref{Sm3}) satisfies the estimation for $\left|\xi \right|>1$ and $\xi=\mathcal{O}(1)$,
            \begin{equation}
               ||S||_{L^\infty \to L^\infty} \lesssim t^{-1/4}. \label{22d}
            \end{equation}
            Hence, the resolvent operators $(I-S)^{-1}$ exists and then the operator equation (\ref{Me}) has  a unique solution.

           \end{proposition}
           \begin{proof}
               We just give a proof for the case $z\in \Omega_{21}$ and other cases can be proved by similar methods.
             Let $f\in L^{\infty}(\Omega_{21})$, $s=u+iv$ with $u>\xi_2$ and $u>v$  and $z=\alpha+i\eta$. Then we obtain
            \begin{align}\label{Sf}
                   |S(f)|
                   %\lesssim ||f||_{L^\infty}\iint_{\Omega_{21}} \frac{|W^{(3)}|}{|s-z|} \, \mathrm{d} A(s) \\
                   \lesssim ||f||_{L^\infty} ||M^{(2)}_{RHP}||_{L^\infty}^2 \iint_{\Omega_{21}} \frac{|\bar{\partial} R_{22}|e^{-\re(2it\theta)}}{|s-z|}\, \mathrm{d} A(s) \lesssim I_1+I_2,
            \end{align}
           where
           \begin{align}
             I_1= \iint_{\Omega_{21}} \frac{|r'(u)|e^{-\re(2it\theta)}}{|s-z|}\, \mathrm{d} A(s), \quad I_2= \iint_{\Omega_{21}} \frac{|s-\xi_2|^{-1/2}e^{-\re(2it\theta)}}{|s-z|}\, \mathrm{d} A(s).\nonumber
           \end{align}
           According to the asymptotic behavior (\ref{G}), we find
          \begin{equation}
            e^{-\re(2it\theta)}=e^{2t\theta''(\xi_2)(u-\xi_2)v},\quad t \to \infty.
          \end{equation}
          Then, we know
          \begin{align*}
            I_1 &\le \int_{0}^{\infty} e^{2t\theta''(\xi_2)v^2} |r'(u)| |(s-z)^{-1}|  \, \mathrm{d}v \\
            &\le \int_{0}^{\infty} e^{2t\theta''(\xi_2)v^2} ||r'||_{L^2(v+\xi_2,\infty)}||(s-z)^{-1}||_{L^2(v+\xi_2,\infty)} \, \mathrm{d}v,
         \end{align*}
           where
           \begin{align*}
            ||(s-z)^{-1}||_{L^2(v+\xi_2,\infty)} & \le \int_{v+\xi_2}^\infty \frac{1}{|s-z|^2}\mathrm{d}u \le \int_{-\infty}^\infty \frac{1}{(u-\alpha)^2 + (v-\eta)^2} \mathrm{d} u \\
            &= \frac{1}{|v-\eta|} \int_{-\infty}^{\infty} \frac{1}{1+y^2} \mathrm{d} y = \frac{\pi}{|v-\eta|}
           \end{align*}
          and $y = \frac{u-\alpha}{v-\eta}$.
          Thus,  $I_1$ can be estimated by
          \begin{align}
             I_1  \lesssim t^{-1/4}. \label{wee}
          \end{align}

          Next, we consider the estimation of $I_2$
          \begin{align*}
            I_2 &\le  \int_{0}^{\infty} e^{2t\theta''(\xi_2)v^2}   |s-\xi_2|^{-1/2} |s-z|^{-1} \mathrm{d} v\\
                &\le  \int_{0}^{\infty} e^{2t\theta''(\xi_2)v^2}    ||(|s-\xi_2|)^{-1/2}||_{L^p(v+\xi_2,\infty)}    ||(s-z)^{-1}||_{L^q(v+\xi_2,\infty)}  \mathrm{d} v
          \end{align*}
         where  $p>2$ and $p^{-1}+q^{-1}=1$. It is easy to derive
       %  \begin{align}\label{ests2}
       %     ||(|s-\xi_2|)^{-1/2}||_{L^p(v+\xi_2,\infty)} & = \left( \int_{v+\xi_2}^\infty  \frac{1}{|u-\xi_2+iv|^{p/2}} \mathrm{d} u     \right)^{1/p} \lesssim v^{1/p-1/2}
       %    \end{align}
         \begin{align}\label{ests2}
            ||(|s-\xi_2|)^{-1/2}||_{L^p(v+\xi_2,\infty)} & = \left( \int_{v+\xi_2}^\infty  \frac{1}{|u-\xi_2+iv|^{p/2}} \mathrm{d} u     \right)^{1/p}  = \left( \int_{v}^\infty \frac{1}{\left(u^2 + v^2\right)^{p/4}}  \mathrm{d} u \right)^{1/p} \nonumber\\
            % = \left( \int_{v+\xi_2}^\infty \frac{1}{\left((u-\xi_2)^2 + v^2\right)^{p/4}}  \mathrm{d} u \right)^{1/p} \\
            &= v^{1/p-1/2} \left( \int_{1}^\infty \frac{1}{\left(1+x^2\right)^{p/4}}    \right)   \lesssim v^{1/p-1/2}.\nonumber
         \end{align}
         Similar to the above estimates, it can be shown that
         \begin{align}
           ||(s-z)^{-1}||_{L^q(v+\xi_2,\infty)} \lesssim |v-\eta|^{1/q-1}.
          \end{align}
          Then we have $|I_2| \lesssim t^{-1/4}$, which together with (\ref{wee})  and (\ref{Sf}) gives (\ref{22d}).

            \end{proof}

        In order to discuss the long-time asymptotic behavior of the Cauchy  problem (\ref{q})-(\ref{inq}), we need to consider the asymptotic behavior of $M^{(3)}$ as $t\to \infty$. We first make the Taylor expansion of $M^{(3)}(z)$ at infty
        \begin{equation}
            M^{(3)}(z)=I + \frac{M^{(3)}_1(x,t) }{z} +\mathcal{O}(z^{-2}),
        \end{equation}
         where $M^{(3)}_1(x,t)=\frac{1}{\pi} \iint_{\mathbb{C}} M^{(3)}(s)W^{(3)}(s)\, \mathrm{d}A(s)$ and has the following property:
        \begin{proposition}
           For $|\xi|>1(\xi=\mathcal{O}(1))$,  $|M^{(3)}_1(x,t)| \lesssim t^{-3/4}$ as $t\to\infty$.
        \end{proposition}
        \begin{proof}
            We only give details for the case $z \in \Omega_{21}$, and other cases can be handled with the same method.
            Since $M^{(2)}_{RHP}$ is bounded in the entire complex plane except  the poles, $M^{(2)}_{RHP}$ is bounded in the region $\Omega_{21}$.
            As we have done in the proof of the proposition (\ref{ests}), let $s=u+iv$ with $u>\xi_2$ and $u>v$. Then, $M^{(3)}_1$ has the following estimation
            \begin{align}
                |M^{(3)}_1| \lesssim \frac{1}{\pi} ||M^{(3)}||_{L^\infty} ||M^{(2)}_{RHP}||_{L^\infty}^2 \iint_{\Omega_{21}} |\bar{\partial}R_{21} e^{-2it\theta}|\, \mathrm{d}A(s) \lesssim I_3+I_4,
            \end{align}
             where
             \begin{align}
                 &I_3= \iint_{\Omega_{21}} |r'(u)|e^{2t\theta''(\xi_2)(u-\xi_2)v}\, \mathrm{d}A(s), \nonumber\\
                 &I_4=\iint_{\Omega_{21}} |s-\xi_2|^{-1/2}e^{2t\theta''(\xi_2)(u-\xi_2)v}\, \mathrm{d}A(s).\nonumber
             \end{align}
       Using the inequality of Cauchy-Schwarz, we obtain
             \begin{align}
                |I_3| \le \int_{0}^{\infty} ||r'(u)||_{L^2(v+\xi_2,\infty)} \left( \int_{v+\xi_2}^\infty e^{4t\theta''(\xi_2)(u-\xi_2)v}\, \mathrm{d} u      \right)^{1/2}\, \mathrm{d}v  \lesssim t^{-3/4}.
            \end{align}
           For $2 < p<4$ and $1/p+1/q=1$,  using the formula (\ref{ests2}), we have
           \begin{align}
               |I_4| \le \int_{0}^\infty v^{1/p-1/2}  \left( \int_{v}^\infty e^{2t\theta''(\xi_2)quv}\, \mathrm{d}u \right)^{1/q} \, \mathrm{d}v  \lesssim t^{-3/4}.\nonumber
           \end{align}
           Therefore, we finish the proof.
        \end{proof}

         \section{Large-time asymptotics for   NLS equation  }
          Finally, from the above results in Sections 2-7,  we start to construct the large time asymptotics for the solution of  the Cauchy problem of the NLS equation  (\ref{q})-(\ref{inq})
           and  the obtained  main   result  is as below.

    \begin{theorem}\label{lta}
        Let $q(x,t)$ be the solution of the Cauchy problem  (\ref{q})-(\ref{inq}) with the initial data $q_0(x) \in \tanh(x) + H^{4,4}(\mathbb{R})$.
        %and the corresponding scattering data $\sigma_d=\left\{  r(z), ({z_n,c_n})_{n=1}^N          \right\}$.
       % Denote $q_{sol,N}(x,t)$  be the $N$-soliton solution.
        %associated with the scattering data
        %$$\sigma^{(out)}_d=\left\{\tilde{r}=0, (z_n,\tilde{c}_n)_{n=1}^N    \right\}$$ where
       % \begin{equation}
        %    \tilde{c}_n=c_n \exp\left(-2i \int_{I(\xi)} v(s) \left(\frac{1}{s-z}-\frac{1}{2s}\right) \, \mathrm{d}s \right).
       % \end{equation}
      Then, for  $|\xi|>1, \ \xi=\mathcal{O}(1)$,  there exists a   constant $T_0=T_0(q_0,\xi)$ such that for all $t>T_0$,
         \begin{equation}
              q(x,t)= e^{-i\alpha(\infty)} \left( 1 +t^{-1/2} h(x,t) \right)+\mathcal{O}\left(t^{-3/4}\right),
         \end{equation}
        where
        \begin{align}
        &\alpha(\infty)=\exp\left( \int_{I(\xi)} \frac{v(s)}{s}\, \mathrm{d}s \right),\label{alpinf}\\
          &  h(x,t)= \frac{\left(v(\xi_1)\right)^{1/2}}{2\sqrt{t\pi}\left(1-\xi_1^2\right)i} \left[ \frac{\xi_1^2 e^{-i\Phi_1}+ e^{i\Phi_1}}{\sqrt{|\theta''(\xi_1)|}} +\frac{e^{-i\Phi_2}+ \xi_1^2 e^{i\Phi_2}}{\sqrt{|\theta''(\xi_1^{-1})|}}    \right],           \label{hmout}
        \end{align}
       with
       \begin{align*}
        &v(z) =  -\frac{1}{2} \log(1-|r(z)|^2),\Phi_1 = \frac{\pi}{4} + \arg \Gamma(iv(\xi_1))-\arg \left(r_{\xi_1} \right),\Phi_2 = \Phi_1 + \alpha  - iv(\xi_1),\\
        &\alpha = \frac{\pi}{2} + 4t\theta(\xi_1) + v(\xi_1) \log \left(4t^2|\theta''(\xi_1)\theta''(\xi_1^{-1})| \right)+ 2 \arg \Gamma(iv(\xi_1)) + 2 \arg \frac{T_1(\xi_1)}{T_2(\xi_1^{-1})},
       \end{align*}
       and $T_k(z)(k=1,2)$ have been  given in (\ref{betai}).

    %   $v(\xi_1) =  -\frac{1}{2} \log(1-|r(\xi_1)|^2),\Phi_1 = \frac{\pi}{4} + \arg \Gamma(iv(\xi_1))-\arg \left(r_{\xi_1} \right),\Phi_2 = \Phi_1 + \alpha  - iv(\xi_1)$.
    %   \begin{align*}
    %       \Phi_1 = \frac{\pi}{4} + \arg \Gamma(iv(\xi_1))-\arg \left(r_{\xi_1} \right),\Phi_2 = \frac{3\pi}{4} - \arg \Gamma(iv(\xi_2))-\arg \left(r_{\xi_2} \right)
          % \Phi_1 = \frac{\pi}{4} + \arg \Gamma(iv(\xi_1))-\arg \left(r_{\xi_1} \right),\\
          % \Phi_2 = \frac{3\pi}{4} - \arg \Gamma(iv(\xi_2))-\arg \left(r_{\xi_2} \right)
    %   \end{align*}

    \end{theorem}

    \begin{proof}
       Inverting the sequence of transformations (\ref{mtom1}), (\ref{m1tom2}), (\ref{m2rhp}) and (\ref{m3m2m2rhp}), we have for $z\in \mathbb{C} \backslash \mathcal{U}_{\xi }  $
      \begin{equation}
          M(z)=T(\infty)^{\sigma_3}M^{(3)}(z)E(z)M^{out}(z)\big(R^{(2)}\big)^{-1}(z)T(z)^{-\sigma_3}.
      \end{equation}
       For convenience, we take $z\to\infty$ out of $\bar{\Omega}$ and obtain
      \begin{align*}
        \begin{split}
         M &=T(\infty)^{\sigma_3} \left( I+\frac{M_1^{(3)}}{z}+ \cdots\right)  \left(I +\frac{E_1}{z}+\cdots \right)\\
           &\left(I+\frac{M^{out}_1}{z}+\cdots \right)\left( I+\frac{T_1 \sigma_3}{z}+\cdots\right) T(\infty)^{-\sigma_3}\\
           &=I+z^{-1} T(\infty)^{\sigma_3} \left[ M_1^{(3)} +E_1 + M^{out}_1+T_1 \sigma_3    \right] T(\infty)^{-\sigma_3} +\mathcal{O}(z^{-2}),
        \end{split}
      \end{align*}
      where the sign with the subscript $1$ indicates the coefficient of $z^{-1}$ of this term.
      By the reconstruction formula (\ref{reconst}),  we find
      \begin{align}
          q(x,t)= T(\infty)^{-2} \left( 1  + t^{-1/2} h(x,t)   \right) + \mathcal{O}(t^{-3/4}).
      \end{align}
      Moreover, $T(\infty)^{-2}$ can be rewritten into the following form:
     \begin{equation}
         T(\infty)^{-2} =\exp \left( -2i\int_{ I(\xi)} \frac{v(s)}{2s}\, \mathrm{d}s \right)=e^{-i\alpha(\infty)}.
     \end{equation}
     Thus, we have completed  the proof.
    \end{proof}

  %  \begin{remark}
   %     In fact, following the step in Section 4-6,  we can obtain the similar result for $\xi<-1$.
      %  the theorem \ref{lta} still holds with
      %  \begin{align}
      %     & \vartheta= \begin{cases}-1, z\in \mathcal{U}_{\xi_1},\\1, z\in \mathcal{U}_{\xi_2}, \end{cases}\quad    \hat{\beta}= \begin{cases}
     %           \tilde{\beta},\quad i=1,\\
     %           \beta,\quad i=2,
     %       \end{cases} \\
      %      \tilde{c}_k=c_k \exp\left(-2i \int_{(0,\infty) \cup (\xi_2,\xi_1)} v(s) \left(\frac{1}{s-z}-\frac{1}{2s}\right) \, \mathrm{d}s \right),
     %   \end{align}
     %   and \begin{equation}\label{alphaneg}
      %  \alpha(\infty)=\exp \left( \int_{(0,\infty) \cup (\xi_2,\xi_1)} \frac{v(s)}{s}\, \mathrm{d}s \right).
     %   \end{equation}
 %   \end{remark}

\noindent\textbf{Acknowledgements}

This work is supported by  the National Natural Science
Foundation of China (Grant No. 11671095,  51879045, 11801323).

\end{document}